\documentclass[12pt]{amsart}

\usepackage{graphicx}
\usepackage{amsmath}
\usepackage{amscd}
\usepackage{amsfonts}
\usepackage{amssymb}
\usepackage{color}
\usepackage{fullpage}
\usepackage{stmaryrd}
\usepackage{hyperref}

\numberwithin{equation}{subsection}

\let\oldtocsection=\tocsection

\let\oldtocsubsection=\tocsubsection

\let\oldtocsubsubsection=\tocsubsubsection

\renewcommand{\tocsection}[2]{\hspace{0em}\oldtocsection{#1}{#2}}
\renewcommand{\tocsubsection}[2]{\hspace{1em}\oldtocsubsection{#1}{#2}}
\renewcommand{\tocsubsubsection}[2]{\hspace{2em}\oldtocsubsubsection{#1}{#2}}

\newtheorem{theorem}[equation]{Theorem}
\newtheorem{lemma}[equation]{Lemma}

\newtheorem{conjecture}[equation]{Conjecture}
\newtheorem{definition}[equation]{Definition}
\newtheorem{corollary}[equation]{Corollary}

\theoremstyle{definition}

\newtheorem{remark}[equation]{Remark}

\renewcommand{\qedsymbol}{{\vrule height5pt width5pt depth1pt}}

\newcommand{\be}{\begin{equation}}
\newcommand{\ee}{\end{equation}}
\newcommand{\bes}{\begin{equation*}}
\newcommand{\ees}{\end{equation*}}

\newcommand{\cD}{\mathcal{D}}

\newcommand{\cH}{\mathcal{H}}
\newcommand{\cK}{\mathcal{K}}
\newcommand{\cL}{\mathcal{L}}
\newcommand{\cM}{\mathcal{M}}

\newcommand{\cF}{\mathcal{F}}

\newcommand{\cA}{\mathcal{A}}
\newcommand{\cB}{\mathcal{B}}

\newcommand{\cO}{\mathcal{O}}
\newcommand{\cS}{\mathcal{S}}
\newcommand{\cT}{\mathcal{T}}

\newcommand{\cI}{\mathcal{I}}

\newcommand{\mb}[1]{\mathbb{#1}}

\newcommand{\rank}{\operatorname{rank}}
\newcommand{\Lat}{\operatorname{Lat}}
\newcommand{\Id}{\operatorname{Id}}

\newcommand{\alg}{\operatorname{alg}}
\newcommand{\Aut}{\operatorname{Aut}}

\newcommand{\dist}{\operatorname{dist}}

\newcommand{\Mult}{\operatorname{Mult}}

\newcommand{\spn}{\operatorname{span}}

\newcommand{\wot}{\textsc{wot}}
\newcommand{\Han}{\operatorname{Han}}
\newcommand{\mcc}{M\textsuperscript{c}Carthy}

\makeatletter
\let\@wraptoccontribs\wraptoccontribs
\makeatother

\begin{document}

\title[Drury-Arveson space]{Operator theory and function theory in Drury-Arveson space and its quotients}

\author{Michael Hartz}
\address{Fachrichtung Mathematik, Universit\"at des Saarlandes, 66123 Saarbr\"ucken, Germany}
\email{hartz@math.uni-sb.de}
\thanks{M.H. was partially supported by the Emmy Noether Program
of the German Research Foundation (DFG Grant 466012782).}

\author{Orr Moshe Shalit}
\thanks{O.S. is supported by ISF Grant no. 431/20.}
\address{Technion Israel Institute of Technology\\
Technion City, Haifa\; 3200003\\
Israel}
\email{oshalit@technion.ac.il}

\begin{abstract}
The Drury-Arveson space $H^2_d$ (also known as {\em symmetric Fock space} or {\em the $d$-shift space}), is the reproducing kernel Hilbert space on the unit ball of $\mathbb{C}^d$ with the kernel $k(z,w) = (1 - \langle z, w \rangle )^{-1}$. The operators $M_{z_i} : f(z) \mapsto z_i f(z)$, arising from multiplication by the coordinate functions $z_1, \ldots, z_d$, form a commuting $d$-tuple $M_z = (M_{z_1}, \ldots, M_{z_d})$. The $d$-tuple $M_z$ --- which is called the {\em $d$-shift} --- gives the Drury-Arveson space the structure of a Hilbert module. 

This Hilbert module is arguably the correct multivariable generalization of the Hardy space on the unit disc $H^2(\mathbb{D})$. It turns out that the Drury-Arveson space $H^2_d$ plays a universal role in operator theory (every pure, contractive Hilbert module is a quotient of an ampliation of $H^2_d$) as well as in function theory (every irreducible complete Pick space is essentially a restriction of $H^2_d$ to a subset of the ball). 
These universal properties resulted in the Drury-Arveson space being the subject of extensive studies, and the theory of the Drury-Arveson is today broad and deep. 

This survey aims to introduce the Drury-Arveson space, to give a panoramic view of the main operator theoretic and function theoretic aspects of this space, and to describe the universal role that it plays in multivariable operator theory and in Pick interpolation theory. 
\end{abstract}

\maketitle

\tableofcontents
   
\section{Introduction}

The {\em Drury-Arveson space} is a Hilbert function space which plays a universal role in operator theory as well as function theory. This space, denoted $H^2_d$ (or sometimes $\cF_+(E)$), and also known as the {\em $d$-shift space}, {\em Arveson's Hardy space} or the {\em symmetric Fock space}, has been the object of intensive study in the last fifteen years or so. Arguably, it is the subject of so much interest because it is the correct generalization of the classical Hardy space $H^2(\mb{D})$ from one variable to several. The goal of this survey is to collect together various important features of $H^2_d$, with detailed references and sometimes proofs, so as to serve as a convenient reference for researchers working with this space. 

Of course, a Hilbert space is a Hilbert space, and any two are isomorphic. Thus, when one sets out to study the Drury-Arveson space one is in fact interested in a certain concrete realization of Hilbert space which carries some additional structure. The additional structures are of two kinds: operator theoretic or function theoretic. For the operator theorist, the object of interest is the space $H^2_d$ together with a particular $d$-tuple $S = (S_1,\ldots, S_d)$ of commuting operators called the $d$-shift; in other words, the object of interest is a {\em Hilbert module} over the algebra $\mb{C}[z]$ of polynomials in $d$ variables. The function theorist would rather view $H^2_d$ as a Hilbert space comprised of functions on the unit ball $\mb{B}_d$ of $\mb{C}^d$, in which point evaluation is a bounded functional --- in other words: a {\em Hilbert function space}. 

There are many Hilbert modules and many Hilbert function spaces that one may study. Many of the results presented below have versions that work in other spaces. This survey focuses on the results in Drury-Arveson space for three reasons. First, as is explained below, $H^2_d$ is a universal object both as a Hilbert module and as Hilbert functions space, and results about $H^2_d$ have consequences in other spaces of interest. Second, $H^2_d$ is an interesting object of study in itself: being a natural analogue of $H^2(\mb{D})$ it enjoys several remarkable properties, and it could be useful to have an exposition which treats various facets of this space. Third, the study of $H^2_d$ is now quite developed, and can serve as a model for a theory in which multivariable operator theory and function theory are studied together. 

Most results are presented below without proof, but with detailed references. When a proof is presented it is usually because the result and/or the proof are of special importance. Sometimes a proof is also provided for a piece of folklore for which a convenient reference is lacking. 

\vskip 5pt
\noindent {\bf Prefatory note.} This survey was written by O.S. in 2014 for the 1st edition of the Springer reference work {\em Operator Theory}. 
In the decade that passed deep and interesting results on the Drury-Arveson space continued to be discovered, and the 2nd edition of {\em Operator Theory} is a good opportunity to collect and present them in concentrated form. 
O.S. invited M.H. to assist with the task of incorporating new results into this survey, and M.H. agreed write an appendix containing important developments that took place since the appearance of the survey. 
The first twelve sections of the survey remain largely unchanged, with only a few updates and corrections. 
New references have been added, most of which are cited in the appendix. The relatively recent survey papers \cite{FX19,Har22}  are cited at this point as they may serve the readers as useful alternative introductions to the subject. 

\vskip 5pt
\noindent {\bf Acknowledgements.}
The authors are grateful to the referees of the first and the second versions of this survey for their thoughtful and helpful feedback.

\section{Notation and terminology}\label{sec:notation}
\subsection{Basic notation}
Let $d$ be an integer or $\infty$ (the symbol $\infty$ will always stand for a countable infinity). $\mb{C}^d$ denotes $d$-dimensional complex Hilbert space.
$\mb{B}_d$ denotes the (open) unit ball in $\mb{C}^d$. The unit disc $\mb{B}_1$ is also denoted $\mb{D}$. 
It has become a convenient notational convention in the field to treat $d$ as a finite integer even when it is not. Some of the results are valid (or are known to be valid) only in the case of $d<\infty$, and these cases will be pointed out below. 

Let $H$ be a Hilbert space. The identity operator on $H$ is denoted by $I_H$ or $I$. If $M$ is a closed subspace of $H$ then $P_M$ always denotes the orthogonal projection from $H$ onto $M$. If $\cS$ is a subset of $H$, then $[\cS]$ denotes the closed subspace spanned by $\cS$. 
All operators below are assumed to be bounded operators on a separable Hilbert space. 

If $z_1, \ldots, z_d$ are $d$ commuting variables, then let $z = (z_1, \ldots, z_d)$ and write $z^\alpha$ for the product $z_1^{\alpha_1} \cdots z_d^{\alpha_d}$ for every multi-index $\alpha = (\alpha_1, \ldots, \alpha_d) \in \mb{N}^d$. The algebra of polynomials in $d$ commuting variables is denoted $\mb{C}[z_1, \ldots, z_d]$ or $\mb{C}[z]$ (this has an obvious interpretation also when $d=\infty$). The symbols $\alpha!$ and $|\alpha|$ are abbreviations for $\alpha_1! \cdots \alpha_d!$ and $|\alpha| = \alpha_1 + \ldots + \alpha_d$, respectively.

For the purposes of this survey, a function $f : \mb{B}_d \rightarrow \mb{C}$ is said to be {\em analytic} if it can be expressed as an absolutely convergent power series $f(z) = \sum_{\alpha\in \mb{N}^d} c_\alpha z^\alpha$ (when $d < \infty$ this is equivalent to the usual local definition). $\cO(\mb{B}_d)$ will denote the analytic functions on $\mb{B}_d$.

\subsection{Tuples of operators}\label{subsec:tuples} A $d$-tuple of operators on a Hilbert space $H$ is denoted $T = (T_1, \ldots, T_d)$. If $A$ and $B$ are $d$-tuples on two Hilbert spaces $H$ and $K$, and $U: H \rightarrow K$ is a unitary such that $UA_iU^* = B_i$ for all $i=1, \ldots, d$, then one says that $A$ and $B$ are {\em unitarily equivalent}, and one writes $UAU^* = B$. Similarly, one writes $T^*$ for the tuple $(T_1^*, \ldots, T_d^*)$, and so forth.

\subsection{Commuting and essentially normal tuples} Let $T$ be a $d$-tuple in $B(H)$. $T$ is said to be {\em commuting} if $[T_i,T_j] := T_i T_j - T_j T_i = 0$ for all $i,j$. If $T$ is a commuting contraction then for every $p \in \mb{C}[z]$ one may evaluate $p(T)$; for example, $T^\alpha = T_1^{\alpha_1} \cdots T_d^{\alpha_d}$.

A commuting tuple $T$ is said to be {\em normal} if $[T_i, T_j^*] = 0$ for all $i,j$, and {\em essentially normal} if $[T_i, T_j^*]$ is compact for all $i,j$.  Finally, a commuting tuple $T$ that satisfies $\textrm{trace} |[T_i, T_j^*]|^p < \infty$ is said to be {\em $p$-essentially normal}. 

$T$ is said to be {\em subnormal} if there is a Hilbert space $K \supseteq H$ and a normal $d$-tuple $N$ on $K$ such that $T = N\big|_H$.

\subsection{Row contractions and $d$-contractions} 
The tuple $T$ is said to be a {\em row contraction} if $\sum_{i=1}^d T_i T^*_i \leq I_H$ (when $d = \infty$ it is assumed that the partial sums are bounded by $I$, and hence that the sum converges in the strong operator topology to a positive operator less than the identity). Equivalently, this means that the row operator 
\[
\begin{bmatrix} T_1 &  T_2 & \cdots &  T_d \end{bmatrix} : \underbrace{H \oplus \cdots \oplus H}_{d \textrm{ times}} \rightarrow H
\]
is a contraction. With every row contraction one associates a completely positive map $\Theta_T : B(H) \rightarrow B(H)$ given by $\Theta_T(A) = \sum_{i=1}^d T_i A T_i^*$. Note that when $d = \infty$ the assumption $\sum T_i T_i^* \leq I$ ensures that $\sum_{i=1}^d T_i A T_i^*$ indeed converges in the strong operator topology. $T$ is said to be {\em pure} if $\Theta_T^n(I) \longrightarrow_{n \rightarrow \infty} 0$ in the strong operator topology. A commuting row contraction $T = (T_1, \ldots, T_d)$ is also called a {\em $d$-contraction}.  

\subsection{Defect operator and defect space}
\label{ss:defect}
The {\em defect operator} of a row contraction $T$ is the operator $\Delta_T = \sqrt{I - \Theta_T(I)}$, and the {\em defect space} is $\cD_T = \overline{\Delta_T H}$. The {\em rank} of $T$ is defined to be the dimension of the defect space, $\rank(T) = \dim(\cD_T)$. When no confusion may arise the notation $\Delta = \Delta_T$ is used.

\subsection{Hilbert modules} A popular and fruitful point of view for studying commuting operators on Hilbert space is that of {\em Hilbert modules} \cite{DougPaul} (see the chapters on Hilbert modules by Sarkar \cite{Sar25a,Sar25b} in this reference work). If $T$ is a commuting $d$-tuple on $H$, then $T$ induces on $H$ the structure of a Hilbert module via
\[
p\cdot h = p(T) h \,\, , \,\, p \in \mb{C}[z], h \in H.
\]
A Hilbert module is said to be {\em pure/contractive/of finite rank/essenitally normal/etc.}, if $T$ is pure/a row contraction/of finite rank/essenitally normal/etc., respectively. In \cite{DougPaul} Douglas and Paulsen put emphasis on Hilbert modules over {\em function} algebras, but Arveson \cite{Arv07} has found it useful to consider Hilbert modules over $\mb{C}[z]$. In general there is a big difference between these approaches, but by Section \ref{sec:model_theory} below every pure contractive Hilbert module over $\mb{C}[z]$ is in fact a Hilbert module over a certain natural, canonical algebra of functions.

\subsection{Hilbert function spaces} A {\em Hilbert function space} is a Hilbert space $H$ consisting of functions on some space $X$, such that for every $x \in X$ the point evaluation $f \mapsto f(x)$ is bounded linear functional on $H$ (such spaces are also commonly referred to as {\em reproducing kernel Hilbert spaces}). The reader is referred to \cite{AM02} as a reference for Hilbert function spaces.

\section{Drury-Arveson space as a function space}\label{sec:function}

The Drury-Arveson space is named after Drury, who basically introduced it into multivariable operator theory \cite{Drury}, and after Arveson, who has brought this space to the center of the stage \cite{Arv98}. 

\subsection{$H^2_d$ as a graded completion of the polynomials}\label{subsec:completion}
The most elementary definition of the Drury-Arveson space $H^2_d$ is as a graded completion of the polynomials \cite{Arv07}. Define an inner product on $\mb{C}[z]$ by setting
\be\label{eq:ort}
\langle z^\alpha, z^\beta \rangle = 0 \,\, , \,\, \textrm{ if } \alpha \neq \beta ,
\ee
and 
\be\label{eq:weights}
\langle z^\alpha, z^\alpha \rangle = \frac{\alpha!}{|\alpha|!}. 
\ee
The condition (\ref{eq:ort}) may seem natural, but the choice of weights (\ref{eq:weights}) might appear arbitrary at this point; see Section \ref{subsec:identification}. The completion of $\mb{C}[z]$ with respect to this inner product is denoted by $H^2_d$. It is clear that $H^2_d$ can be identified with the space of holomorphic functions $f: \mb{B}_d \rightarrow \mb{C}$ which have a power series $f(z) = \sum_\alpha c_\alpha z^\alpha$ such that 
\[
\|f\|^2_{H^2_d} \equiv \|f\|^2 := \sum_{\alpha} |c_\alpha|^2 \frac{\alpha!}{|\alpha|!} < \infty. 
\]

\subsection{$H^2_d$ as a Hilbert function space}\label{subsec:HFS} The space $H^2_d$ turns out to be the Hilbert function space on $\mb{B}_d$ determined by the kernel 
\be\label{eq:kernel}
k(z,w) = k_w(z) = \frac{1}{1-\langle z, w \rangle}. 
\ee
Indeed, for $|w|<1$, $k_w(z) = \sum_{n=0}^\infty \langle z, w \rangle^n  = \sum_{n=0}^\infty \sum_{|\alpha| = n} \frac{|\alpha|!}{\alpha!} \overline{w}^\alpha z^\alpha $  is clearly in $H^2_d$, and 
\be
f(w) = \sum_\alpha c_\alpha w^\alpha = \sum_\alpha c_\alpha \frac{|\alpha|!}{\alpha!} w^\alpha \langle z^\alpha, z^\alpha \rangle = \langle f, k_w \rangle . 
\ee
This shows that point evaluation is a bounded functional on $H^2_d$, so $H^2_d$ is a Hilbert function space \cite{AM02}, and it has the kernel (\ref{eq:kernel}). Since the only function that is orthogonal to all the kernel function $k_w$ is the zero function, $\spn \{k_w : w \in \mb{B}_d\}$ is dense in $H^2_d$. When $d = 1$, $H^2_d = H^2(\mb{D})$ is the usual Hardy space on the unit disc (see \cite{Garnett} for a thorough treatment of $H^2(\mb{D})$).

\subsection{The multiplier algebra of $H^2_d$}

As every Hilbert function space, $H^2_d$ comes along with its {\em multiplier algebra}
\[
\cM_d := \Mult(H^2_d) = \{f : \mb{B}_d \rightarrow \mb{C} \big| fh \in H^2_d \textrm{ for all } h \in H^2_d \} .
\]
To every multiplier $f \in \cM_d$ there is associated a {\em multiplication operator} $M_f : h \mapsto fh$. Standard arguments (see \cite{AM02}) show that $M_f$ is bounded and that 
\be\label{eq:ineq_norm}
\|f\|_\infty \leq \|M_f\| .
\ee
In the case $d = 1$ it is known that $\|f\|_\infty = \|M_f\|$ and that $\cM_1 = H^\infty(\mb{D})$ \cite{AM02}. 
The {\em multiplier norm} of $f \in \cM_d$ is given by 
\be\label{eq:mult_norm}
\|f\|_{\cM_d} = \|M_f\|, 
\ee
and this norm gives $\cM_d$ the structure of an operator algebra. On the other hand, $\cM_d$ is also an algebra of analytic functions contractively contained in $H^\infty(\mb{B}_d)$. It will be shown below that if $d>1$ then $\cM_d$ is strictly contained in $H^\infty$, the supremum norm is not comparable with the multiplier norm, and hence that $\cM_d$ is not a {\em function algebra}. 

A trivial but useful observation is that since $1 \in H^2_d$, one immediately obtains $\cM_d \subset H^2_d$, as spaces of functions.

\subsection{The $d$-shift}\label{subsec:dshift}

The most natural $d$-tuple of operators occurring in the setting of $H^2_d$ is the $d$-shift, given by $M_z = (M_{z_1}, \ldots, M_{z_d})$, where $z_1, \ldots, z_d$ are the coordinate functions in $\mb{C}^d$; thus
\be\label{eq:shift}
(M_{z_i}h)(z) = z_i h(z) \,\, , \,\, i=1, \ldots, d, \,\,\, \,h \in H^2_d .
\ee
It is straightforward that multiplication by every coordinate function is a bounded operator, hence the coordinate functions are all in $\cM_d$. In fact, by \ref{subsec:thedshift} and \ref{subsec:identification} below, $M_z$ is a pure row contraction. Consequently, $\mb{C}[z] \subseteq \cM_d$. 
When $d=1$ then the $d$-shift is nothing but the unilateral shift on $H^2(\mb{D})$.

\subsection{Homogeneous decomposition of functions}\label{subsec:homogeneous_dec}

Every $f \in \cO(\mb{B}_d)$ has a Taylor series $f(z) = \sum_{\alpha} a_\alpha z^\alpha$ convergent in $\mb{B}_d$, so in particular $f$ has a decomposition into its homogeneous parts:
\be\label{eq:homogeneous}
f(z) = \sum_{n=0}^\infty f_n(z) ,
\ee
where $f_n(z) = \sum_{|\alpha|=n} a_\alpha z^\alpha$, and the series (\ref{eq:homogeneous}) converges uniformly on compact subsets of the ball. When $f$ happens to be in one of the function spaces studied here then more can be said. 

For $f \in H^2_d$ the homogeneous components $f_n$ are all orthogonal one to another, the series (\ref{eq:homogeneous}) converges in norm and $\|f\|^2 = \sum \|f_n\|^2$. If $f \in \cM_d$ this is still true since $\cM_d \subset H^2_d$, but understanding (\ref{eq:homogeneous}) in terms of the structure of $\cM_d$ is a more delicate task. The series does not necessarily converge in norm (as can be seen by considering the case $d=1$). 

Recall that (\ref{eq:mult_norm}) allows one to consider $\cM_d$ as an algebra of operators on $H^2_d$. For $t \in \mb{R}$, let $U_t$ be the unitary on $H^2_d$ sending $h(z)$ to $h(e^{it}z)$, and let $\gamma_t$ be the automorphism on $B(H^2_d)$ implemented by $U_t$ . A computation shows that $\cM_d$ is stable under $\gamma$ and that $\gamma_t(f)(z) = f(e^{it}z)$ for $f \in \cM_d$. 

\begin{lemma}
For all $n=0,1,\ldots,$, the integral 
\[
\frac{1}{2\pi} \int_0^{2\pi} \gamma_t(f) e^{-int} dt
\]
converges in the strong operator topology to $f_n$. In particular, $\|f_n\|_{\cM_d} \leq \|f\|_{\cM_d}$. 
\end{lemma}

For $r \in (0,1)$, the function $f_r(z) := f(rz)$ has homogeneous decomposition
\be\label{eq:f_r}
f_r(z) = \sum_{n=0}^\infty r^n f_n(z),  
\ee
and this series converges absolutely in the multiplier norm, by the lemma. Rewrite
\be\label{eq:f_rPoisson}
f_r = \frac{1}{2\pi} \int_0^{2\pi} \gamma_t(f) P_r(t) dt,
\ee
where $P_r(t)$ denotes the Poisson kernel on the disc. By well known techniques of harmonic analysis, one has the following theorem. 
\begin{theorem}\label{thm:Poisson}
Let $f \in \cM_d$, and for all $r \in (0,1)$ denote $f_r(z) = f(rz)$. Then $f_r \in \cM_d$, $\|f_r\|_{\cM_d} \leq \|f\|_{\cM_d}$, and the series (\ref{eq:homogeneous}) is Poisson summable to $f$ in the strong operator topology: $\lim_{r\rightarrow 1} f_r = f$. 
\end{theorem}

\subsection{The structure of $\cM_d$}\label{subsec:structure}

Denote by $\overline{\alg}^\wot (M_z)$ the unital weak-operator topology (\wot) closed operator algebra generated by the $d$-shift. The $d$-shift generates $\cM_d$ in the sense of the following theorem.

\begin{theorem}\label{thm:wot_gen}
The unital \wot -closed algebra generated by $M_z$ is equal to $\{M_f : f \in \cM_d\}$. 
\end{theorem}
The following lemma is required for the proof of the theorem. 
\begin{lemma}
Let $\{f_\alpha\}$ be a bounded net in $\cM_d$ that is bounded in the multiplier norm. If $f \in \cM_d$, then $M_{f_\alpha}$ converges to $M_f$ in the weak-operator topology if and only if $f_\alpha(z) \rightarrow f(z)$ for all $z \in \mb{B}_d$. 
\end{lemma}
\begin{proof}
See, e.g., \cite[Lemma 11.10]{DRS11}. 
\end{proof}
\noindent {\bf Proof of Theorem \ref{thm:wot_gen}}. Every multiplier algebra is \wot -closed, so $\overline{\alg}^\wot (M_z)$ is contained in $\{M_f : f \in \cM_d\}$. Let $f \in \cM_d$. For $r \in (0,1)$, define $f_r(z) = f(rz)$. Then by (\ref{eq:f_r}) $M_{f_r}$ is in the norm closed algebra generated by $M_z$. By \ref{thm:Poisson} the net $\{f_r\}_{r \in (0,1)}$ is bounded by $\|f\|$. Since $f _r \rightarrow f$ pointwise, the lemma implies the \wot -convergence $M_{f_r} \rightarrow M_f$. \qedsymbol

The above theorem allows one to make the identification 
\be
\cM_d = \overline{\alg}^\wot(M_z) . 
\ee

\subsection{The strict containment $\cM_d \subsetneq H^\infty(\mb{B}_d)$}

When $d = 1$, $H^2_d = H^2(\mb{D})$ is the usual Hardy space, its multiplier algebra is equal to $H^\infty$, and the multiplier norm of a multiplier $f$ is equal to $\|f\|_\infty = \sup_{z \in \mb{D}}|f(z)|$. When $d>1$ this is no longer true. 

\begin{theorem}\label{thm:MnotHinfty}
For $d>1$ the norms $\|\cdot \|_\infty$ and $\| \cdot \|_{\cM_d}$ are not comparable on $\cM_d$, there is a strict containment 
\be\label{eq:strict}
\cM_d \subsetneq H^\infty(\mb{B}_d),
\ee
and the $d$-tuple $M_z$ is not subnormal. 
\end{theorem}
\begin{proof}
If $f \in \cM_d$ and $\lambda \in \mb{B}_d$, then for all $h \in H^2_d$ 
\[
\langle h, M_f^* k_\lambda \rangle = f(\lambda) h(\lambda) = \langle h, \overline{f(\lambda)} k_\lambda \rangle .
\]
Thus $\overline{f(\lambda)}$ is an eigenvalue of $M_f^*$ and in particular $|\overline{f(\lambda)}|\leq \|M_f\|$. It follows that $f$ is bounded on $\mb{B}_d$ and that $\sup_{\mb{B}_d}|f| \leq \|M_f\|$ (this argument works for any multiplier algebra). 
Since $1 \in H^2_d$ it follows that $f = f \cdot 1$ is analytic, thus $\cM_d \subseteq H^\infty(\mb{B}_d)$. 

For the strictness of the containment it suffices to consider the case $d<\infty$. Direct computations show that for a suitable choice of constants $a_1, a_2, \ldots$, the functions
\[
f_N(z) := \sum_{n=0}^N a_n (z_1 \cdots z_d)^n
\] 
satisfy $\|f_N\|_\infty \leq 1$ while $\|f_N\|_{\cM_d} \rightarrow \infty$. Moreover, the limit $f := \lim_{N\rightarrow \infty} f_N$ exists uniformly, and serves as an explicit example of a function that is in the ``ball algebra" $A(\mb{B}_d)$ (that is, the algebra of continuous functions on the closed ball which are analytic on the interior), but is not in $\cM_d$.  

That $M_z$ is not subnormal follows from the incomparability of the norms; see Section 3 in \cite{Arv98} for full details (see also Section 2 in \cite{DavPittsPick} for a slightly different derivation of the first parts of the theorem). 
\end{proof}

\subsection{Vector valued $H^2_d$ and operator valued multipliers}

Let $K$ be a Hilbert space. The Hilbert space tensor product $H^2_d \otimes K$ can be considered as the space of all holomorphic functions $f : \mb{B}_d \rightarrow K$ with Taylor series $f(z) = \sum_\alpha a_\alpha z^\alpha$, where the coefficients $a_\alpha$ are in $K$ and 
\[
\sum_\alpha \frac{\alpha!}{|\alpha|!} \|a_\alpha\|^2 < \infty .
\]
Let $K_1$ and $K_2$ be two Hilbert spaces, and let $\Phi : \mb{B}_d \rightarrow B(K_1, K_2)$ be an operator valued function. For $h \in H^2_d \otimes K_1$, define $M_\Phi h$ to be the function $\mb{B}_d \rightarrow K_2$ given by 
\[
M_\Phi h (z) = \Phi(z) h(z) \,\, , \,\, z \in \mb{B}_d. 
\]
Denote by $\cM_d(K_1, K_2)$ the space of all $\Phi$ for which $M_\Phi h \in H^2_d \otimes K_2$ for all $h \in H^2_d \otimes K_1$ (another common notation is $\Mult(H^2_d \otimes K_1, H^2_d \otimes K_2)$). An element $\Phi$ of $\cM_d(K_1, K_2)$ is said to be a {\em multiplier}, and in this case $M_\Phi$ (which can be shown to be bounded) is called a {\em multiplication operator}. If $K_1 = K_2 = K$ then $\cM_d(K_1, K_2)$ is abbreviated to $\cM_d(K)$. The space $\cM_d(K_1, K_2)$ is endowed with the norm $\|\Phi\| = \|M_\Phi\|$. 

The following characterization of multipliers, which is useful also in the scalar case, holds in any Hilbert function space (the proof is straightforward, see \cite[Theorem 2.41]{AM02}). 

\begin{theorem}\label{thm:what_multpliers_do}
Let $\Phi : \mb{B}_d \rightarrow B(K_1, K_2)$ be an operator valued function. If $\Phi$ is a multiplier then
\be\label{eq:what_mutlpliers_do}
M_\Phi^* (k_\lambda \otimes v) = k_\lambda \otimes \Phi(\lambda)^* v
\ee
for all $\lambda \in \mb{B}_d$ and $v \in K_2$. Conversely, if $\Phi : \mb{B}_d \rightarrow B(K_1, K_2)$ and the mapping $k_\lambda \otimes v \mapsto k_\lambda \otimes \Phi(\lambda)^* v$ extends to a bounded operator $T \in B(H^2_d \otimes K_2, H^2_d \otimes K_1)$, then $\Phi \in \cM_d(K_1, K_2)$ and $T = M^*_\Phi$. 
\end{theorem}

It is immediate from (\ref{eq:what_mutlpliers_do}) that any multiplier $\Phi$ is bounded (in the sense that there is $M>0$ such that $\|\Phi(z)\|\leq M$ for all $z \in \mb{B}_d$) and holomorphic (in the sense that for all $u \in K_1, v \in K_2$ the function $z \mapsto \langle \Phi(z)u,v \rangle$ is holomorphic in the ball).

The following theorem, due to Ball, Trent and Vinnikov, provides a characterization of multipliers in $\cM_d(K_1, K_2)$, which is specific to the setting of $H^2_d$. For a proof and additional characterizations, see \cite[Section 2]{BTV01} (see also \cite{AT02,EP02}). 

\begin{theorem}[\cite{BTV01}, Theorem 2.1; \cite{EP02}, Theorem 1.3]
  \label{thm:transfer}
Let $\Phi : \mb{B}_d \rightarrow B(H^2_d \otimes K_1, H^2_d \otimes K_2)$. Then the following statements are equivalent:
\begin{enumerate}
\item $\Phi \in \cM_d(K_1, K_2)$ with $\|\Phi\| \leq 1$. 
\item The kernel 
\be\label{eq:KPhi1}
K_\Phi(z,w) = \frac{I - \Phi(z) \Phi(w)^*}{1 - \langle z, w \rangle}
\ee
is a positive sesqui-analytic $B(K_2)$ valued kernel on $\mb{B}_d \times \mb{B}_d$; i.e., there is an auxiliary Hilbert space $H$ and a holomorphic $B(H, K_2)$-valued function $\Psi$ on $\mb{B}_d$ such that for all $z,d \in \mb{B}_d$, 
\be\label{eq:KPhi2}
K_\Phi(z,w) = \Psi(z) \Psi(w)^* .
\ee
\item There exists an auxiliary Hilbert space $H$ and a unitary operator 
\be\label{eq:lurking}
U = \begin{pmatrix}
A & B \\
C & D
\end{pmatrix} : \begin{pmatrix}
H \\ K_1
\end{pmatrix} \rightarrow \begin{pmatrix}
\oplus_{1}^d H \\ K_2
\end{pmatrix}
\ee
such that 
\be\label{eq:realization}
\Phi(z) = D + C(I- Z(z)A)^{-1}Z(z)B ,
\ee
where $Z(z) =  \begin{bmatrix} z_1I_H & \cdots &  z_d I_H \end{bmatrix} : \oplus_1^d H \rightarrow H$. 
\end{enumerate}
\end{theorem}

The formula (\ref{eq:realization}) is referred to as {\em the realization formula}. Sometimes, $U$ is said to be a {\em unitary colligation}, and $\Phi$ is called the associated {\em transfer function}. The papers \cite{BBF07a,BBF07b,BBF08} of Ball, Bolotnikov and Fang provide more details on the connections of the transfer function with systems theory in the context of Drury-Arveson space.

\subsection{The commutant of $\cM_d$}\label{subsec:Mtag}
The {\em commutant} of an operator algebra $\cB \subseteq B(H)$ is defined to be
\[
\cB' = \{a \in B(H) :  ab = ba \,\textrm{  for all } b \in B\}.
\]

A standard argument shows that $\cM_d$ is its own commutant:
\be
\cM_d' = \cM_d.
\ee
More generally, one has the following, which is a special case of the commutant lifting theorem (Theorem \ref{thm:BTV} below). 

\begin{theorem}
Let $K_1, K_2$ be Hilbert space, and let $X \in B(K_1, K_2)$ such that 
\be
X (M_f \otimes I_{K_1}) = (M_f \otimes I_{K_2}) X ,
\ee
for all $f \in \cM_d$. Then there exists $\Phi \in \cM_d(K_1, K_2)$ such that $X = M_\Phi$. 
\end{theorem}

\subsection{$H^2_d$ as a Besov-Sobolev space}\label{sec:BS}

The Drury-Arveson space also fits into a family of function spaces which have been of interest in harmonic analysis (see, e.g., \cite{ARS08,CSW12,VW12}). In this subsection it is assumed that $d<\infty$. 

For an analytic function $f \in \cO(\mb{B}_d)$, {\em the radial derivative of $f$} is defined to be $R f = \sum_{i=1}^d z_i \frac{\partial f}{\partial z_i}$. 
It is useful to note that if $f$ is a homogeneous polynomial of degree $n$, then $Rf = nf$. 

Let $\sigma \geq 0$, $p \in [1,\infty)$, and let $m$ be an integer strictly greater than $d/p-\sigma$. For every $f \in \cO(\mb{B}_d)$, one can consider the norm $\|f\|_{m,\sigma,p}$ defined by 
\[
\|f\|^p_{m,\sigma,p} = \sum_{|\alpha|<m} \left|\frac{\partial^\alpha f}{\partial z^\alpha}(0)\right|^p + \int_{\mb{B}_d} |R^m f(z)|^p (1 - |z|^2)^{p(m+\sigma) - d - 1} d\lambda(z) ,
\]
where $\lambda$ is Lebesgue measure on the ball. It turns out that choosing different $m > d/p-\sigma$ results in equivalent norms. One defines {\em the analytic Besov-Sobolev spaces $B^\sigma_p(\mb{B}_d)$} as
\[
B^\sigma_p(\mb{B}_d) = \{f \in \cO(\mb{B}_d) : \|f\|_{m,\sigma,p} < \infty\}.
\]
When $p = 2$ one obtains a family of Hilbert function spaces, which --- up to a modification to an equivalent norm --- have reproducing kernel (for $\sigma > 0$) 
\[
k^\sigma (z,w) = \frac{1}{(1-\langle z, w \rangle)^{2 \sigma}} .
\]
The proof of this is straightforward, using basic integral formulas on the ball (available in \cite[Section 1.4]{RudinBall} or \cite[Section 1.3]{ZhuBook}) and the fact that the reproducing kernel in a Hilbert function space is given by $\sum e_k(z)\overline{e_k(w)}$, where $\{e_k\}_{k=1}^\infty$ is any orthonormal basis. 
In particular this scale of spaces contains the Bergman space $L^2_a(\mb{B}_d)$ ($\sigma  = (d+1)/2$) and the Hardy space $H^2(\mb{B}_d)$ ($\sigma = d/2)$. For $p=2$ and $\sigma = 1/2$ one gets the Drury-Arveson space.

\begin{theorem}\label{thm:equiv_def}
Fix an integer $m>(d-1)/2$. For $f \in \cO(\mb{B}_d)$ the following are equivalent:
\begin{enumerate}
\item $f \in H^2_d$. 
\item $R^{(d-1)/2} f \in H^2(\mb{B}_d)$ (the Hardy space of the ball). 
\item $\|f\|_{m,1/2,2}< \infty$. 
\item $|||f||| < \infty$, where 
\[|||f|||^2 = \sum_{|\alpha|<m} \left|\frac{\partial^\alpha f}{\partial z^\alpha}(0)\right|^2 + \sum_{|\alpha|=m}\int_{\mb{B}_d} \left|\frac{\partial^\alpha f}{\partial z^\alpha} (z)\right|^2 (1 - |z|^2)^{2m - d} d\lambda(z) .
\]
\end{enumerate}
Moreover, the norms $\|\cdot\|_{m,1/2,2}$, $|||\cdot|||$ and $\|\cdot\|_{H^2_d}$ are equivalent. 
\end{theorem}

Theorem \ref{thm:equiv_def} appears as Theorem 1 in \cite{Chen03} (one should beware that the same paper included another characterization of Drury-Arveson space \cite[Theorem 2]{Chen03}, but unfortunately that other result (which will not be stated here) is incorrect --- see \cite{FX13}). In \cite{Chen03} the result was stated only for the smallest integer $m$ satisfying $m > (d-1)/2$, but the proof of the theorem --- which boils down to calculations of the various integrals defining the norms (using formulas from \cite[Section 1.4]{RudinBall} or \cite[Section 1.3]{ZhuBook}) --- works for all $m > (d-1)/2$.

\section{Drury-Arveson space as symmetric Fock space}\label{sec:sym}

A crucial fact is that the Hilbert function space $H^2_d$ can be identified with the familiar symmetric Fock space. This identification (essentially contained in \cite{Drury}, but most clearly explained in \cite{Arv98}) accounts for the universal properties of $H^2_d$, and among other things also explains the significance of the choice of weights (\ref{eq:weights}). 

\subsection{Full Fock space} 
Let $E$ be a $d$-dimensional Hilbert space. The {\em full Fock space} is the space 
\[
\cF(E) = \mb{C} \oplus E \oplus E^{\otimes 2} \oplus E^{\otimes 3}  \oplus \ldots 
\]

\subsection{The noncommutative $d$-shift} 
Fix a basis $\{e_1, \ldots, e_d\}$ of $E$. On $\cF(E)$ define $L = (L_1, \ldots, L_d)$ by 
\[
L_i x_1 \otimes \cdots \otimes x_n = e_i \otimes x_1 \otimes \cdots \otimes x_n .
\] 
$L$ is called the {\em noncommutative $d$-shift}. The tuple $L$ is easily seen to be a {\em row isometry}, meaning that the row operator  $\begin{bmatrix} L_1 & L_2 & \cdots & L_d \end{bmatrix}$  from the direct sum of $\cF(E)$ with itself $d$ times into $\cF(E)$ is an isometry; equivalently, this means that $L_1, \ldots, L_d$ are isometries with pairwise orthogonal ranges. The tuple $L$ plays a central role in noncommutative multivariable operator theory, see, e.g., \cite{AriasPopescu00, DavPittsPick, DavPitts2, DavPitts1, Popescu89, Popescu91, Popescu06}. The noncommutative $d$-shift is a universal row contraction, see Section \ref{subsec:noncommutative} below. 

The construction does not depend on the choice of the space $E$ or the orthonormal basis, and henceforth $\cF(E)$ will be sometimes denoted $\cF_d$, understanding that some choice has been made.

\subsection{The noncommutative analytic Toeplitz algebra $\cL_d$}\label{sec:ncat}

The noncommutative analytic Toeplitz algebra $\cL_d$ is defined to be $\overline{\alg}^\wot (L)$. This algebra was introduced by Popescu in \cite{Popescu91}, where it was shown that it is the same as the noncommutative multiplier algebra of the full Fock space. $\cL_d$ is also referred to as the {\em left regular representation free semigroup algebra}, and plays a fundamental role in the theory of {\em free semigroup algebras} (see the survey \cite{DavidsonSurvey}).

Since $\cL_d$ is \wot-closed, it is also weak-$*$ closed as a subspace of $B(\cF_d)$, the latter considered as the dual space of the trace class operators on $\cF_d$. Thus it is a {\em dual algebra}, that is, an operator algebra that is also the dual space of a Banach space. One then has a weak-$*$ topology on $\cL_d$, and weak-$*$ continuous functionals come into play. The following factorization property for weak-$*$ functionals has proved very useful \cite{BFP85}. 

\begin{definition}
Let $\cB \subseteq B(H)$ be a dual algebra, and denote by $\cB_*$ its predual. $\cB$ is said to have {\em property $\mb{A}_1$} if for every $\rho \in \cB_*$ there exist $g,h \in H$ such that 
\[
\rho(b) = \langle bg,h \rangle \,\, , \,\, b \in \cB.
\]
If, for every $\epsilon > 0$, $g$ and $h$ can be chosen to satisfy $\|g\| \|h\|<(1+ \epsilon) \|\rho\|$, then $\cB$ is said to have {\em property $\mb{A}_1(1)$}. 
\end{definition}

\begin{theorem}[\cite{DavPitts1}, Theorem 2.10]\label{thm:LnA1}
$\cL_d$ has property $\mb{A}_1(1)$. 
\end{theorem}

\begin{corollary}
The weak-$*$ and \wot-topologies on $\cL_d$ coincide. 
\end{corollary}

\subsection{Quotients of $\cL_d$}\label{subsec:invNidealsNC}

The following theorem is a collection of results from \cite[Section 4]{AriasPopescu00} and \cite[Section 2]{DavPittsPick}. 

\begin{theorem}
Fix a \wot-closed two sided ideal $J$ and denote $N = [J\cF_d]^\perp$. Put $B = P_N L P_N$. Then the map $\pi : A \mapsto P_NAP_N$ is a homomorphism from the algebra $\cL_d$ onto $P_N \cL_d P_N$ which annihilates $J$. Moreover:
\begin{enumerate}
\item $P_N\cL_d P_N = \overline{\alg}^\wot(B)$ --- the unital \wot-closed algebra generated by $B$. 
\item $P_N\cL_d P_N$ has property  $\mb{A}_1(1)$. 
\item $P_N\cL_d P_N = (P_N\cL_d P_N)''$. 
\item $\pi$ promotes to a natural completely isometric isomorphism and weak-$*$ hom\-eomorphism $\cL_d / J$ onto $P_N \cL_d P_N$. 
\end{enumerate}
\end{theorem}

%
%

\subsection{Symmetric Fock space} 
For every permutation $\sigma$ on $n$ elements, one defines a unitary operator $U_\sigma$ on $E^{\otimes n}$ by 
\[
U_\sigma (x_1 \otimes \cdots \otimes x_n) = x_{\sigma(1)} \otimes \cdots \otimes x_{\sigma(n)}.
\]
The {\em $n$th-fold symmetric tensor product of $E$}, denoted $E^n$, is defined to be the subspace of $E^{\otimes n}$ which consists of the vectors fixed under the unitaries $U_\sigma$ for all $\sigma$. The {\em symmetric Fock space} is the subspace of $\cF(E)$ given by 
\[
\cF_+(E) = \mb{C} \oplus E \oplus E^2 \oplus E^3 \oplus \ldots .
\]
If $x_1 \in E^{n_1}, \ldots x_k \in E^{n_k}$, write $x_1 x_2 \cdots x_k$ for the projection of $x_1 \otimes x_2 \otimes \cdots \otimes x_k$ into $E^{n_1 + \ldots +n_k}$. Letting $\{e_1, \ldots, e_d\}$ be an orthonormal basis for $E$, $e^\alpha$ is shorthand for $e_1^{\alpha_1} \cdots e_d^{\alpha_d}$ for all $\alpha\in \mb{N}^d$. A computation shows that $\{e^\alpha\}_{|\alpha| = n}$ is an orthogonal basis for $E^n$ and that
\be\label{eq:norm_e}
\|e^\alpha\|^2 = \frac{\alpha!}{|\alpha|!} .
\ee

The space $\cF_+(E)$ is not invariant under the noncommutative $d$-shift $L$, but it is {\em co-invariant}, meaning that $L^*_i \cF_+(E) \subseteq \cF_+(E)$ for all $i$.

\subsection{The $d$-shift}\label{subsec:thedshift}
The {\em (commutative) $d$-shift} is the $d$ tuple $S = (S_1, \ldots, S_d)$ of operators given by compressing the noncommutative $d$-shift to $\cF(E)$. Thus, for all $n$ and all $x \in E^n$
\be\label{eq:defS}
S_i x = e_i x \quad, \quad i=1, \ldots, d .
\ee
It is straightforward to check that the $d$-shift has the following properties:
\begin{enumerate}
\item $S$ is commuting, i.e., $S_i S_j - S_j S_i = 0$. 
\item $\sum_{i=1}^d S_i S_i^* = I - P_{\mb{C}}$, and in particular $S$ is a row contraction. 
\item $S$ is pure. 
\end{enumerate}
 
Many results on the $d$-shift can be obtained by ``compressing theorems" about the noncommutative $d$-shift down to $\cF_+(E)$; see, e.g.,  \cite{DL10,DavPittsPick,DRS15,Popescu06}, the proof of Theorem \ref{thm:comp_Pick} or Sections \ref{subsec:identification2} and \ref{subsec:invNideals} below as well as the appendix.  This is a powerful technique, due to the availability of strong results for the noncommutative $d$-shift, e.g., \cite{DavPitts2,DavPitts1,Popescu89,Popescu91} or more generally \cite{MS98}. Another advantage of this technique is that it allows to obtain similar results for a very large class of Hilbert modules by compressing the noncommutative $d$-shift to other co-invariant spaces; see \cite{Popescu06,ShalitSolel}.

\subsection{Essential normality of the $d$-shift}\label{subsec:ess_norm_d_shift}

Let $N$ be the densely defined unbounded operator $N$ on $\cF_+(E)$ defined by $N h = n h$ for $h \in E^n$. $N$ is usually referred to in this context as the {\em number operator} (it is equal to the restriction of the radial derivative $R$ from \ref{sec:BS}). A straightforward computation (see \cite[Proposition 5.3]{Arv98}) shows that 
\be
[S^*_i,  S_j] = S^*_i S_j - S_j S^*_i = (1+N)^{-1}(\delta_{ij}1 - S_j S_i^*).
\ee
It follows readily that if $d<\infty$ then $S$ is $p$-essentially normal for all $p>d$ (but not for $p=d$). In particular $[S_i, S^*_j]$ is compact when $d<\infty$. It is not compact when $d= \infty$.

\subsection{Identification of $H^2_d$ with symmetric Fock space}\label{subsec:identification}

Fix $d \in \{1,2,\ldots, \infty\}$ and let $E$ be a $d$-dimensional Hilbert space with orthonormal basis $\{e_n\}_n$. Define $V : \mb{C}[z_1, \ldots, z_d] \rightarrow \cF_+(E)$ by 
\[
V\left(\sum_\alpha c_\alpha z^\alpha\right) = \sum_\alpha c_\alpha e^\alpha .
\]
By equations (\ref{eq:weights}) and (\ref{eq:norm_e}) $V$ extends to a unitary from $H^2_d$ onto $\cF_+(E)$. All separable infinite dimensional Hilbert spaces are isomorphic; the important feature here is that 
\[
V M_z V^* = S .
\]
Alternatively, there is also an anti-unitary identification of these two spaces. Every $f \in H^2_d$ can be written in a unique way as 
\[
f(z) = \sum \langle z^n , \xi_n \rangle, 
\]
where $z^n$ denotes the $n$th symmetric product of $z \in \mb{C}^d$ with itself, $\xi_n \in (\mb{C}^d)^n$, and $\sum_n \|\xi_n\|^2 < \infty$ (see \cite[Section 1]{Arv98}). Then the map $J : H^2_d \rightarrow \cF_+(E)$ given by $J f = \sum_n \xi_n$ is an anti-unitary and $J M_z J^{-1} = S$. 

Because of the above identification, the notation $S$ is also used for the tuple $M_z$ acting on $H^2_d$. It is safe to switch from $\cF_+(E)$ to $H^2_d$ and back, as convenient. Together with this identification, the results of Section \ref{subsec:structure} allow one to identify between $\cM_d$ and the unital \wot-closed algebra generated by $S$.

\subsection{Identification of $\cM_d$ with the compression of $\cL_d$}\label{subsec:identification2}

The {\em antisymmetric Fock space (over $E$)} is defined to be $\cF_-(E) = \cF(E) \ominus \cF_+(E)$. By \ref{subsec:invNidealsNC} and \ref{subsec:identification} $\cM_d$ can be identified with the compression of $\cL_d$ to $\cF_+(E)$, or as the quotient of $\cL_d$ by the 
two sided \wot-closed commutator ideal corresponding to $\cF_-(E)$. From \ref{subsec:invNidealsNC} the following theorem follows. 

\begin{theorem}
$\cM_d$ is a dual algebra which has property $\mb{A}_1(1)$. In particular, the weak-$*$ and weak operator topologies on $\cM_d$ coincide. The same holds for quotients of $\cM_d$ by weak-$*$ closed ideals. 
\end{theorem}

\subsection{Subproduct systems}\label{subsec:sps} 

The commutative and noncommutative $d$-shifts were defined above in a way which might make it seems to depend on the choice of an orthonormal basis in a $d$-dimensional space $E$ (and, in the function space picture, on a choice of coordinate system in $\mb{C}^d$). Of course, the same structure is obtained regardless of the choice of basis (see, e.g., \cite{Arv98}). Alternatively, a coordinate free definition of the shift is given by viewing it as a representation of a {\em subproduct system}; see \cite{ShalitSolel} for details.

\section{Operator algebras associated to the $d$-shift}

\subsection{The norm closed algebra and the Toeplitz algebra}
Let $\cA_d$ be the norm closed algebra generated by $S$ on $H^2_d$. This algebra is sometimes referred to as the ``algebra of continuous multipliers", but this terminology is misleading --- see \ref{eq:not_cont_mult} below.
The {\em Toeplitz algebra} $\cT_d$ is defined to be the unital C*-algebra generated by $S$, that is, 
\be 
\cT_d = C^*(\cA_d) = C^*(1, S).
\ee
From \ref{subsec:thedshift} and \ref{subsec:ess_norm_d_shift} the following theorem follows (for proof see \cite[Theorem 5.7]{Arv98}).

\begin{theorem}\label{thm:exact_sequence}
Fix $d<\infty$ and denote the compact operators on $H^2_d$ by $\cK$. Then $\cK \subset \cA_d$, and 
\be\label{eq:quotient_cong}
\cT_d / \cK \cong C(\partial \mb{B}_d).
\ee
Thus, there exists an exact sequence 
\be
0 \longrightarrow \cK \longrightarrow \cT_d \longrightarrow C(\partial \mb{B}_d) \longrightarrow 0 .
\ee
\end{theorem}
The isomorphism (\ref{eq:quotient_cong}) is the natural one given by sending the image of $S_i$ in the quotient to the coordinate function $z_i$ on $\partial \mb{B}_d$. It follows that the essential norm $\|M_f\|_e$ of an element $f \in \cA_d$ is given by 
\be\label{eq:ess_norm_Ad}
\|M_f\|_e = \sup_{z \in \mb{B}_d} |f(z)| \quad , \quad f \in \cA_d.
\ee
Another consequence of the above theorem is
\be
\cT_d = \overline{\spn \cA_d \cA_d^*}^{\|\cdot\|}.
\ee

It is worth noting that for $d = \infty$ equation (\ref{eq:quotient_cong}) fails, because $S$ is not essentially normal in that case. There is a naturally defined ideal $\cI \triangleleft \cT_d$ (that coincides with $\cK$ when $d<\infty$) such that $\cT_d / \cI$ is  commutative. This ideal $\cI$ is given by 
\[
\cI = \{A \in \cT_d : \lim_{n\rightarrow 0}\|AP_{E^n}\| = 0\} , 
\]
where $P_{E^n}$ is the orthogonal projection $\cF_+(E) \rightarrow E^n$. 
The counterpart of (\ref{eq:quotient_cong}) still fails, instead one has
\[
\cT_\infty / \cI = C(\overline{\mb{B}_\infty}) . 
\]
See \cite[Example 3.6]{Viselter} for details.

\subsection{Continuous multipliers versus $\cA_d$}
It follows from (\ref{eq:ineq_norm}) and (\ref{eq:mult_norm}) that $\cA_d \subseteq C(\overline{\mb{B}_d}) \cap \cM_d$. When $d=1$ this containment is an equality, but for $d>1$ the reverse containment does not hold. 

Indeed, in \cite{FX11} it is proved that there is a sequence of continuous multipliers $\{\psi_k\}$ such that $\lim_{k\rightarrow \infty} \|\psi_k\|_\infty = 0$ while $\inf_k \|M_{\psi_k}\|_e \geq 1/2$. 
It follows that (\ref{eq:ess_norm_Ad}) does not hold for the multipliers $\psi_k$.  Since $\psi_k \in C(\overline{\mb{B}_d}) \cap \cM_d$, it follows that
\be\label{eq:not_cont_mult}
\cA_d \subsetneq C(\overline{\mb{B}_d}) \cap \cM_d .
\ee

\subsection{Nullstellensatz and approximation in homogeneous ideals}

\begin{definition}
Let $\cB \subseteq \cO(\mb{B}_d)$ be an algebra. An ideal $J \triangleleft \cB$ is said to be a {\em homogeneous ideal} if for every $f \in J$ with homogeneous decomposition (\ref{eq:homogeneous}) and every $n \in \mb{N}$, it holds that $f_n \in J$.
\end{definition}

\begin{definition}
Let $\cB \subseteq \cO(\mb{B}_d)$ be an algebra and $J \triangleleft \cB$ an ideal. The {\em radical of $J$} is the ideal
\[
\sqrt{J}  = \{f \in \cB : \exists N .\, f^N \in J\}. 
\]
An ideal $J$ is said to be a {\em radical ideal} if $\sqrt{J} = J$. 
\end{definition}

If $\cB \subseteq \cO(\mb{B}_d)$ is an algebra and $X \subseteq \mb{B}_d$ is a set, denote 
\[
I_\cB(X) = \{f \in \cB : f\big|_X \equiv 0\}. 
\]
For $J \subseteq \cB$ denote
\[
V(J) = \{z \in \mb{B}_d : f(z) = 0 \, \textrm{ for all } f \in J\}. 
\]

\begin{theorem}[\cite{DRS11}, Theorem 6.12; \cite{RamseyThesis}, Theorem 2.1.30]
Let $\cB$ be either $\cA_d$ equipped with the norm topology, or $\cM_d$ equipped with the weak-$*$ topology, and let $J \triangleleft \cB$ be a closed homogeneous ideal. Then 
\be
\sqrt{J} = I_\cB (V(J)). 
\ee
\end{theorem}

The above may be considered as a Nullstellensatz for homogeneous ideals in the algebra $\cB$. Besides its intrinsic interest, it also implies the following approximation-theoretic result. 

\begin{theorem}[\cite{DRS11}, Corollary 6.13; \cite{RamseyThesis}, Corollary 2.1.31]
Let $\cB$ be either $\cA_d$ equipped with the norm topology, or $\cM_d$ equipped with the weak-$*$ topology, and let $I$ be a radical homogeneous ideal in $\mb{C}[z]$. If $f \in \cB$ vanishes on $V(I)$, then $f \in \overline{I}$. 
\end{theorem}

In other words, if a function $f \in \cA_d$ vanishes on a homogeneous variety $V \subset \mb{B}_d$, then it can be approximated in norm (thus, uniformly) by polynomials that vanish on $V$. 

\begin{remark}
The results for $\cB = \cA_d$ were obtained in \cite{DRS11}, while the extension to $\cB = \cM_d$ is from \cite{RamseyThesis}. For brevity, this Section describes the results in the setting of either $\cA_d$ or $\cM_d$; but --- as the proof depends only on the fact that $\cA_d$ and $\cM_d$ are algebras of multipliers on a Hilbert function space with circular symmetry --- similar results hold in a more general setting, in particular in the setting of the ball algebra $A(\mb{B}_d)$ or $H^\infty(\mb{B}_d)$ (see the \cite{DRS11,RamseyThesis} for further details). In the setting of non-homogeneous ideals, however, not much is known. 
\end{remark}

\section{Model theory}\label{sec:model_theory}

The importance of the $d$-shift stems from the fact that it is a universal model for $d$-contractions, in fact, {\em the unique} universal model for $d$-contractions. The results of \ref{subsec:uni_pure} and \ref{subsec:Drury_inequality} have become well known thanks to their appearance in \cite{Arv98}, though these results and the techniques that give them have been known before, at least in some form or other (see, e.g., \cite{Ath90,Ath92,Drury,MV93,Popescu89,Popescu99,Vas92}), and have been extended and generalized afterwards (see, e.g.,  \cite{AEM02,AE02,MS09,Popescu06,Popescu10a}).

\subsection{Universality of the $d$-shift among pure row contractions}\label{subsec:uni_pure}

Recall the notation from Section \ref{sec:notation}.

\begin{lemma}\label{lem:Poisson_kernel}
Let $T$ be a pure $d$-contraction on a Hilbert space $H$. Then there exists an isometry $W : H \rightarrow H^2_d \otimes \cD_T$ such that for every multi-index $\alpha$ and all $g \in \cD_T = \overline{\Delta_T H}$, 
\be\label{eq:Poisson_kernel}
W^*(z^\alpha \otimes g) = T^\alpha \Delta g .
\ee
\end{lemma}
\begin{proof}
Fix a Hilbert space $E$ with orthonormal basis $\{e_1, \ldots, e_d\}$. In this proof, $H^2_d$ and $\cF_+(E)$ will be identified. Define an operator $W : H \rightarrow \cF(E) \otimes \cD_T$ by 
\[
Wh = \sum_{n=0}^\infty \sum_{i_1, \ldots, i_n=1}^d e_{i_1} \otimes \cdots \otimes e_{i_n} \otimes \Delta T^*_{i_n} \cdots T^*_{i_1} h.
\]
By purity, one has
\begin{align*}
\|Wh\|^2 &= \sum_{n=0}^\infty \sum_{i_1, \ldots, i_n=1}^d \langle T_{i_1} \cdots T_{i_n }\Delta^2 T^*_{i_n} \cdots T^*_{i_1} h, h \rangle \\ 
&= \lim_{N\rightarrow \infty}\sum_{n=0}^N \langle (\Theta_T^n(I) - \Theta_T^{n+1}(I))h, h \rangle \\
&= \langle h, h \rangle - \lim_{N\rightarrow \infty} \langle \Theta_T^{N+1}(I)h, h \rangle = \|h\|^2. 
\end{align*}
From commutativity of $T$ is follows that $W$ maps $H$ into $\cF_+(E) \otimes \cD_T$. Finally, letting $g \in \cD_T$ and $h \in H$, it holds that
\begin{align*}
\langle W^* (e^\alpha \otimes g) ,  h \rangle 
&= \sum_{n=0}^\infty \sum_{i_1, \ldots, i_n=1}^d \langle e^\alpha \otimes g, e_{i_1} \otimes \cdots \otimes e_{i_n} \otimes \Delta T^*_{i_n} \cdots T^*_{i_1} h \rangle \\
&= \frac{|\alpha|!}{\alpha!} \|e^\alpha\|^2 \langle T^\alpha \Delta g, h \rangle \\
&= \langle T^\alpha \Delta g, h \rangle. 
\end{align*}
Identifying $z^\alpha$ with $e^\alpha$ gives (\ref{eq:Poisson_kernel}). 
\end{proof}

If $A$ is tuple of operators on $G$, a subspace $K \subseteq G$ is said to {\em co-invariant for} $A$ if $K$ is invariant for $A^*$, (equivalently, if $A K^\perp \subseteq K^\perp$).

\begin{theorem}\label{thm:model_pure}
Let $T$ be a pure $d$-contraction on $H$. Then there exists a subspace $K \subset H^2_d \otimes \cD_T$ invariant for $S^*$ such that $T$ is unitarily equivalent to the compression of $S \otimes I_{\cD_T}$ to $K$. To be precise, there is an isometry $W: H \rightarrow H^2_d \otimes \cD_T$ such that $W(H) = K$ and 
\be\label{eq:model_pure}
T^* = W^* \big( S^* \otimes I_{\cD_T} \big)\big|_K W. 
\ee
\end{theorem}
\begin{proof}
Let $W$ be as in Lemma \ref{lem:Poisson_kernel} and denote $K = W(H)$. From (\ref{eq:Poisson_kernel}) one finds $W^* (S \otimes I_{\cD_T}) = T W^*$, thus  $(S \otimes I_{\cD_T})^* W = W T^*$. From this the invariance of $K$ under $S^* \otimes I_{\cD_T}$ as well as (\ref{eq:model_pure}) follow. 
\end{proof}

In particular, if one identifies $H$ with $K$ via $W$, then for every polynomial $p \in \mb{C}[z]$
\be\label{eq:dilation_S}
p(T) = P_K (p(S) \otimes I) P_K. 
\ee

\subsection{Drury's inequality}\label{subsec:Drury_inequality}

The following facts are well known (see \cite{SzNF2010} or the chapter on commutative dilation theory by Ambrozie and M\"{u}ller \cite{AM25} in this reference work):
\begin{enumerate}
\item (von Neumann's inequality \cite{vN}) For every contraction $T$ and every polynomial $p$, 
\[
\|p(T)\| \leq \sup_{|z|\leq 1}|p(z)| .
\]
\item (Ando's inequality \cite{Ando}) For every pair of commuting contractions $S,T$ and every bivariate polynomial $p$,
\[
\|p(S,T)\| \leq \sup_{|y|,|z|\leq 1}|p(y,z)|. 
\]
\item (Varopoulos's example \cite{Varopoulos74}) There exists a triple of commuting contractions $R,S,T$ and a polynomial in three variables $p$ such that 
\[
\|p(R,S,T)\|>\sup_{|x|,|y|,|z|\leq 1}|p(x,y,z)|.
\]
\end{enumerate}

Thus, the naive generalization of von Neumann's inequality to the multivariate setting,  
\be\label{eq:vNpoly}
\|p(T)\| \stackrel{?}{\le} \|p\|_{\infty, \mb{D}^k}
\ee
for every $k$-tuple of commuting contractions, fails. The failure of von Neumann's inequality (\ref{eq:vNpoly}) in the multivariate setting, and the search for a suitable replacement that does work for several commuting operators, have been and are still the subject of great interest. A candidate for a replacement of von Neumann's inequality was obtained by Drury \cite{Drury}. 

\begin{theorem}\label{thm:DruryVNineq}
Let $T$ be a $d$-contraction. Then for every matrix valued polynomial $p \in \mb{C}[z_1, \ldots, z_d] \otimes M_k(\mb{C})$, 
\be\label{eq:vNineq}
\|p(T)\| \leq \|p(S)\|. 
\ee
\end{theorem}
\begin{proof}
It is enough to prove this inequality for $rT$ instead of $T$, for all $r \in (0,1)$. But as $rT$ is pure, the inequality $\|p(rT)\| \leq \|p(S)\|$ is a direct consequence Theorem \ref{thm:model_pure} (or equality (\ref{eq:dilation_S})). 
\end{proof}

When $d=1$ then the above theorem reduces to von Neumann's inequality. When $d=2$ then the above theorem fundamentally differs from Ando's inequality: one cannot replace the right hand side by multiple of the sup norm of $p$ on the ball (cf. Theorem \ref{thm:MnotHinfty}).

\subsection{Universality of the $d$-shift among $d$-contractions}\label{subsec:uni}

The model theory for $d$-contractions reached final form in \cite[Theorem 8.5]{Arv98}, and is presented in Theorem \ref{thm:arvesons_model} below. 
For a precise formulation additional terminology is required. 

\begin{definition}
Let $A$ be a tuple of operators on a Hilbert space $G$ and $K$ a subspace of $G$ which is co-invariant for $A$. $K$ is said to be {\em full} if 
\[
G = [C^*(1,A)K].
\]
\end{definition}

\begin{definition}
A {\em spherical unitary} is a $d$-tuple $Z$ of commuting normal operators such that $\sum_i Z_i Z^*_i = 1$. 
\end{definition}

Fix $d \in \{1, 2, \ldots, \infty\}$. Given $n \in \{0, 1,2, \ldots, \infty\}$, one denotes by $n \cdot S$ the direct sum of $S$ with itself $n$ times acting on $n \cdot H^2_d$. Given a spherical unitary $Z = (Z_1, \ldots, Z_d)$ on a Hilbert space $H_Z$, one writes $n \cdot S \oplus Z$ for the $d$-contraction 
\[
(\underbrace{S_1 \oplus \cdots \oplus S_1}_{n \textrm{ times}} \oplus Z_1, \ldots, \underbrace{S_d \oplus \cdots \oplus S_d}_{n \textrm{ times}} \oplus Z_d)
\]
on $\underbrace{H^2_d \oplus \cdots \oplus H^2_d}_{n \textrm{ times}} \oplus H_Z$. This notation is extended to allow also the cases in which one of the summands is absent. The case in which $S$ is absent corresponds to $n=0$. In the case where $Z$ is absent we say that $Z$ is {\em nil}. 

\begin{theorem}\label{thm:arvesons_model}
Let $d < \infty$ and let $T$ be a $d$-contraction on a separable Hilbert space. Then there is an $n \in \{0,1,2, \ldots, \infty\}$, a spherical unitary $Z$ on $H_Z$, and subspace $K \subseteq n \cdot H^2 \oplus H_Z$ that is co-invariant and full for $n \cdot S \oplus Z$, such that $T$ is unitarily equivalent to the compression of $n \cdot S \oplus Z$ to $K$. 

The triple $(n,Z,K)$ is determined uniquely, up to unitary equivalence, by the unitary equivalence class of $T$. Moreover, $Z$ is the nil operator if and only if $T$ is pure, and $n = \rank(T)$.
\end{theorem}

\begin{proof}
The main ingredient of the proof is a combination of Arveson's extension theorem \cite{Arv69} and Stinespring's dilation theorem \cite{Stinespring}. This method has appeared first in \cite{Arv72}, and has been reused many times to obtain many dilation theorems. It runs as follows. 

Suppose that $T$ acts on $H$. By Theorem \ref{thm:DruryVNineq}, the map $S_i \mapsto T_i$ extends to a unital completely contractive homomorphism $\Psi : \cA_d \rightarrow B(H)$. By Arveson's extension theorem \cite[Theorem 1.2.9]{Arv69}, $\Psi$ extends to a unital completely positive map $\Psi : \cT_d \rightarrow B(H)$. By Stinespring's theorem \cite{Stinespring}, there is a Hilbert space $G$, an isometry $V: H \rightarrow G$, and a $*$-representation $\pi: \cT_d \rightarrow B(G)$ such that 
\[
\Psi(X) = V^* \pi(X) V \,\, , \,\, X \in \cT_d,
\]
and such that $G = [\pi(\cT_d)VH]$. The space $K = VH$ is full and co-invariant for $\pi(S)$, and $V$ implements a unitary equivalence between $T$ and a compression of $\pi(S)$. 

Using Theorem \ref{thm:exact_sequence}, basic representation theory (see \cite[Section 1.3]{Arv76}) shows that $\pi$ breaks up as a direct sum $\pi = \pi_a \oplus \pi_s$, where $\pi_a$ is a multiple of the identity representation and $\pi_s$ annihilates the compacts. It follows that $\pi_a(S) = n \cdot S$, that $Z:=\pi_s(S)$ is a spherical unitary, and that $\pi(S) = n \cdot S \oplus Z$ dilates $VTV^*$. That shows that a model as stated in the first part of the proof exists. The remaining details are omitted. 
\end{proof}

\begin{remark}
The above theorem and proof are also valid in the case $d = \infty$, with the important change that $Z$ is no longer a spherical unitary, but merely a commuting tuple satisfying $\sum Z_i Z^*_i = 1$. In particular, $Z$ is not necessarily normal, hence in this case the model reveals far less than in the $d<\infty$ case. 
\end{remark}

Theorem \ref{thm:arvesons_model} implies the following subnormality result due originally to Athavale \cite{Ath90}. 
\begin{corollary}
Let $T$ be a commuting $d$-tuple ($d<\infty$) on a Hilbert space such that $T_1^* T_1 + \ldots + T_d^*T_d = 1$. Then $T$ is subnormal. 
\end{corollary}

\subsection{Uniqueness of the $d$-shift}\label{subsec:unique_model}

The $d$-shift serves as a universal model for pure row contractions (Theorems \ref{thm:model_pure} and \ref{thm:DruryVNineq}). For $d>1$, and in contrast to the case $d=1$, the $d$-shift turns out to be the {\em unique} model for pure row contractions in the following sense. 

\begin{theorem}[\cite{Arv98}, Lemma 7.14; see also \cite{RS10}]
Suppose $d \geq 2$, let $T$ be a $d$-contraction acting on $H$, and let $K \subseteq H$ be a subspace such that the compressed tuple $P_K TP_K$ is unitarily equivalent to the $d$-shift. Then $K$ reduces $T$. 
\end{theorem}

For additional uniqueness and maximality properties of the $d$-shift, see \cite[Section 7]{Arv98}.

\subsection{The noncommutative setting}\label{subsec:noncommutative}

The methods used above to show that $S$ is a universal model for $d$-contractions work in a greater generality, to provide various universal models for tuples of operators satisfying certain constraints. 

The key to these results is to examine what happens to the proof of Lemma \ref{lem:Poisson_kernel} when a row contraction $T$ satisfies more, or less, assumptions other than the assumption of being a commuting tuple. When $T$ satisfies no assumptions besides that it be a row contraction, then the range of $W$ appearing in the proof of the lemma might be larger than $\cF_+(E)$. Thus the commutative $d$-shift $S$ has to be replaced by the noncommutative $d$-shift $L$ on $\cF(E)$.

A tuple $V = (V_1 \ldots, V_d)$ on a Hilbert space $G$ is said to be a {\em row isometry} if $V^*_i V_j = \delta_{ij} I_G$ for all $i,j$, which means that all the $V_i$s are isometries with mutually orthogonal ranges. A row isometry is said to be of {\em Cuntz type} if $\sum V_i V_i^* = I_G$ (the convergence being understood as strong operator convergence in the case $d=\infty$). Applying the same reasoning one obtains the following theorem of Bunce \cite{Bunce}, Frazho \cite{Frazho} and Popescu \cite{Popescu89}, which is a natural generalization of the Sz.-Nagy isometric dilation theorem \cite{SzN53}.

\begin{theorem}\label{thm:model_NC}
Let $d \in \{1, 2, \ldots, \infty\}$ and let $T$ be a row contraction on a separable Hilbert space. Let $L$ be the noncommutative shift acting on $\cF(E)$, where $\dim E = d$. Then there is an $n \in \{0,1,2, \ldots, \infty\}$, a row isometry $V$ of Cuntz type acting on $H_V$, and a subspace $K \subseteq n \cdot \cF(E) \oplus H_V$ that is co-invariant and full for $n \cdot L \oplus V$, such that $T$ is unitarily equivalent to the compression of $n \cdot L \oplus V$ to $K$. 

The triple $(n,V,K)$ is determined uniquely, up to unitary equivalence, by the unitary equivalence class of $T$. Moreover, $V$ is the nil operator if and only if $T$ is pure, and $n = \rank(T)$. 
\end{theorem}

\subsection{Constrained dilations}\label{subsec:constrained}

The universality of the commutative and noncommutative $d$-shifts (Theorems \ref{thm:arvesons_model} and \ref{thm:model_NC}) can be interpreted in the following way. 

Fix $d$ and let $E$ be a $d$-dimensional Hilbert spaces with fixed orthonormal basis $\{e_1, \ldots, e_d\}$, giving rise to the noncommutative $d$-shift $L = (L_1, \ldots, L_d)$.  
Let $\mb{C}\langle z \rangle = \mb{C}\langle z_1, \ldots, z_d\rangle$ denote the free algebra in $d$ variables. Let $\mathfrak{C}$ be the commutator ideal in $\mb{C}\langle z \rangle$, that is, the ideal generated by $fg - gf$, where $f,g \in \mb{C}\langle z \rangle$. Note that $\mb{C} \langle z \rangle / \mathfrak{C} = \mb{C}[z]$. 
Now consider the closed subspace $[\mathfrak{C}]$ in $\cF(E)$ (here $\mb{C}\langle z \rangle$ is identified with a dense subspace of $\cF(E)$ in the natural way). Then $[\mathfrak{C}]$ is an invariant subspace for $L$, and $\cF_+(E) = [\mathfrak{C}]^\perp$. Recall also that $S = P_{\cF_+(E)} L P_{\cF_+(E)}$. 

The noncommutative $d$-shift $L$ is a universal for row contractions, and the commutative $d$-shift $S$ is universal for {\em commuting} row contractions. Now, a row contraction $T$ is commuting if and only if it satisfies the relations in $\mathfrak{C}$, that is, $p(T) = 0$ for every $p \in \mathfrak{C}$. Thus the above discussion can be summarized in the following way: {\em the universal model for row contractions which satisfy the relations in $\mathfrak{C}$ is obtained by compressing $L$ to $\cF_+(E) = [\mathfrak{C}]^\perp$}. 

Popescu discovered that the same holds when $\mathfrak{C}$ is replaced by an arbitrary ideal $J \triangleleft \mb{C}\langle z \rangle$: using more or less the same methods as above one obtains a universal model for row contractions satisfying the relations in $J$ by compressing the noncommutative $d$-shift $L$ to the co-invariant subspace $\cF_J = [J]^\perp$. See \cite{Popescu06} for details; similar results for special classes of ideals appear in \cite{BB02,ShalitSolel}. 

\subsection{Constrained dilations in the commutative case}\label{subsec:con_comm}
The results of \cite{Popescu06} discussed in the previous paragraph can be compressed to the commutative case, yielding the following model theory for $d$-contraction satisfying polynomial relations. 

For $J \triangleleft \mb{C}[z_1, \ldots, z_d]$ an ideal in the algebra of $d$-variable (commutative) polynomials, let $[J]$ be its closure in $H^2_d$, and denote $\cF_J = [J]^\perp$ and $S^J = P_{\cF_J} S P_{\cF_J}$. The tuple $S^J$ gives $\cF_J$ the structure of a Hilbert module, and it can be identified naturally with the quotient of $H^2_d$ by the submodule $[J]$.

A row contraction $V$ is said to be {\em of Cuntz type} if $\sum V_i V_i^* = 1$. 

\begin{theorem}
Fix $d$, and let $J \triangleleft\, \mb{C}[z_1, \ldots, z_d]$ be an ideal. Let $T$ be a $d$-contraction such that $p(T) = 0$ for every $p \in J$. Then there is a cardinal $n$, a row contractions $V$ of Cuntz type on $H_V$ satisfying $p(V) = 0$ for all $p \in J$, and subspace $K \subseteq n \cdot \cF_J \oplus H_V$ that is co-invariant and full for $n \cdot S^J \oplus V$, such that $T$ is unitarily equivalent to the compression of $n \cdot S^J \oplus V$ to $K$. 
Moreover, $V$ is the nil operator if and only if $T$ is pure, and $n = \rank(T)$. 
\end{theorem}

\begin{remark}
Under some additional conditions (for example, if $J$ is a homogeneous ideal) the triple $(n,V,K)$ is determined uniquely, up to unitary equivalence, by the unitary equivalence class of $T$. 
\end{remark}

\begin{remark}
For non-pure $d$-contractions the above model may not be very effective, since there is not much information on what $V$ looks like. It can be shown, however, that if $S^J$ is essentially normal (equivalently, if $\cF_J$ is an essentially normal Hilbert module) then $V$ is a normal tuple with spectrum in $\overline{V(J)} \cap \partial \mb{B}_d$. 
\end{remark}

\subsection{Other commutative models}
See the chapter on commutative dilation theory by Ambrozie and M\"{u}ller \cite{AM25} in this reference work for a systematic construction of alternative models, given either by weighted shifts or by multiplication operators on spaces of analytic functions, which include the $d$-shift as a special case.

\subsection{Noncommutative domains}

In a different direction of generalization, Popescu obtained universal models for tuples satisfying a variety of different norm constraints, which include the row contractive condition as a special case \cite{Popescu10a}. For example, under some assumptions on the coefficients $a_\alpha$, Popescu obtains a model for all tuples $T$ which satisfy 
\[
\sum_{\alpha} a_\alpha T^\alpha T^{\alpha*} \leq I. 
\]

\subsection{Commutant lifting}\label{subsec:CLT}

The classical Sz.-Nagy and Foias model theory \cite{SzNF2010} finds some of its most profound applications via the commutant lifting theorem \cite{SzNF68} (see also \cite{FoiasFrazhoBook}). It is natural therefore to expect a commutant lifting theorem in the setting of the model of \ref{subsec:uni}. The following theorem is due to Ball, Trent and Vinnikov \cite{BTV01} (see also \cite{AT02}). 
\begin{theorem}[\cite{BTV01}, Theorem 5.1]\label{thm:BTV}
Let $K_1$ and $K_2$ be Hilbert spaces. For $i=1,2$, suppose that $M_i \subseteq H^2_d \otimes K_i$ is co-invariant for $S \otimes I_{K_i}$. Suppose that $X \in B(M_1, M_2)$ satisfies
\[
X^* (S \otimes I_{K_2})^*\big|_{M_2} = (S \otimes I_{K_1})^* X^*. 
\]
Then there exists $\Phi \in \cM_d(K_1, K_2)$ such that 
\begin{enumerate}
\item $M_\Phi^*\big|_{M_2} = X^*$,
\item $\|M_\Phi\| = \|X\|$.
\end{enumerate}
\end{theorem}

Theorem \ref{thm:BTV} provides a commutant lifting result for the model of \ref{subsec:uni} only in the case where $Z$ is the nil operator. The following theorem of Davidson and Le handles the non-pure case. If $T$ is a $d$-contraction and $\tilde{T} = n \cdot S \oplus Z$ is the dilation given by Theorem \ref{thm:arvesons_model} on $\tilde{H} = n \cdot H^2_d \oplus H_Z$, then one may consider $H$ as a subspace of $\tilde{H}$ and $T$ as the co-restriction of $\tilde{T}$ to $H$.  

\begin{theorem}[\cite{DL10}, Theorem 1.1]\label{thm:DavLe}
Suppose that $T = (T_1, \ldots, T_d)$ is a $d$-contraction on a Hilbert
space $H$, and that $X$ is an operator on $H$ that commutes with $T_1, \ldots, T_n$. Let $\tilde{T} = (\tilde{T}_1, \ldots, \tilde{T}_d)$ on $\tilde{H}$ denote the dilation of T on provided by Theorem \ref{thm:arvesons_model}. Then there is an operator $Y$ on $\tilde{H}$ that commutes with each $\tilde{T}_i$ for $i=1, \ldots, d$, such that 
\begin{enumerate}
\item $ Y^* \big|_H = X^*$ .
\item $\|Y\| = \|X\|$. 
\end{enumerate}
\end{theorem}

\begin{remark}
There is also a commutant lifting theorem in the setting of \ref{subsec:noncommutative} (see \cite[Theorem 3.2]{Popescu89}), and this commutant lifting theorem can be ``compressed" down to co-invariant subspaces of $L$, giving rise to a commutant lifting theorem (for pure row contractions) in the constrained setting of \ref{subsec:constrained}. In particular one can obtain Theorem \ref{thm:BTV} above as a bi-product of the noncommutative theory in this way (see  \cite[Section 3]{DL10} or \cite[Theorem 5.1]{Popescu06}). 
\end{remark}

\section{Interpolation theory and function theory on subvarieties}

\subsection{Zero sets and varieties}\label{sec:zero_sets}

\begin{definition}
Let $\cF$ be a space of functions on a set $X$. Then a set $Y \subseteq X$ is said to be a {\em zero set for $\cF$} if there is an $f \in \cF$ such that $Y = \{x \in X : f(x) = 0\}$. $Y$ is said to be a {\em weak zero set} if it is the intersection of zero sets. 
\end{definition}

As $\cM_d \subseteq H^2_d$, every zero set of $\cM_d$ is a zero set of $H^2_d$. 
The converse also holds; see Theorem \ref{thm:Smirnov_factor} and the succeeding paragraph. 
For now we need the following partial result. 

\begin{theorem}\label{thm:sero_set}
If $V \subseteq \mb{B}_d$ is a zero set for $H^2_d$, then it is a weak zero set for $\cM_d$. 
\end{theorem}
\begin{proof}
See \cite[Theorem 9.27]{AM02}, where this result is proved for any complete Pick Hilbert function space and its multiplier algebra. 
\end{proof}

\begin{definition}
Say that $V$ is a {\em variety} in $\mb{B}_d$ if it is is a weak zero set of $\cM_d$, that is, if it is defined as 
\[
V = V(F) :=\{\lambda \in \mb{B}_d :  f(\lambda) = 0 \textrm{ for all } f \in F\}, 
\]
for some $F\subseteq \cM_d$.
\end{definition} 

\begin{remark}
By Theorem \ref{thm:sero_set}, replacing $H^2_d$ by $\cM_d$ would lead to an equivalent definition. 
\end{remark}

\begin{remark}
This is not the usual definition of {\em analytic variety}, as only subsets $F\subseteq \cM_d$ are allowed. Considering the familiar case $d=1$ shows that the above definition is more restrictive than the usual one: any discrete set in $\mb{D}$ is an analytic variety, but only sequences satisfying the {\em Blaschke condition} can be zero sets of functions in $H^\infty(\mb{D}) = \cM_1$ \cite[Section II.2]{Garnett}). 
\end{remark}

It is immediate that if $J$ is the \wot-closed ideal generated by $F$, then $V(F) = V(J)$. 
Given $X \subseteq \mb{B}_d$, denote by $J_X$ the \wot-closed ideal
\[
J_X = \{f \in \cM_d : f(x) = 0 \textrm{ for all } x \in X\}. 
\]
Then $J_{X} = J_{V(J_X)}$.

For $X \subseteq \mb{B}_d$, denote by $\cH_X = \overline{\spn}\{k_x : x \in X\}$. 

\begin{lemma}[\cite{DRS11}, Lemma 5.5]
If $J$ is a radical homogeneous ideal in $\mb{C}[z]$, then 
\[
\cH_{V(J)} = \cF_J := H^2_d \ominus J.  
\]
\end{lemma}

\begin{lemma}[\cite{DRS15}, Section 2]
If $V \subseteq \mb{B}_d$ is a variety and $X$ is a set, then $V = V(J_V)$ and $\cH_X = \cH_{V(J_X)}$. 
\end{lemma}

$V(J_X)$ is the smallest variety containing $X$, thus the final assertion of the above lemma can be rephrased to say that the space $\cH_X$ does not change when one replaces $X$ by its ``Zariski closure".

\subsection{The complete Pick property}

\begin{definition}\label{def:CP}
Let $\cH$ be a Hilbert function space on $X$, and let $K^\cH$ be its kernel. Then $\cH$ said to have the {\em complete Pick property} if the following two conditions are equivalent: 
\begin{enumerate} 
\item For all $m,n \in \mb{N}$, all $n$ points $x_1, \ldots, x_n \in X$ and all matrices $W_1, \ldots, W_n \in M_m(\mb{C})$, there is a contractive operator valued multiplier $\Phi \in \Mult(\cH) \otimes M_m(\mb{C})$ such that $\Phi(x_i) = W_i$ for all $i=1, \ldots, n$, 
\item The following $mn \times mn$ matrix is positive semi-definite:
\be\label{eq:CompPick}
\left[(I - W_j W_i^*)K^\cH(x_j,x_i) \right]_{i,j=1}^n \geq 0 .
\ee
\end{enumerate}
\end{definition}

If $\cH$ has the complete Pick property then it is said to be a {\em complete Pick space}, the kernel $K^\cH$ is said to be a {\em complete Pick kernel}, and the multiplier algebra $\Mult(\cH)$ is said to be a {\em complete Pick algebra}. Some researchers use the term {\em complete Nevanlinna-Pick kernel} instead of complete Pick kernel, etc. The terminology comes from the fact that, if $m=1$, $\cH$ is the Hardy space on the disc $H^2(\mb{D})$ and $K^\cH$ is the Szeg\H{o} kernel $s(z,w) = \frac{1}{1-z\overline{w}}$, then (\ref{eq:CompPick}) is the necessary and sufficient condition given by Pick's classical interpolation theorem \cite[Theorem I.2.2]{Garnett}. 

The reader is referred to \cite{AM02} for background and complete treatment of interpolation problems of this sort.  

\begin{remark}
One may also consider the {\em operator valued Pick property}, where the matrices $W_1, \ldots, W_n \in M_m(\mb{C})$ in the above definition are replaced with an $n$-tuple of operators on some Hilbert space $K$, and the required $\Phi$ is a $B(K)$ valued function on $X$ multiplying $\cH\otimes K$ into itself. However, it can be shown that the operator valued Pick property is equivalent to the complete Pick property. 
\end{remark}

In any Hilbert function space (\ref{eq:CompPick}) is a necessary condition for the existence of a contractive multiplier $\Phi$ that satisfies $\Phi(x_i) = W_i$ for all $i = 1, \ldots, n$ \cite[Theorem 5.8]{AM02}. Complete Pick spaces are the spaces in which (\ref{eq:CompPick}) is also a sufficient condition. 

\begin{theorem}\label{thm:comp_Pick}
The Drury-Arveson space $H^2_d$ has the complete Pick property. 
\end{theorem}
\begin{proof}
This theorem has several proofs. 

A Hilbert function space theoretic proof was given by Agler and M\raise.45ex\hbox{c}Carthy \cite{AM00} (following works of McCullough \cite{McCullough} and Quiggin \cite{Quiggin}). In fact \cite{AM00} characterizes all complete Pick kernels, showing that an irreducible kernel $K^\cH$ is a complete Pick kernel if and only if for any finite set $x_1, \ldots, x_n$, the matrix 
\[
\left[\frac{1}{K^\cH(x_j,x_i)}\right]_{i,j=1}^n
\]
has exactly one positive eigenvalue. The kernel (\ref{eq:kernel}) is easily seen to satisfy this property. 

A proof based on the commutant lifting theorem \ref{thm:BTV} was given by Ball, Trent and Vinnikov \cite[p. 118]{BTV01} (see also \cite{AriasPopescu00} for a proof via noncommutative commutant lifting). The proof, based on a deep idea which goes back to \cite{Sarason}, runs as follows. 

Let $x_1, \ldots, x_n \in \mb{B}_d$ and $W_1, \ldots, W_n \in M_m(\mb{C})$ be as in Definition \ref{def:CP}. Put $H = \mb{C}^m$, and define
\[
N_1 = \overline{\spn}\{k_{x_i} \otimes h : i=1, \ldots, n; \, h \in H\}
\]
and 
\[
N_2 = \overline{\spn}\{k_{x_i} \otimes W_i^*h : i=1, \ldots, n; \, h \in H\}.
\]
By (\ref{eq:what_mutlpliers_do}), $N_1$ and $N_2$ are co-invariant. Now define $X : N_2 \rightarrow N_1$ to be the adjoint of the operator $X^*: N_1 \rightarrow N_2$ defined by 
\[
X^*(k_{x_i} \otimes h) = k_{x_i} \otimes W_i^*h \quad, \quad i=1, \ldots, n; \,  h \in H. 
\]
It is clear that $X^* (S \otimes I)^*\big|_{N_1} = (S \otimes I)^* X^*$, and the condition (\ref{eq:CompPick}) implies that $\|X^*\|\leq 1$. By Theorem \ref{thm:BTV} there exists a contractive multiplier $\Phi \in \cM_d(H)$ satisfying $M_\Phi^* \big|_{N_1} = X^*$. Since 
\[
M_\Phi^* k_{x_i} \otimes h = k_{x_i} \otimes \Phi(x_i)^* h
\]
for all $h \in H$, it follows that $\Phi(x_i) = W_i$. 

An alternative proof is provided in \cite[p. 108]{BTV01} (see also \cite{EP02}) using what is sometimes called ``the lurking isometry" argument. The main idea is that (\ref{eq:CompPick}) is used to construct directly a unitary as in (\ref{eq:lurking}) which realizes the interpolating multiplier by formula (\ref{eq:realization}).

Finally, there is also a proof that passes through the noncommutative setting via a distance formula, found independently by Davidson and Pitts \cite{DavPittsPick} and by Arias and Popescu \cite{AriasPopescu00}. The roots of this proof can also be traced back to \cite{Sarason}. Here are a few details of the proof, compressed to the commutative setting. 

Suppose that (\ref{eq:CompPick}) holds, and for simplicity assume that $W_1, \ldots, W_n$ are all in $\mb{C}$. It is easy to see that there is {\em some} function $f \in \cM_d$ that satisfies $f(x_i) = W_i$ for $i=1, \ldots, n$. The norm of $f$ could be anything, but it can be modified by adding a function vanishing on $\{x_1, \ldots, x_n\}$. Let $J$ be the ideal
\[
J = \{g \in \cM_d : g(x_i) = 0, \,\, i= 1, \ldots, n\}.
\] 
If $h$ is another multiplier satisfying $h(x_i) = W_i$ for $i=1, \ldots, n$, then there is some $g \in J$ such that $h = f+g$. Thus, there is a multiplier $h \in \cM_d$ satisfying $\|h\|\leq 1$ and $h(x_i) = W_i$ for $i=1, \ldots, n$ if and only if $\inf_{g \in J} \|f+g\| =  \dist(f,J) \leq 1$. By the Arias-Popescu/Davidson-Pitts distance formula alluded to above (\cite[Proposition 1.3]{AriasPopescu00} and \cite[Theorem 2.1]{DavPittsPick}),  
\be\label{eq:dist_form}
\dist(f,J) = \|P_{N} M_f P_N\| ,
\ee
where $N = [J]^\perp = \spn\{k_{x_i} : i=1, \ldots, n\}$. A computation now shows that $\|P_N M_f P_N\| \leq 1$ is equivalent to (\ref{eq:CompPick}). 
\end{proof}

\begin{remark}
The second and fourth proofs described above (using commutant lifting or the  distance formula) generalize easily to give additional interpolation theorems for the algebra $\cM_d$, such as Carath\'{e}odory interpolation (see \cite{AriasPopescu00,DavPittsPick}). The third proof (the ``lurking isometry" argument) can be used to obtain interpolation results in other algebras of functions (for example $H^\infty(\mb{D}^2)$), and further results as well (see \cite{AM99,BT98}). The first proof is based on the characterization of complete Pick kernels, from which it follows that the kernel (\ref{eq:kernel}) of the space $H^2_d$ plays a universal role; this is discussed in the next paragraph.  
\end{remark}

\subsection{The universal kernel}
For $d \in \{1, 2, \ldots, \infty\}$, the notation $k^d$ will be used below to denote the kernel (\ref{eq:kernel}) of $H^2_d$, to emphasize the dependence on $d$. 

\begin{definition}[\cite{AM02}, Definition 7.1]\label{def:irreducible}
Let $\cH$ be a Hilbert function space on a set $X$ with kernel $K^\cH$. The kernel $K^\cH$ is said to be {\em irreducible} if 
\begin{enumerate}
\item For every $x \neq y$ in $X$, $K^\cH_x = K^\cH(\cdot,x)$ and $K^\cH_y = K^\cH(\cdot,y)$ are linearly independent. 
\item For all $x,y \in X$, $K^\cH(x,y) \neq 0$. 
\end{enumerate}
\end{definition}
The reader should be aware that some authors prefer to use the term ``irreducible" for kernels that satisfy the second condition but not the first. 
It is a fact that for every Pick kernel $k$ on a set $X$ there is a partition $X = \uplus_i X_i$ such that $k(x,y) \neq 0$ if and only if $x$ and $y$ belong to the same $X_i$ \cite[Lemma 7.2]{AM02}. 
Thus, in principle, one may reduce the study of general Pick spaces to the  study of those spaces for which the second condition in the definition holds. 
The first condition is satisfied in the cases of most interest, but there are significant examples of spaces where it is not; see Remark \ref{rem:injective}.

\begin{definition}
If $\cH$ is a Hilbert function space on $X$ with kernel $K^\cH$ and $\mu : X \rightarrow \mathbb{C}$ is a non-vanishing function, then one denotes by $\mu \cH$ the Hilbert function space $\{\mu f : f \in \cH\}$. 
\end{definition}

\begin{remark}
The kernel of $\mu \cH$ is given by 
\[
K^{\mu\cH}(x,y) = \mu(x) \overline{\mu(y)} K^\cH(x,y) .
\]
It follows from this that $\cH$ and $\mu \cH$ have identical multiplier algebras, meaning that the set of multipliers is the same and that the multiplier norm is also the same. 
\end{remark}

Agler and M\raise.45ex\hbox{c}Carthy showed that $H^2_d$ is a universal complete Pick space in the sense of the following theorem. 

\begin{theorem}[\cite{AM00}, Theorem 4.2]\label{thm:compPickUniversal}
Let $\cH$ be a separable Hilbert function space with an irreducible kernel $K^\cH$. Then $K^\cH$ is a complete Pick kernel if and only if there is a cardinal number $d \leq \aleph_0$, an injective function $f: X \rightarrow \mb{B}_d$ and a non-vanishing function $\delta : X \rightarrow \mathbb{C}$ such that 
\be\label{eq:form_of_kernel}
K^\cH(x,y) = \delta(x) \overline{\delta(y)} k^d(f(x), f(y)) = \frac{\delta(x) \overline{\delta(y)}}{1 - \langle f(x), f(y) \rangle}.
\ee
Moreover, if this happens, then the map $K^\cH_x \mapsto \delta(x) k^d_{f(x)}$ is an isometry from $\cH$ onto a subspace of $\delta\circ f^{-1} H^2_d$. 
\end{theorem}

\begin{remark}\label{rem:separable}
Theorem \ref{thm:compPickUniversal} was proved in \cite{AM00} in the greater (yet hardly ever considered) generality when $\cH$ is not assumed separable, 
in which case $d$ must be taken to be an uncountable cardinal number (the definition of $H^2_d$ --- either as the completion of polynomials or as a Hilbert function space --- makes sense for any cardinality $d$). 
The fact that $d$ can be taken to be countable when $\cH$ is separable is noted in \cite{AM00}, but the details were not spelled out. 
Let us give a few details, assuming for brevity that $\delta \equiv 1$. 

Suppose that we have the conclusion of Theorem \ref{thm:compPickUniversal} with $d$ some cardinal number.  
Then by making use the isometric map $K^\cH(\cdot, x) \mapsto k^d(\cdot,f(x))$, we can naturally identify $\cH$ with the Hilbert function space on the set $Y = f(X)$ with kernel $k = k^d$. 
If $\cH$ is separable then there exists a sequence of points $\{y_n\}_n \subseteq Y$ such that $\overline{\spn}\{k_{y_n}\}_n = \cH$. 
We claim that the set $Y$ is contained in the separable subspace $\overline{\spn}\{y_n\}_n$. 
Indeed, if a point $z$ is not in this subspace, then the kernel function $k_z$ is not in the closed linear span of the $\{k_{y_n}\}_n$. 
On the other hand, if $z \in Y$ then $k_z \in \cH = \overline{\spn}\{k_{y_n}\}_n$. 
We conclude that $f(X) = Y \subset \overline{\spn}\{y_n\}_n$, and this means that the embedding function $f$ maps into the open unit ball of a separable Hilbert space. 
Thus $d$ can be chosen to be countable when $\cH$ is separable. 
\end{remark}

\begin{remark}\label{rem:injective}
Suppose that $K^\cH$ is a complete Pick kernel for which $K^\cH(x,y) \neq 0$ for all $x,y \in X$, but for which the kernel functions are not necessarily linearly independent. 
In this case the kernel $K^\cH$ is not irreducible as defined in Definition \ref{def:irreducible}. 
However, the conclusion (as well as the proof) of Theorem \ref{thm:compPickUniversal} still holds, with the exception that now the map $f$ is no longer necessarily injective. 
Interesting examples can be obtained by letting $f \colon X \to \mathbb{B}_d$ be any map and defining $\cH$ to be Hilbert function space on $X$ determined by the pullback kernel $K(x,y) = \frac{1}{1 - \langle f(x), f(y) \rangle}$. 
\end{remark}

\subsection{Generalized interpolation problems} For further results on interpolation in $H^2_d$ see \cite{BB08} and the reference therein; for interpolation in a broader framework including Drury-Arveson space see \cite{BtH10}.

\section{Submodules, quotient modules and quotient algebras}

\subsection{Submodules and quotients}\label{subsec:sub_n_quot}
Let $K$ be  Hilbert space. A subspace $L \subseteq H^2_d \otimes K$ that is invariant under $S \otimes I_K$ is a Hilbert module over $\mb{C}[z]$ in its own right, and is referred to as a {\em submodule} of $H^2_d \otimes K$. Algebraically, this determines a quotient module $H^2_d \otimes K / L$. The quotient module can be normed using the quotient norm, making it a Hilbert module. 

Put $N = L^\perp$. As $N$ is co-invariant for $S\otimes I_K$, it is also a Hilbert module determined by the action of $T = P_N S P_N$. The Hilbert modules $H^2_d \otimes K / L$ and $N$ are unitarily equivalent. 

A natural problem is to determine all submodules and all quotients of $H^2_d \otimes K$. This is a fundamental problem, since, by Theorem \ref{thm:model_pure}, every pure contractive Hilbert module is a quotient of $H^2_d \otimes K$ for some $K$. The case $K = \mb{C}$ is the best understood.

\subsection{Invariant subspaces of $H^2_d$ and ideals}\label{subsec:invNideals}

In \cite[Theorem 2.1]{DavPitts2} it was shown that there is a bijective correspondence between two sided \wot-closed ideals in $\cL_d$ and subspaces of $\cF_d$ which are invariant under $L$ and also under the right shift. The bijective correspondence is the map sending an ideal $J$ to its range space $[J\cF_d] = [J\cdot 1]$. The following two theorems concerning ideals and invariant subspaces in $\cM_d$ follow from this bijective correspondence together with \ref{subsec:invNidealsNC} and \ref{subsec:identification2} (see \cite[Section 2]{DRS15} for some details). 

Denote by $\Lat(\cM_d)$ and $\Id(\cM_d)$ the lattices of the closed invariant subspaces of $\cM_d$ and the \wot-closed ideals in $\cM_d$, respectively.  

\begin{theorem}\label{thm:complete_lattice_iso}
Define a map $\alpha : \Id(\cM_d) \rightarrow \Lat(\cM_d)$ by $\alpha(J) = [J \cdot 1]$. Then $\alpha$ is a complete lattice isomorphism whose inverse $\beta$ is given by 
\[
\beta(K) = \{f \in \cM_d : f \cdot 1 \in K\}. 
\]
\end{theorem}

\begin{theorem}\label{thm:complete_quotient}
If $J$ is \wot-closed ideal in $\cM_d$ with $\alpha(J)^\perp = N$, then $\cM_d/J$ is completely isometrically isomorphic and weak-$*$ homeomorphic to $P_N \cM_d P_N$. 
\end{theorem}

\subsection{Quotients of $H^2_d$ and quotients of $\cM_d$ associated to varieties}\label{subsec:quot_var}

Let $V\subseteq \mb{B}_d$ be a variety (see Section \ref{sec:zero_sets}). The space $\cH_V$ can be considered as a Hilbert function space on $V$, and its multiplier algebra $\Mult(\cH_V)$ is an algebra of functions on $V$. Denote $\cM_V = \{g : V \rightarrow \mb{C} : \exists f \in \cM_d .\, f\big|_V = g\}$. 
Using Theorems \ref{thm:comp_Pick} and \ref{thm:complete_quotient} the following theorem is deduced. 

\begin{theorem}\label{thm:restriction_isomorphism}
Let $V\subseteq \mb{B}_d$ be a variety. Then $\cH_V = [J_V \cdot 1]^\perp$, and 
\[
\Mult(\cH_V) = \cM_V \cong \cM_d / J_V \cong P_{\cH_V} \cM_d P_{\cH_V}
\]
where $\cong$ denotes completely isometric and \wot-continuous isomorphisms, given by
\[
f\big|_V \longleftrightarrow f + J_V \longleftrightarrow P_{\cH_V} M_f P_{\cH_V}.
\]
\end{theorem}

\subsection{The universal complete Pick algebra}

Theorems \ref{thm:compPickUniversal} and \ref{thm:restriction_isomorphism} imply the following result. 

\begin{theorem}\label{thm:universal_Pick_algebra}
Let $\cH$ be a separable, irreducible complete Pick Hilbert function space on a set $X$. Then there is a cardinal $d \in \{1, 2, \ldots, \aleph_0\}$ and a variety $V \subseteq \mb{B}_d$ such that $\Mult(\cH)$ is completely isometrically isomorphic to $\cM_V$. 
The variety $V$ can be chosen to be the smallest variety containing $f(X)$, where $f$ is the as in Theorem \ref{thm:compPickUniversal}. 
\end{theorem}

\begin{remark}
As in Remark \ref{rem:injective}, the conclusion of the above theorem holds also if $K^\cH$ is a complete Pick kernel for which $K^\cH(x,y) \neq 0$ for all $x,y \in X$, but for which the kernel functions are not necessarily linearly independent. 
\end{remark}

\subsection{Maximal ideal spaces} 

Being commutative Banach algebras, the algebras $\cM_V$ are determined to a very a large extent by their maximal ideal space $\mathfrak{M}(\cM_V)$, that is, the space of nonzero complex homomorphism from $\cM_V$ to $\mb{C}$. Elements of $\mathfrak{M}(\cM_V)$ are also referred to as {\em characters}. The full maximal ideal space is too big to be tractable --- it is the space of \wot-continuous characters that is amenable to analysis (see Section \ref{sec:isomorphism}). 

\begin{theorem}[\cite{DavPitts2}, Theorem 3.3; \cite{DRS15}, Proposition 3.2]\label{thm:max_id}
Let $V \subseteq \mb{B}_d$ be a variety. There is a continuous projection $\pi : \mathfrak{M}(\cM_V) \rightarrow \overline{\mb{B}_d}$ given by 
\[
\pi(\rho) = (\rho(S_1), \ldots, \rho(S_d)) \,\,\, , \,\,\, \rho \in \mathfrak{M}(\cM_V). 
\] 
For each $\lambda \in V$ there is a character $\rho_\lambda \in \pi^{-1}(\lambda)$ given by
\[
\rho_\lambda(f) = f(\lambda) = \langle M_f k_\lambda , k_\lambda \rangle / \|k_\lambda\|^2 \, , \,\,\, f \in \cM_V. 
\]
$\rho_\lambda$ is \wot-continuous and every \wot-continuous character arises this way. 

If $d<\infty$, then $\pi(\mathfrak{M}(\cM_V)) \cap \mb{B}_d = V$, $\pi^{-1}(v) = \{\rho_v\}$ for all $v \in V$, and $\pi\big|_{\pi^{-1}(V)}$ is a homeomorphism. 

In the case $\cM_V = \cM_d$ (i.e., the case $V = \mb{B}_d$), $\pi$ is onto $\overline{\mb{B}_d}$, and for every $\lambda \in \partial \mb{B}_d$ the fiber $\pi^{-1}(\lambda)$ is canonically homeomorphic to the fiber over $1$ in $\mathfrak{M}(H^\infty(\mb{D}))$. 
\end{theorem}

\begin{remark}
It was previously believed that the part ``$\pi(\mathfrak{M}(\cM_V)) \cap \mb{B}_d = V$, $\pi^{-1}(v) = \{\rho_v\}$ for all $v \in V$, and $\pi\big|_{\pi^{-1}(V)}$ is a homeomorphism" holds also for $d = \infty$, because in the statement of Theorems 3.2 and 3.3 from \cite{DavPitts2} the condition $d<\infty$ does not appear. 
However, those theorems are false for $d=\infty$, and there exist counterexamples showing that, in general $\pi(\mathfrak{M}(\cM_V)) \cap \mb{B}_d$ may strictly contain $V$, and that $\pi^{-1}(v)$ might be bigger than $\{\rho_v\}$ \cite{DHS13Err}. 
\end{remark}

\subsection{Beurling type theorems} 

In \cite{MT00} McCullough and Trent obtained the following generalization of the classical Beurling-Lax-Halmos theorem \cite{Beurling,Halmos,Lax}. 

\begin{theorem}[\cite{MT00}]\label{thm:BLH}
Let $L$ be a subspace of $H^2_d \otimes K$. The following are equivalent. 
\begin{enumerate}
\item $L$ is invariant under $S \otimes I_K$. 
\item $L$ is invariant under $\cM_d\otimes I_K$. 
\item There is an auxiliary Hilbert space $K_*$ and $\Phi \in \cM_d(K_*,K)$ such that $M_\Phi M_\Phi^*$ is the projection onto $L$ and $L = M_\Phi(H_2 \otimes K_*)$. 
\end{enumerate}
\end{theorem}

Actually, a version of this theorem holds in any complete Pick space \cite{MT00}, thus in particular it holds for quotients of $H^2_d$ of the type $\cH_V$ considered in Section \ref{subsec:quot_var}. In \cite{AK00,AK02,BR02} finite dimensional invariant subspaces of $S^* \otimes I_K$ were studied, and further information was obtained. 

Since $M_\Phi M_\Phi^*$ is a projection, $M_\Phi$ is a partial isometry. A multiplier $\Phi$ for which $M_\Phi$ is a partial isometry is called an {\em inner function}. When $d = \dim K = 1$, it can be shown that (unless $L$ is trivial) $K_*$ can be chosen to be one dimensional and $\Phi$ can be chosen so that $M_\Phi$ is an isometry. In this case $\Phi$ is a scalar valued function on the disc which has absolute value $1$ a.e. on the circle, i.e. an {\em inner function} in the classical sense, and one recovers Beurling's theorem \cite{Beurling} (see \cite[Chapter II]{Garnett}). 

Theorem \ref{thm:BLH} was obtained by Arveson in the case where $\dim K = 1$ \cite[Section 2]{Arv00}. In this case $\Phi \in \cM_d(K_*, \mb{C})$, and this means that there is a sequence $\{\phi_n\}_{n=0}^{\dim K_*}$ such that $P_L = \sum M_{\phi_n} M_{\phi_n}^*$ and $L = \sum M_{\phi_n} H^2_d$ (just put $\phi_n = \Phi_n(1 \otimes e_n)$ where $\{e_n\}$ is an orthonormal basis for $K_*$). Now 
\[
\sum |\phi_n(z)|^2 \|k_z\|^2 =  \langle \sum M_{\phi_n}M_{\phi_n}^* k_z, k_z \rangle \leq \|k_z\|^2, 
\]
so $\sup_{\|z\|< 1} \sum |\phi_n(z)|^2 \leq 1$. In particular, for every $n$, $\phi_n \in H^\infty(\mb{B}_d)$, and therefore the radial limit $\tilde{\phi_n}(w) = \lim_{r\nearrow 1} \phi_n(rw)$ exists for a.e. $w \in \partial \mb{B}_d$ (in fact the limit exists through much larger regions of convergence, see \cite[Theorem 5.6.4]{RudinBall}). Arveson raised the problem of whether or not $\sum_n |\tilde{\phi}_n(w)|^2 = 1$ for a.e. $w \in \partial \mb{B}_d$. This problem was solved by Greene, Richter and Sundberg \cite{GRS02}. 

For every $\lambda \in \mb{B}_d$, let $E_\lambda : H^2_d \otimes K \rightarrow K$ denote the point evaluation functional $E_\lambda f \otimes k = f(\lambda)k$. 

\begin{theorem}[\cite{GRS02}]\label{thm:GRS}
Let $K$ be a separable Hilbert space, let $L$ be an invariant subspace of $H^2_d \otimes K$, and let $K_*$ and $\Phi$ be as in Theorem \ref{thm:BLH}. If $d<\infty$, then for a.e. $w \in \partial \mb{B}_d$, the radial limit $\tilde{\Phi}(w) :=\lim_{r\nearrow 1}\Phi(rw)$ exists and is a partial isometry with 
\be\label{eq:GRS}
\rank\Phi(w) = \sup_{ \lambda \in \mb{B}_d} \dim (E_\lambda L) .
\ee
In particular, if $\dim K = 1$, then for a.e. $w \in \partial \mb{B}_d$
\be\label{eq:inner}
\sum_n |\tilde{\phi}_n(w)|^2 = 1 . 
\ee
\end{theorem}

\subsection{Rigidity phenomena}

Recall that Beurling's theorem says that every submodule of $H^2(\mb{D})$ has the form $\phi H^2(\mb{D})$, where $\phi$ is an inner function. Theorems \ref{thm:BLH} and \ref{thm:GRS} show that a very similar result holds for submodules of $H^2_d$, thereby providing a strong analogy between the submodule theories of $H^2(\mb{D})$ and $H^2_d$. On the other hand, there are some big differences as well. For example, a consequence of Beurling's theorem is that every two submodules of $H^2(\mb{D})$ are unitarily equivalent; the following theorem of Guo, Hu and Xu shows that for $d\geq2$, the situation with submodules of $H^2_d$ is almost the opposite. 

\begin{theorem}[\cite{GHX04}, Section 5; \cite{ChenGuo}, Section 6]
Let $M$ and $N$ be submodules of $H^2_d$, $d \geq 2$. Consider the following conditions. 
\begin{enumerate}
\item $M$ or $N$ is the closure of a polynomial ideal,
\item $M \subseteq N$.
\end{enumerate}
Under the assumption that one of the above conditions holds, if $M$ is unitarily equivalent to $N$, then $M = N$. 
\end{theorem}

\section{The curvature invariant of a contractive Hilbert module}

\subsection{The curvature invariant} 

In \cite{Arv00} Arveson introduced a numerical invariant for contractive Hilbert modules of finite rank, the {\em curvature invariant}. 

Let $d<\infty$, and fix a contractive Hilbert module of finite rank $d$. Recall that this means that there is a $d$ contraction $T$ on $H$ such
that $\rank H := \rank(T)  = \dim \Delta H < \infty$, where $\Delta = \Delta_T = \sqrt{I - \sum T_iT_i^*}$. For $z \in \mb{B}_d$, define the operator valued functions
\[
T(z) = \overline{z}_1 T_1 + \ldots  \overline{z}_d T_d , 
\]
and 
\[
F(z) = \Delta (1 - T(z)^*)^{-1}(1-T(z))^{-1} \Delta. 
\]
For $z$, $F(z)$ is an operator on the finite dimensional space $\Delta H$, hence has a trace. 

\begin{lemma}[\cite{Arv00}, Theorem A]
For almost every $w \in \partial \mb{B}_d$ the limit 
\be\label{eq:curv_def}
\kappa_0(w) = \lim_{r\nearrow 1} (1-r^2)  \operatorname{trace}F(rw) 
\ee
exists and satisfies $0 \leq \kappa_0(w) \leq \rank(H)$. 
\end{lemma}

\begin{definition}
The {\em curvature invariant} of $H$ is defined to be
\be
\kappa(H) = \int_{\partial \mb{B}_d} \kappa_0(w) d\sigma(w), 
\ee
where $\sigma$ is normalized area measure on the sphere. 
\end{definition}

One also writes $\kappa(T)$ for the curvature of $H$. From (\ref{eq:curv_def}), $\kappa(H)$ is a real number between $0$ and $\rank(H)$. 
\begin{theorem}[\cite{Arv00}, Theorem 2.1]
Suppose that $H$ is a pure contractive Hilbert module of finite rank. Then $\kappa(H) = \rank(H)$ if and only if $H$ is unitarliy equivalent to the free Hilbert module $H^2_d \otimes \Delta H$ of rank $\rank(H)$.  
\end{theorem}

The curvature invariant is evidently invariant under unitary equivalence. The above theorem shows that the curvature invariant contains non-trivial operator theoretic information. Arveson used the curvature invariant to prove the $\dim K =1$ case of Theorem \ref{thm:GRS} for the case where the submodule $L$ contains a polynomial \cite[Theorem E]{Arv00}.

\subsection{The Euler characteristic}
The analytically defined curvature invariant is closely tied to an algebraic invariant called the {\em Euler characteristic}. 

If $H$ is a finite rank contractive Hilbert module, then the linear space
\[
M_H = \{p \cdot \xi | \,p \in \mb{C}[z],\,  \xi \in \Delta H\}
\]
is a finitely generated Hilbert module over the ring $\mb{C}[z]$. By Hilbert's syzygy theorem \cite[Corollary 19.8]{Eisenbud} there is a finite free resolution 
\be\label{eq:free}
0 \rightarrow F_n  \rightarrow \ldots \rightarrow F_2 \rightarrow F_1 \rightarrow M_H \rightarrow 0
\ee
where each $F_k$ is the (algebraic) module direct sum of $\beta_k$ copies of $\mb{C}[z]$. 

\begin{definition}
The {\em Euler characteristic} of $H$ is defined by 
\[
\chi(H) = \sum_{k=1}^n (-1)^{k+1}\beta_k . 
\]
\end{definition}

\begin{remark}
One can show that $\chi(H)$ does not depend on the choice of free resolution (\ref{eq:free}). 
\end{remark}

\begin{theorem}
$0 \leq \kappa(H) \leq \chi(H) \leq \rank(H)$. 
\end{theorem}

\subsection{Graded modules and Arveson's ``Gauss-Bonnet" theorem}

A Hilbert module $H$ is said to be {\em graded} if there exists a strongly continuous unitary representation $\Gamma$ of the circle $\mb{T} = \{z \in \mb{C} : |z|=1\}$ on $H$ such that 
\[
\Gamma(\lambda) T_k \Gamma(\lambda)^{-1} = \lambda T_k \,\, k=1, \ldots, d, \, \, \lambda \in \mb{T}. 
\]
Denoting $H_n = \{h \in H : \Gamma(\lambda)h = \lambda^n h\}$, one obtains the decomposition 
\be\label{eq:decomposition}
H = \ldots H_{-1} \oplus H_0 \oplus H_1 \oplus H_2 \oplus \ldots, 
\ee
and every operator $T_k$ is of degree one in the sense that $T_k H_n \subseteq H_{n+1}$. The existence of a representation of $\mb{T}$ should be thought of as a kind of minimal symmetry that $H$ possesses. 

The Hilbert module $H^2_d \otimes K$ is a graded Hilbert module, and the decomposition (\ref{eq:decomposition}) is the natural one induced by the degree of polynomials (there are no negatively indexed summands in the grading of $H^2_d \otimes K$). If $I \triangleleft \mb{C}[z]$ is a homogeneous ideal, then its closure in $H^2_d$ is also a graded contractive Hilbert module. 

\begin{theorem}[\cite{Arv00}, Theorem B]\label{thm:curv_euler}
Let $H$ be a contractive, pure, finite rank and graded Hilbert module. Then 
\be\label{eq:curv_euler}
\kappa(H) = \chi(H). 
\ee
In particular, the curvature invariant is an integer. 
\end{theorem}
In \cite[Theorem 18]{XFang03} the above theorem was generalized to quotients of $H^2_d \otimes \mb{C}^r$ by polynomially generated submodules. 

\subsection{Integrality of the curvature invariant}

Theorem \ref{thm:curv_euler} naturally raised the question whether the curvature invariant is always an integer. Using Theorem \ref{thm:GRS}, Greene, Richter and Sundberg proved that this is so. 

Recall that if $H$ is a pure, contractive Hilbert module then, by Theorem \ref{thm:model_pure}, $H$ can be identified with the quotient of $H^2_d \otimes K$ by a submodule $L$, where $\dim K = \rank(H)$. 

\begin{theorem}[\cite{GRS02}]
Let $H$ be a pure, contractive Hilbert module of finite rank. Then 
\[ 
\kappa(H) = \rank (H) - \sup_{\lambda \in \mb{B}_d} \dim( E_\lambda L).
\]
In particular, $\kappa(H)$ is an integer. 
\end{theorem}

\subsection{The curvature invariant as index}\label{sec:curv_ind}

The following theorem of Gleason, Richter and Sundberg exhibits the curvature invariant as the index of a Fredholm tuple (for more details on spectral theory and Fredholm theory of commuting $d$-tuples, see the chapter on Taylor functional calculus by M\"{u}ller \cite{Mu25} in this reference work).

\begin{theorem}[\cite{GRS05}, Theorem 4.5]\label{thm:GRS_index}
Let $T$ be a pure $d$-contraction of finite rank. Denote the essential spectrum of $T$ by $\sigma_e(T)$. Then $\sigma_e(T) \cap \mb{B}_d$ is contained in an analytic variety, and for all $\lambda \in \mb{B}_d \setminus \sigma_e(T)$ the tuple $T - \lambda = (T_1 - \lambda_1, \ldots, T_d - \lambda_d)$ is Fredholm, and 
\[
\kappa(T) = (-1)^d \operatorname{ind}(T - \lambda). 
\]
\end{theorem}
This theorem implies that the curvature invariant is stable under compact perturbations:

\begin{corollary}
Let $T$ and $T'$ be two pure $d$-contractions of finite rank. If $T_i - T_i'$ is compact for all $i$, then $\kappa(T) = \kappa(T')$. 
\end{corollary}

\subsection{Generalizations} 

The curvature invariant has also been defined for row contractions which are not necessarily commutative \cite{Kribs,Popescu01}, and this notion has been generalized further for completely positive maps \cite{MS03}. An extension to the setting where row-contractivity is replaced by a more general condition $p(T,T^*) \geq 0$ (for some polynomial $p$) is considered in \cite{Englis05}.

\section{Essential normality and the conjectures of Arveson and Douglas}

In \cite{Arv00} it was shown that the curvature invariant (as well as the Euler characteristic) is stable under finite rank perturbations, but left open whether it is invariant under compact perturbations. This problem was taken up in \cite{Arv02} for graded Hilbert modules. By exhibiting the curvature invariant of $H$ as the index of a certain operator --- the {\em Dirac operator} of the $d$-contraction $T$ associated with $H$ --- it was shown that if $H$ is essentially normal then $\kappa(T) = \kappa(T')$ whenever $T'$ is unitarily equivalent to $T$ modulo compacts. 
Based on these considerations Arveson raised the question whether every pure graded contractive Hilbert module of finite rank is essentially normal \cite[Problem 2]{Arv02}. 
The curvature invariant was eventually shown to be equal to the index of a Fredholm tuple --- hence invariant under  compact perturbations --- by Gleason, Richter and Sundberg (see \ref{sec:curv_ind} above), but Arveson's question remained a subject of growing interest for other reasons, see \cite{Dou06b}. 
In fact, following the examination of several classes of examples, Arveson conjectured that every pure graded contractive Hilbert module of finite rank is $p$-essentially normal for all $p>d$. By Theorem \ref{thm:model_pure} this can be reformulated as follows. 

\begin{conjecture}
Let $K$ be a finite dimensional Hilbert space, and let $L \subseteq H^2_d \otimes K$ be a graded submodule. Then $(H^2_d \otimes K)/L$ is $p$-essentially normal for all $p>d$. 
\end{conjecture}

This conjecture attracted a lot of attention \cite{Arv05,Arv07,Dou06a,Dou06b,DS11,DW11,DW12,E10,FX12,GuoWang,GuoWangII,GZ13,Ken13,KenSha12,Sh11}, where the conjecture was proved in particular classes of submodules, but it is still far from being solved . 
In all cases where the conjecture was verified, the following stronger conjecture due to Douglas was also shown to hold.

\begin{conjecture}\label{conj:AD}
Let $K$ be a finite dimensional Hilbert space, and let $L \subseteq H^2_d \otimes K$ be a graded submodule. Then $(H^2_d \otimes K)/L$ is $p$-essentially normal for all $p>\dim (L)$. 
\end{conjecture}
Here $\dim (L)$ is defined as follows. Let $H = H_0 \oplus H_1 \oplus \ldots$ be the grading of $(H^2_d \otimes K)/L$. It is known that there is a polynomial $p_L(x)$ such that $p_L(n) = \dim H_n$ for sufficiently large $n$. Then $\dim L$ is defined to be $\deg p_L(x) + 1$.

\subsection{$K$-homology}

Let $I \triangleleft \mb{C}[z]$ be an ideal of infinite co-dimension. Denote $S^I = P_{I^\perp} S P_{I^\perp}$. Thus, $S^I$ is the $d$-contraction acting on the quotient Hilbert module $H^2_d/ [I]$. 
Define $\cT_I = C^*(S^I,1)$, and let $\cK$ denote the compact operators on $H^2_d \ominus I$. 

\begin{lemma}
$\cK \subseteq \cT_I$. 
\end{lemma}

If $H^2_d / I$ is essentially normal, then by the Lemma one has the following exact sequence
\be
0 \longrightarrow \cK \longrightarrow \cT_I \longrightarrow C(X) \longrightarrow 0.
\ee
It can be shown (see, e.g., \cite[Section 5]{GuoWang}) that if $I$ is homogeneous then $X = V(I) \cap \partial \mb{B}_d$, where $V(I) = \{z \in \mb{C}^d : p(z) = 0 \textrm{ for all } p \in I\}$. Thus one obtains an element in the odd $K$-homology group of the space $V(I) \cap \partial \mb{B}_d$. Douglas raises in \cite[Section 3]{Dou06b} the problem of determining which element of $K_1(V(I) \cap \partial \mb{B}_d)$ this extension gives rise to, and conjectures that it is a certain specific element, the so-called {\em fundamental class} of $V(I) \cap \partial \mb{B}_d$. Guo and Wang have found some evidence for this conjecture, see \cite{GuoWang, GuoWangII}.

\subsection{Some positive results}

In this section some positive  results in the direction of Conjecture \ref{conj:AD} are listed. For simplicity, only the case $L = [I]$, where $I$ is a homogeneous ideal in $\mb{C}[z]$,  is treated. There is not much loss in this; \cite[Corollary 8.4]{Arv07} reduces the problem to the case where the submodule $L \subseteq H^2_d \otimes K$ is generated by linear homogeneous polynomials, and \cite[Section 5]{Sh11} reduces the problem further to the case where $\dim K = 1$ and $L = [I]$ is the closure of a homogeneous ideal $I$ that is generated by scalar valued polynomials of degree $2$ (the second reduction changes the range of $p$'s for which $p$-essential normality holds).  

\begin{theorem}
Let $I$ be a homogeneous ideal in $\mb{C}[z]$, and let $L = [I]$ be its closure in $H^2_d$. If $I$ satisfies any one of the following assumptions, then $H^2_d/L$ is $p$-essentially normal for all $p>\dim I$. 
\begin{enumerate}
\item $I$ is generated by monomials. 
\item $I$ is principal. 
\item $\dim I \leq 1$.
\item $d \leq 3$.  
\item $I$ is the radical ideal corresponding to a union of subspaces. 
\end{enumerate}
\end{theorem}
\begin{proof}
The first item is proved in \cite{Arv05,Dou06a,Sh11} and the last one is proved in \cite{KenSha12}; the rest are proved in \cite{GuoWang}.
\end{proof}

In \cite{KenSha15} a more operator-algebraic approach was initiated to detect the essential normality of quotient modules. 
While progress has regularly been made by verifying that a certain class of ideals satisfies the essential normality conjecture, \cite{KenSha15} supplied evidence for the conjecture by confirming that some of its consequences hold for all homogeneous ideals\footnote{The homogeneous ideal needs to be ``sufficiently non-trivial"; see \cite{KenShaCorr}.}. 
For example, it was shown that the operator algebra generated by the image of $S^I$ in $\cT_I/\cK$ is equal to the closure of the polynomials in the supremum norm on $X=V(I) \cap \partial \mb{B}_d$, and hence that this operator algebra has the C*-envelope $C(X)$ as predicted by the conjecture (whereas essential normality of $S^I$ is equivalent to the stronger requirement that the C*-algebra generated by the image of $S^I$ in the quotient is equal to $C(X)$). 
Further, it was shown that $I$ satisfies the conjecture if and only if the tuple $S^I$ is \emph{hyperrigid} in the sense of Arveson. 
The connection to hyperrigidity was further pursued in \cite{CH18,CT21}, however, new classes of ideals for which this approach leads to a confirmation of the essential normality conjecture remain to be found. 

\subsection{Further positive results and a non-graded counter example} 

Conjecture \ref{conj:AD} is stated for quotients of $H^2_d \otimes K$ by a graded submodule $L$. 
There is reason to believe that the conclusion is true also for the case where $L$ is generated by $K$-valued polynomials, indeed some positive results have been obtained for quasi-homogeneous submodules \cite{BS18, DS11,GuoWangII,GZ13} or in the case of principal submodules \cite{DW11,FX12,FX18}.  
The conjecture was verified for varieties smooth away from the origin, by Engli\v{s} and Eschmeier \cite{EE15} and independently by Douglas, Tang and Yu \cite{DTY16}, and later for more complex configurations \cite{WD18}. 
In recent years progress has been made in several different directions; the reader is referred to the survey \cite{GW_Surv} for a more detailed account than can be given here. 

The positive results notwithstanding, it is important to note that the conjecture cannot be stretched further to arbitrary submodules. In \cite[p. 72]{GRS05} an example is given of a submodule $L \subset H^2_2$ such that $L$ (and therefore also the quotient $H^2_2 / L$) is not essentially normal. Thus, in general, a pure $d$-contraction of finite rank need not be essentially normal.

\section{The isomorphism problem for complete Pick algebras}\label{sec:isomorphism}

Let $V\subseteq \cM_d$ be a variety as in Section \ref{subsec:quot_var}. A natural problem is to study how the structures of $V$ and $\cM_V$ are related, and to try to classify the algebras $\cM_V$ in terms of the varieties. 
Theorem \ref{thm:universal_Pick_algebra} gives this problem additional motivation. 
The circle of problems related to this theme is referred to as \emph{the isomorphism problem}. 
We refer the reader to the survey \cite{SalSh16} for a detailed treatment. 

\subsection{Isometric and completely isometric isomorphism}

Let $\Aut(\mb{B}_d)$ denote the group of automorphisms of the ball, that is, the biholomorphisms of $\mb{B}_d$ onto itself. 
\begin{theorem}[\cite{FX11} Section 2; see also \cite{DavPitts2} Section 4;\cite{DRS11}, Theorem 9.2; \cite{Popescu10b}, Theorems 3.5 and 3.10]\label{thm:aut_unitary}
For every $\phi \in Aut(\mb{B}_d)$ there exists a unitary $U: H^2_d \rightarrow H^2_d$ given by 
\be\label{eq:unitary}
Uh(z) = (1-|\phi^{-1}(0)|^2)^{1/2}k_{\phi^{-1}(0)}(z) h(\phi(z)) .
\ee
Conjugation with $U$ is an automorphism $\Phi$ of $\cM_d$ and implements  composition with $\phi$,
\[
\Phi(f) = UfU^* = f\circ \phi. 
\]
\end{theorem}

The following theorem due to Davidson, Ramsey and Shalit completely solves the classification problem of the algebras $\cM_V$ up to completely isometric isomorphism. 
\begin{theorem}[\cite{DRS15}, Theorems 4.4 and 5.10. See also \cite{AL11}]\label{thm:cis_aut}
The algebras $\cM_V$ and $\cM_W$ are completely isometrically isomorphic if and only if there exists $\phi \in \Aut(\mb{B}_d)$ such that $\phi(W) = V$. Every completely isometric isomorphism $\Phi : \cM_V \rightarrow \cM_W$ arises as $\Phi(f) = f \circ \phi$ where $\phi$ is such an automorphism.
When $d<\infty$ the algebras $\cM_V$ and $\cM_W$ are isometrically isomorphic if and only if they are completely isometrically isomorphic. 
\end{theorem}

\begin{remark}\label{rem:iso_uni}
Given a variety $V \subseteq \mb{B}_d$ one can consider the Hilbert function space $\cH_V$ as in Section \ref{sec:zero_sets}. 
Recall that $\cH_V$ can be considered as a Hilbert function space on $V$ and by Theorem \ref{thm:restriction_isomorphism} $\Mult(\cH_V) = \cM_V$. 
Given another variety $W \subseteq \mb{B}_d$, it is natural to ask in what way are $\cH_V$ and $\cH_W$ related. 
By using (the adjoint of) the unitary $U$ from Theorem \ref{thm:aut_unitary}, one can show that every $\phi \in \Aut(\mb{B}_d)$ such that $\phi(W) = V$ gives rise to an \emph{isometric isomorphism of Hilbert function spaces}, that is, a unitary map $V : \cH_W \to \cH_V$ determined by $V : k_w \mapsto \delta(w) k_{\phi(w)}$ where $\delta$ is a non-vanishing function on $W$. 
In turn, given an isometric isomorphism of Hilbert function spaces $V : \cH_W \to \cH_V$, one can show that $\Phi(M_f) = V^* M_f V = M_{f\circ \phi}$ for all $f \in \Mult(\cH_V)$, thus $f \mapsto f \circ \phi$ is a completely isometric isomorphism of $\cM_V$ onto $\cM_W$. 
Finally, if $\Phi$ is a completely isometric isomorphism from $\cM_V$ onto $\cM_W$, then Theorem \ref{thm:cis_aut} tells us that there is $\phi \in \Aut(\mb{B}_d)$ such that $\phi(W) = V$. 
We conclude that $V$ is an automorphic image of $W$ if and only if $\cH_V$ is isometrically isomorphic to $\cH_W$ as Hilbert function spaces, and this happens if and only if $\cM_V$ and $\cM_W$ are completely isometrically isomorphic. 
\end{remark}

\subsection{Algebraic isomorphism}
A more delicate question is when two algebras $\cM_V$ and $\cM_W$ are algebraically isomorphic (since these algebras are semi-simple, this is equivalent to existence of a bounded isomorphism). 

\begin{theorem}[\cite{DRS15}, Theorem 5.6; \cite{DHS13}]\label{thm:alg_iso}
Suppose that $V,W$ are both subvarieties of $\mb{B}_d$, $d<\infty$, which are comprised of a finite union of irreducible components and a sequence of points. Let $\Phi: \cM_V \rightarrow \cM_W$ be an isomorphism. Then there exist holomorphic maps $\phi,\psi: \mb{B}_d \rightarrow \mb{C}^d$ such that 
\begin{enumerate}
\item $\phi(W) = V$ and $\psi(V) = W$,
\item $\phi \circ \psi\big|_V = {\bf id}\big|_V$ and $\psi \circ  \phi|_W = {\bf id}\big|_W$,
\item $\Phi(f) = f \circ \phi$ for all $f \in \cM_V$ and $\Phi^{-1}(f) = f \circ \psi$ for all $f \in \cM_W$. 
\item The restrictions of $\psi$ to $V$ and of $\phi$ to $W$ are bi-Lipschitz maps with respect to the pseudohyperbolic metric. 
\end{enumerate} 
\end{theorem}

The following corollary follows from the above theorem and Theorem \ref{thm:aut_unitary}. 

\begin{corollary}
Every algebraic automorphism of $\cM_d$ is given by composition with an automorphism of the ball, hence is completely isometric and unitarily implemented. 
\end{corollary}

Two varieties $V,W$ for which there are maps $\phi,\psi$ as in Theorem \ref{thm:alg_iso} are said to be {\em biholomorphic}, and the maps $\psi$ and $\phi$ are said to be {\em biholomorphisms} from $V$ to $W$ or vice versa. In light of the above result, it is natural to ask: {\em given a biholomorphism $\phi : W \rightarrow V$, does it induce an algebraic isomorphism $\cM_V \rightarrow \cM_W$?} If $f \in \cM_V$ and $\phi \in W \rightarrow V$ is holomorphic then evidently $f \circ \phi \in H^\infty(W)$; the crux of the matter is whether or not it is a multiplier. 
The answer is negative in general \cite{DHS13,DRS15}. 
The first positive result in this direction was obtained by Alpay, Putinar and Vinnikov.

\begin{theorem}[\cite{APV03}, Proposition 2.1]\label{thm:APV}
Let $d<\infty$, and let $\phi : \overline{\mb{D}} \rightarrow \overline{\mb{B}}_d$ be a proper injective $C^2$ function that is a biholomorphism of $\mb{D}$ onto $V = \phi(\mb{D}) \subset \mb{B}_d$. Then the map 
\[\Phi : \cM_V \rightarrow H^\infty(\mb{D})\,\, , \,\, \Phi(f) = f \circ \phi, \]
is a bounded isomorphism. In particular, $\cM_V = H^\infty(V)$. 
\end{theorem}

Combining this theorem with \ref{thm:restriction_isomorphism} one obtains the following variant of a hard-analytic extension theorem of Henkin \cite{Henkin}. 
\begin{corollary}[\cite{APV03}, Theorem 2.2]\label{cor:APV_ext}
Let $V$ be as in Theorem \ref{thm:APV}. Then there is a constant $C$ such that for any bounded analytic function $f$ on $V$ there is a multiplier $F \in \cM_d$ (in particular, $F \in H^\infty(\mb{B}_d)$) such that $f = F\big|_V$ and 
\[
\|F\|_\infty \leq \|F\|_{\cM_d} \leq C \|f\|_\infty .
\]
\end{corollary}

Theorem \ref{thm:APV} and Corollary \ref{cor:APV_ext} were extended to the case where $\mb{D}$ is replaced by a planar domain by Arcozzi, Rochberg and Sawyer \cite[Section 2.3.6]{ARS08} or a finite Riemann surface  by Kerr, M\raise.45ex\hbox{c}Carthy and Shalit \cite[Section 4]{KMS13}, and in these extensions $\phi$ was allowed to be a finitely ramified holomap. In the three papers mentioned an additional assumption about $V$ meeting the boundary of $\mb{B}_d$ transversally were imposed, but this assumption was later shown to be satisfied automatically \cite{DHS13}. 
The case of a bihlomorphic embedding of a disc into $\mb{B}_\infty$ was studied in \cite{DHS13,DRS15}, and in particular it was shown that a continuum of non-isomorphic algebras can arise this way.

In \cite[Theorem 5.1]{DHS13} it is shown that for an embedding map $\phi : \overline{\mb{D}} \to \overline{\mb{B}}_d$ that satisfies all the conditions in Theorem \ref{thm:APV} except that $\phi(-1) = \phi(1)$, the corresponding multiplier algebra $\cM_V$ is not equal to $H^\infty(V)$ hence not isomorphic to $H^\infty(\mb{D})$, even though $V$ is biholomorphic to $\mb{D}$. 
The question arises what kind of multiplier algebras arise as $\cM_V$ for such embedded discs with self intersections on the boundary. 
Does the isomorphism class depend only on the ``intersection pattern"? 
In \cite{MironovPaper} Mironov considered such discs and found that, although the intersection pattern is an invariant of the multiplier algebras, it is not a complete invariant: there do exist uncountably many embedded discs with the same intersection pattern (for example, $\phi(\exp(\pm\frac{2\pi i}{3})) = \phi(1)$) that give rise to mutually non-isomorphic algebras. 

\begin{remark}\label{rem:iso_via_nc}
From the proof of Theorem \ref{thm:alg_iso} (see \cite[Theorem 5.6]{DRS15}) it follows that $\phi$ and $\psi$ are more than just bounded holomorphic maps --- they are vector valued multipliers. 
In this situation we say that $V$ and $W$ are {\em multiplier biholomorphic}. 
Thus, in the setting of the theorem, we obtain that if $\cM_V$ and $\cM_W$ are isomorphic then $V$ and $W$ are multiplier biholomorphic. 
The question arises whether multiplier biholomorphism is also a sufficient condition for $\cM_V$ and $\cM_W$ to be isomorphic. 
This turns out to be false in general (see \cite[Example 6.6]{DHS13}), but it is unknown whether it holds under the assumptions of Theorem \ref{thm:alg_iso}. 
In fact, multiplier biholomorphism is not an equivalence relation (see \cite[Remark 6.7]{DHS13}). 
This raises the problem of describing an equivalence relation on varieties that encodes when the corresponding multiplier algebras are isomorphic. 
Such a relation was found in \cite{SSS20} by considering the isomorphism problem within a noncommutative framework: it was shown that $\cM_V$ and $\cM_W$ are isomorphic if and only if there exists a certain kind of noncommutative holomorphic map between appropriate quantizations of $V$ and $W$ (see Remark 6.9, loc. cit.). 
\end{remark}

\subsection{Homogeneous varieties}\label{subsec:homogeneous_var}
\begin{definition}
A variety $V \subseteq \mb{B}_d$ is said to be {\em homogeneous} if for all $v \in V$ and all $\lambda \in \mb{D}$ it holds that $\lambda v \in V$. 
\end{definition}

A variety is homogeneous if and only if it is the zero set of a homogeneous ideal. There are some satisfactory results for the isomorphism problem in the case where $V$ and $W$ are homogeneous varieties. The following theorem was obtained by Davidson, Ramsey and Shalit in \cite{DRS15} under some technical assumptions, which were removed by Hartz in \cite{Har12}. 

\begin{theorem}[\cite{DRS11}, Theorems 8.5 and 11.7;  \cite{Har12}, Theorem 5.9]\label{thm:iso_homo}
Let $V$ and $W$ be two homogeneous varieties in $\mb{B}_d$, with $d<\infty$. Then $\cM_V$ and $\cM_W$ are isomorphic if and only if there exist linear maps $A,B : \mb{C}^d \rightarrow \mb{C}^d$ such that $A(W) = V$, $B(V) = W$, $AB\big|_V = {\bf id}\big|_V$ and $BA\big|_W = {\bf id}\big|_W$. 
\end{theorem}

\begin{remark}\label{rem:iso_linear}
Let $V$ and $W$ be varieties. 
The Hilbert function spaces $\cH_V$ and $\cH_W$ are said to be \emph{isomorphic as Hilbert function spaces} if there is a bijective map $\phi : W \to V$ and a non-vanishing function $\delta$ on $W$ such that $k_w \mapsto \delta(w) k_{\phi(w)}$ extends to a bounded invertible linear map $T : \cH_W \to \cH_V$ (cf. Remark \ref{rem:iso_uni}). 
In all cases where some kind of converse to Theorem \ref{thm:alg_iso} was shown to hold (that is, when it can be shown that the existence of a biholomorphism $\phi : W \to V$ gives rise to an isomorphism $\cM_V \to \cM_W$) the proof goes through first constructing an isomorphism of Hilbert function spaces $\cH_W \to \cH_V$ and then obtaining the isomorphism $\cM_V \to \cM_W$ via conjugation. 
Thus, for example, in Theorem \ref{thm:iso_homo} one can add that the two equivalent conditions are equivalent to $\cH_V$ and $\cH_W$ being isomorphic as Hilbert function spaces. 
\end{remark}

\subsection{The isomorphism problem for norm closed algebras of multipliers} 

The algebras $\cA_V:=\cA_d \big|_V = \{f\big|_V : f \in \cA_d\}$ and $\cA_d / I$ (where $I$ is a closed ideal in $\cA_d$) have also been considered, but in this setting less is known. The case of homogeneous varieties is completely settled by results of \cite{DRS11} and \cite{Har12}. Some partial results are contained in \cite{DHS13,DRS15,KMS13}. 

\begin{theorem}
Let $V$ and $W$ be two homogeneous varieties in $\mb{B}_d$. $\cA_V$ and $\cA_W$ are completely isometrically isomorphic if and only if there is a unitary $U$ such that $U(W) = V$. If $d<\infty$, then $\cM_V$ and $\cM_W$ are isomorphic if and only if there exist linear maps $A,B : \mb{C}^d \rightarrow \mb{C}^d$ such that $A(W) = V$, $B(V) = W$, $AB\big|_V = {\bf id}\big|_V$ and $BA\big|_W = {\bf id}\big|_W$. 
\end{theorem}

\subsection{The quantitative isomorphism problem} 
Let $V$ and $W$ be two finite subsets of $\mb{B}_d$ with the same number of points. 
It is clear in this case that $V$ is biholomorphic to $W$, that $\cH_V$ is isomorphic as a Hilbert function space to $\cH_W$, and that $\cM_V$ is isomorphic to $\cM_W$. 
However, by \ref{thm:cis_aut}, $\cM_V$ is (completely) isometrically isomorphic to $\cM_W$ if and only if there is a $\phi \in \Aut(\mb{B}_d)$ such that $\phi(W) = V$. 
One may ask what happens if $V$ is not quite, but \emph{almost} the image of $W$ under an automorphism; can we then say that $\cM_V$ and $\cM_W$ are \emph{almost} unitarily equivalent? 

To make sense of this question Ofek, Pandey and Shalit introduced in \cite{OPS21} distance functions that measure how far Hilbert function spaces or their multiplier algebras are from one another. 
If $\cH_1$ and $\cH_2$ are Hilbert function spaces, the \emph{reproducing kernel Banach-Mazur distance} is defined to be 
\[
\rho_{RK}(\cH_1, \cH_2) = \log \left( \delta_{RK}(\cH_1, \cH_2) \right)
\]
where 
\[
\delta_{RK}(\cH_1, \cH_2) = \inf\left \{ \|T\| \|T^{-1}\| : T : \cH_1 \to \cH_2 \textrm{ is a Hilbert function space isomorphism}\right\} .
\]
(As usual, we interpret the infimum of the empty set to be $\infty$). 
Two spaces $\cH_1$ and $\cH_2$ are isometrically isomorphic as Hilbert function spaces if and only if $\rho_{RK}(\cH_1,\cH_2) = 0$. 
If $\rho_{RK}(\cH_1,\cH_2)$ is positive but finite then $\cH_1$ and $\cH_2$ are isomorphic as Hilbert function spaces and $\rho_{RK}(\cH_1,\cH_2)$ is a measure of how far this isomorphism is from being isometric. 
Similarly, letting $\cM_i = \Mult(\cH_i)$ for $i=1,2$,  the \emph{multiplier Banach-Mazur distance} is defined to be 
\[
\rho_{M}(\cM_1, \cM_2) = \log \left( \delta_{M}(\cM_1, \cM_2) \right)
\]
where 
\[
\delta_{M}(\cM_1, \cM_2) = \inf\left \{ \|\Phi\|_{cb} \|\Phi^{-1}\|_{cb} : \Phi : \cM_1 \to \cM_2 \textrm{ is an isomorphism}\right\} .
\]
The main result of \cite{OPS21} is a quantitative version of the isomorphism results mentioned earlier in this section. 
Roughly, \cite[Theorem 5.4]{OPS21} says that for two finite sets $V,W \subset \mb{B}_d$ the following are equivalent: 
\begin{enumerate}
\item Some image of $V$ under an automorphism of $\Aut(\mb{B}_d)$ is close to $W$ in the Hausdorff metric, 
\item $\cH_V$ and $\cH_W$ are close in the reproducing kernel Banach-Mazur distance, 
\item $\cM_V$ and $\cM_W$ are close in the multiplier Banach-Mazur distance. 
\end{enumerate}

Examples show that these results cannot be extended verbatim to arbitrary varieties $V$ and $W$. 
However, recent work by Watted successfully treated homogeneous varieties \cite{Watted}.

\section{Some harmonic analysis in $H^2_d$}

The $d=1$ instance of $\cM_d$, which is simply the algebra $H^\infty(\mb{D})$ of bounded analytic functions on the disc, has been the arena of a long-standing, beautiful and fruitful interaction between function theory and functional analysis \cite{Garnett}. Among the most profound results in this setting are Carleson's interpolation and corona theorems \cite{Carleson58,Carleson62}, and a technical tool which Carleson introduced --- now called {\em Carleson measures} --- has been of lasting significance. This section surveys some recent results in the case $1<d<\infty$ regarding these three topics: interpolating sequences, Carleson measures, and the corona theorem. For a survey with emphasis on the harmonic analysis side of $H^2_d$ see \cite{Arcozzi13}.

\subsection{Carleson measures for $H^2_d$}

Recall the Besov-Sobolev spaces $B^\sigma_p(\mb{B}_d)$ from Section \ref{sec:BS}. 
\begin{definition}
A positive measure $\mu$ on $\mb{B}_d$ is said to be a {\em Carleson measure for $B^\sigma_p(\mb{B}_d)$} if there exists a constant $C$ such that for all $f \in B^\sigma_p(\mb{B}_d)$, 
\be\label{eq:Carleson}
\|f\|_{L^p(\mu)} \leq C \|f\|_{B^\sigma_p(\mb{B}_d)} \,.
\ee
The space of all Carleson measures on $B^\sigma_p(\mb{B}_d)$ is denoted $CM(B^\sigma_p(\mb{B}_d))$. The infimum of $C$'s appearing in the right hand side of (\ref{eq:Carleson}) is the {\em Carleson measure norm} of $\mu$, denoted $\|\mu\|_{CM(B^\sigma_p(\mb{B}_d))}$.
\end{definition}
An understanding of Carleson measures has turned out to be a key element in the analysis of the spaces $B^\sigma_p(\mb{B}_d)$. The focus of this survey is $H^2_d = B^{1/2}_2(\mb{B}_d)$, but in the literature one often finds a treatment for an entire range of $p$'s or $\sigma$'s. A characterization of the Carleson measures of $B^\sigma_p(\mb{B}_d)$ for ranges of $p$ and $\sigma$ that include $p=2, \sigma = 1/2$ was obtained in \cite{ARS08},\cite{Tch08} and \cite{VW12}. The reader is referred to these papers for additional details. 

\begin{remark}
Consider the scale of spaces $B^\sigma_2(\mb{B}_d)$. 
It is interesting that the value $\sigma = 1/2$ seems to play a critical role in some approaches, while in others it does not. For example, the characterization of Carleson measures given in \cite[Theorem 23]{ARS08} holds for $0 \leq \sigma < 1/2$, the case $\sigma = 1/2$ is handled differently. 
On the other hand, the methods of Tchoundja \cite{Tch08} work for the range $\sigma \in (0,1/2]$, but not for $\sigma > 1/2$. However, using different techniques, Volberg and Wick give in \cite[Theorem 2]{VW12} a characterization of Carleson measures for $B^\sigma_2(\mb{B}_d)$ for all $\sigma > 0$. 
\end{remark}

\subsection{Characterization of multipliers}

The strict containment (\ref{eq:strict}) and the incomparability of the multiplier norm and the sup norm lead to the problem of characterizing multipliers in function theoretic terms. One of the applications of Carleson measures is such a characterization. A geometric characterization of Carleson measures such as the one given in \cite[Theorem 34]{ARS08} then enables, in principle, to determine in intrinsic terms whether a function is multiplier. 

\begin{theorem}[Theorem 2, \cite{ARS08}; Theorem 3.7, \cite{OF00}]
  \label{thm:mult_char}
Let $d<\infty$, let $f$ be a bounded analytic function on $\mb{B}_d$, and fix $m > (d-1)/2$. Then $f \in \cM_d$ if and only if the measure
\[
d\mu_{f,k} = \sum_{|\alpha|=m} \left|\frac{\partial^\alpha f}{\partial z^\alpha} (z)\right|^2 (1 - |z|^2)^{2m - d} d\lambda(z)
\]
is a Carleson measure for $H^2_d$. In this case one has the following equivalence of norms
\be\label{eq:Mult_norm_carl}
\|f\|_{\cM_d} \sim \|f\|_{\infty} + \|\mu_{f,m}\|_{CM(H^2_d)}.
\ee
\end{theorem}

The equivalence of norms (\ref{eq:Mult_norm_carl}) together with Theorem \ref{thm:DruryVNineq} (Drury's von Neumann inequality) gives a version of von Neumann's inequality for $d$-contractions that avoids mention of the $d$-shift, but is valid only up to equivalence of norms.

\begin{corollary}
Let $T$ be a $d$-contraction ($d<\infty$), and fix $m>(d-1)/2$. Then there exists a constant $C$ such that for every polynomial $p \in \mb{C}[z]$, 
\[
\|p(T)\| \leq C \left(\sup_{z\in \mb{B}_d} |p(z)| + \|\mu_{p,m}\|_{CM(H^2_d)} \right) .
\]
\end{corollary}
For an explicit description of the right hand side see \cite[Theorem 4]{ARS08}. 
A function theoretic version of von Neumann's inequality for $d$-contractions resulting from the above corollary was also noted by Chen \cite[Corollary 3]{Chen03}. 

\subsection{Interpolating sequences} 
\label{ss:is}

\begin{definition}
Let $Z = \{z_n\}_{n=1}^\infty$ be a sequence of points in $\mb{B}_d$. $Z$ is said to be an {\em interpolating sequence for $\cM_d$} if the map
\[
\cM_d \ni f \mapsto (f(z_n))_{n=1}^\infty \in \ell^\infty
\] 
maps $\cM_d$ onto $\ell^\infty$. 
\end{definition}

There is also a notion of interpolating sequence for $H^2_d$, but since $H^2_d$ contains unbounded functions, the definition has to be modified. 

\begin{definition}
Let $Z = \{z_n\}_{n=1}^\infty$ be a sequence of points in $\mb{B}_d$. Define a sequence $\{w_n\}_{n=1}^\infty$ of weights by $w_n = (1-\|z_n\|)^{1/2}$. $Z$ is said to be an {\em interpolating sequence for $H^2_d$} if the map
\[
H^2_d \ni h \mapsto (w_n h(z_n))_{n=1}^\infty
\] 
maps $H^2_d$ into and onto $\ell^2$. 
\end{definition} 

\begin{remark}
There exists a similar notion of interpolating sequence for an arbitrary Hilbert function space $\cH$ with kernel $K^\cH$, where the weights are given by $w_n = \|K_{z_n}^\cH\|^{-1}$. 
\end{remark}

\begin{theorem}
  \label{thm:ms_is}
Let $Z = \{z_n\}_{n=1}^\infty$ be a sequence of points in $\mb{B}_d$ ($1\leq d \leq \infty$). Then $Z$ is an interpolating sequence for $\cM_d$ if and only if $Z$ is an interpolating sequence for $H^2_d$.
\end{theorem}

\begin{proof}
The theorem, due to Marshall and Sundberg, holds for arbitrary Hilbert function spaces with the Pick property. See \cite[Theorem 9.19]{AM02} or \cite[Corollary 7]{MarSun94} for a proof. 
\end{proof}

The thrust of the above theorem is that it allows to approach the problem of understanding interpolating sequences for the algebra $\cM_d$ by understanding the interpolating sequences for the (presumably more tractable) Hilbert space $H^2_d$. A characterization of interpolating sequences in $B^\sigma_p(\mb{B}_d)$ and $\Mult(B^\sigma_p(\mb{B}_d))$ for $\sigma \in [0,1/2)$ was found by Arcozzi, Rochberg and Sawyer \cite[Section 2.3.2]{ARS08} based on work of B\o e \cite{Boe}. For $\sigma = 1/2$ (i.e., Drury-Arveson space) see Theorem \ref{thm:interpolating_DA}.

\subsection{The corona theorem for multipliers of $H^2_d$}

Lennart Carleson's corona theorem \cite{Carleson62} for $H^\infty(\mb{D})$ is the following. 
\begin{theorem}[Carleson's corona theorem, \cite{Carleson62}]\label{thm:CCT}
Let $\delta>0$, and suppose that $f_1, \ldots, f_N \in H^\infty(\mb{D})$ satisfy 
\[
\sum_{i=1}^N |f_i(z)|^2 \geq \delta \,\, , \,\, \textrm{ for all } z \in \mb{D}. 
\]
Then there exist $g_1, \ldots, g_N \in H^\infty(\mb{D})$ such that
\[
\sum_{i=1}^N g_i f_i = 1. 
\]
\end{theorem}
An equivalent way of phrasing this theorem is that the point evaluation functionals 
\[
H^\infty(\mb{D}) \ni f \mapsto f(\lambda)
\]
are weak-$*$ dense in the maximal ideal space of $H^\infty(\mb{D})$, in other words $\mb{D}$ is dense in $\mathfrak{M}(H^\infty(\mb{D}))$ --- hence the metaphor {\em corona}. In fact, Carleson proved a stronger result, which included bounds on the norm of $g_1, \ldots, g_N$ in terms of $\delta$ and the norms $f_1, \ldots, f_N$. 

Over the years a lot of effort was put into proving an analogue of this celebrated theorem in several variables, and some results were obtained \cite{AC00,KL95,Lin94,TW05,TW09,Varopoulos77}; see also the recent survey \cite{Dou12}. However, the most natural several variables analogues of Theorem \ref{thm:CCT}, which are precisely the same statement in the theorem but with the disc $\mb{D}$ replaced by either the unit ball $\mb{B}_d$ or the polydisc $\mb{D}^d$, remain to this day out of reach. 

The growing role that the Drury-Arveson space played in multivariable operator theory suggests that the ``correct" multivariable analogue of $H^\infty(\mb{D})$ is not $H^\infty(\mb{B}_d)$ or $H^\infty(\mb{D}^d)$, but $\cM_d$. Indeed, using a mixture of novel harmonic analytic techniques with available operator theoretic machinery, Costea, Sawyer and Wick \cite{CSW12} proved a corona theorem for $\cM_d$. Their main technical result is the following result that they call the {\em baby corona theorem}. 

\begin{theorem}[Baby corona theorem. Theorem 2, \cite{CSW12}]\label{thm:baby_corona}
Fix $\delta>0$ and $d<\infty$. Let $f_1, \ldots f_N \in \cM_d$ satisfy
\be\label{eq:bound_cor2}
\sum_{n=1}^N |f_n(z)|^2 \geq \delta \,\, , \,\,\ \textrm{ for all } z \in \mb{B}_d \,.
\ee
Then for all $h \in H^2_d$, there exist $g_1, \ldots, g_N \in H^2_d$ 
such that 
\be\label{eq:baby_corona}
\sum_{n=1}^N f_n g_n = h . 
\ee
Moreover, there is a constant $C = C(d,\delta)$ such that whenever $f_1, \ldots, f_N$ satisfy 
\be\label{eq:bound_cor1}
\sum_{n=1}^N M_{f_n}^* M_{f_n} \leq I
\ee
then $g_1, \ldots, g_N$ can be chosen to satisfy
\be\label{eq:bound_cor3}
\sum_{n=1}^N \|g_n\|^2 \leq C \|h\|^2 .
\ee
\end{theorem}

\begin{remark}\label{rem:sigma}
Note that $C$ does not depend on $N$. In fact, the theorem also holds for $N= \infty$, and also in a semi-infinite matricial setting. Moreover, the theorem holds with $B^\sigma_p(\mb{B}_d)$ replacing $H^2_d$ and $\Mult(B^\sigma_p(\mb{B}_d))$ replacing $\cM_d$ for all $1<p<\infty$ and $\sigma \geq 0$ (see \cite{CSW12}). 
\end{remark}

To see why Theorem \ref{thm:baby_corona} is called the ``baby" corona theorem note the following. A full (or ``grown-up") corona theorem for $\cM_d$ would be that given $f_1, \ldots, f_N \in \cM_d$ satisfying (\ref{eq:bound_cor2}), there are $\tilde{g}_1, \ldots, \tilde{g}_N$ in $\cM_d$ for which $\sum f_n \tilde{g}_n = 1$ (implying that $\mb{B}_d$ is dense in $\mathfrak{M}(\cM_d)$). In the baby corona theorem (Theorem \ref{thm:baby_corona}) $g_1, \ldots, g_N$ are only required to be in the (much larger) space $H^2_d$. Clearly the full corona theorem implies the baby theorem, because if $\tilde{g}_1, \ldots, \tilde{g}_N$ are as in the full corona theorem, then given $h$ the functions $g_n : = \tilde{g}_n h \in H^2_d$ clearly satisfy (\ref{eq:baby_corona}). 

Stated differently, the assertion of Theorem \ref{thm:baby_corona} is that, given (\ref{eq:bound_cor2}), the row operator $T:=[M_{f_1} \, M_{f_2} \, \cdots \, M_{f_N}] : H^2_d \otimes \mb{C}^N \rightarrow H^2_d$ is surjective, equivalently, it says that 
\be\label{eq:Tcond}
\sum_{n=1}^N M_{f_n}M_{f_n}^* \geq \epsilon^2 I
\ee
for some $\epsilon > 0$. On the other hand, the full corona theorem asserts that under the same hypothesis the tuple $(M_{f_1}, \ldots, M_{f_N})$ is an invertible tuple in the Banach algebra $\cM_d$. 

In \cite[Section 6]{Arv75} Arveson showed, in the setting of $H^\infty(\mb{D})$, that (\ref{eq:Tcond}) implies a full corona theorem. This was extended to several variables by Ball, Trent and Vinnikov, using their commutant lifting theorem (Theorem \ref{thm:BTV}). 
 
\begin{theorem}[Toeplitz corona theorem. p. 119, \cite{BTV01}]\label{thm:TCT}
Suppose $f_1, \ldots f_N \in \cM_d$ satisfy (\ref{eq:Tcond}). Then there are $g_1, \ldots, g_N \in \cM_d$ such that 
\[
\sum_{n=1}^N f_n g_n = 1. 
\]
Moreover, $g_1, \ldots, g_N$ can be chosen such that $\sum \|M_{g_n}\|^2 \leq \epsilon^{-2}$. 
\end{theorem}

\begin{remark}
The converse is immediate. 
\end{remark}
\begin{remark}\label{rem:Pick}
Both the theorem and its converse hold for $d = \infty$. In fact, the theorem and its converse hold for any multiplier algebra of a complete Pick space. 
\end{remark}

As a consequence of Theorems \ref{thm:baby_corona} and \ref{thm:TCT}, one has the full corona theorem for $\cM_d$. 

\begin{theorem}[Corona theorem for $\cM_d$. Theorem 1, \cite{CSW12}]\label{thm:CSW_corona}
Let $\delta>0$, and suppose that $f_1, \ldots, f_N \in \cM_d$ satisfy 
\[
\sum_{i=1}^N |f_i(z)|^2 \geq \delta \,\, , \,\, \textrm{ for all } z \in \mb{B}_d. 
\]
Then there exist $g_1, \ldots, g_N \in \cM_d$ such that
\[
\sum_{i=1}^N g_i f_i = 1. 
\]
\end{theorem}

\begin{remark}
Since for $\sigma \in [0,1/2]$ the space $B^\sigma_2(\mb{B}_d)$ is a complete Pick space, the above theorem also holds for the algebra $\Mult(B^\sigma_2(\mb{B}_d))$, $\sigma \in [0,1/2]$ (see Remarks \ref{rem:sigma} and \ref{rem:Pick}). 
\end{remark}

\appendix
\section{Recent developments in the Drury--Arveson space}

The theory of the Drury--Arveson space has progressed since the first version
of this article was written. In this appendix, some of these new developments are surveyed
    (in some cases, earlier results are also mentioned for context).
Because of the large number of new articles featuring the Drury--Arveson space, not everything could be covered here.
The author of the appendix offers his apologies to anyone whose work was overlooked.

We continue to write $H^2_d$ for the Drury-Arveson space, $\cM_d$ for the multiplier algebra of $H^2_d$, and $\cA_d$ for the norm closure of the polynomials in $\cM_d$.
Unless otherwise stated, we will assume throughout that $d < \infty$.
However, it will be mentioned that a number of results in fact hold for all normalized complete Pick spaces,
and this in particular includes $H^2_d$ for $d = \infty$.

\subsection{Connection to noncommutative function theory}

As explained in Section \ref{sec:sym}, $H^2_d$ can be identified with the symmetric Fock space over $\mathbb{C}^d$,
which in turn is a subspace of the full Fock space over $\mathbb{C}^d$.
It has become clear that there is also a function space picture of the full Fock space, involving noncommutative holomorphic functions.
This picture makes the procedure of ``compressing theorems'' from full Fock space to $H^2_d$, which was also mentioned  in Section \ref{sec:sym}, especially transparent.
This noncommutative approach to $H^2_d$ was for instance used to obtain an inner/outer factorization
in $H^2_d$, which will be discussed below (see also Remark \ref{rem:iso_via_nc} for another application of the noncommutative approach).

Let $F_d^+$ be the free monoid on $d$ generators, meaning that $F_d^+$ consists of all words of finite length
in the letters $1,\ldots,d$, along with the empty word. Let $x = (x_1,\ldots,x_d)$ be non-commuting variables.
Given $w = w_1\ldots w_r \in F_d^+$, we form the noncommutative monomial $x^w = x_{w_1} x_{w_2} \ldots x_{w_r}$.
The \emph{noncommutative Hardy space} $H^2_{nc}$ in $d$ variables is the space of all formal noncommutative power series
$F = \sum_{w \in F_d^+} a_w x^w$ satisfying
\begin{equation*}
  \|F\|^2 = \sum_{w \in F_d^+} |a_w|^2 < \infty.
\end{equation*}
Thus, the noncommutative monomials $(x^w)_{w \in F_d^+}$ form an orthonormal
basis of $F_d^+$.

Denoting the standard basis of $\mathbb{C}^d$ with $e_1,\ldots,e_d$,
the space $H^2_{nc}$ can be identified with the full Fock space over $\mathbb{C}^d$
by identifying a monomial $x^w$, where $w = w_1 \ldots w_r$, with the elementary
tensor $e_{w_1} \otimes \ldots \otimes e_{w_r}$. The crucial point is that in this identification,
the noncommutative $d$-shift $L = (L_1,\ldots,L_d)$ on the full Fock space corresponds
to the tuple of left multiplication by the variables $x_1,\ldots,x_d$.

If $d=1$, then $H^2_{nc}$ is the classical Hardy space $H^2(\mathbb{D})$, and the formal power series
in $H^2(\mathbb{D})$ in fact converge on $\mathbb{D}$ and define bona fide holomorphic functions there.
If $d > 1$, the key idea is to evaluate the noncommutative formal power series
not only on tuples of scalars, but on certain tuples of matrices.
Given a (not neccessarily commuting) tuple
$X = (X_1,\ldots,X_d) \in M_n(\mathbb{C})^d$
and $w = w_1 \ldots w_r \in F^+_d$, we let $X^w = X_{w_1} \ldots X_{w_r}$.
We also denote the row norm of $X$ by
\begin{equation*}
  \|X\|_{row} = \Big\| \sum_{j=1}^d X_j X_j^* \Big\|^{1/2}
\end{equation*}
and let
\begin{equation*}
  \mathbb{B}_d^{n \times n} = \{X \in M_n(\mathbb{C})^d: \|X\|_{row} < 1 \}.
\end{equation*}
A simple estimate using the Cauchy--Schwarz inequality shows that if $X \in \mathbb{B}_d^{n \times n}$, then
for each $F = \sum_{w \in F^+_d} a_w x^w \in H^2_{nc}$, the sum
\begin{equation*}
  F(X) = \sum_{w \in F^+_d} a_w X^w
\end{equation*}
converges in $M_n(\mathbb{C})$.
Thus, one can think of elements of $H^2_{nc}$ as functions on the disjoint union $\mathbb{B}_d^{nc} = \bigcup_{n=1}^\infty \mathbb{B}_d^{n \times n}$,
which is called the noncommutative row ball.

Now, passing from an element $F \in H^2_{nc}$ to an element of $H^2_d$ can be achieved
by simply restricting $F$ to $\mathbb{B}_d = \mathbb{B}_d^{1 \times 1}$, which is usually referred to as ``level $1$''
of the noncommutative set $\mathbb{B}_d^{nc}$.
This is the function space picture corresponding to the orthogonal projection from full Fock space to symmetric Fock space.
More explictly, the restriction map
\begin{equation}
  \label{eqn:Fock_restriction}
   H^2_{nc} \to H^2_d, \quad F \mapsto F \big|_{\mathbb{B}_d},
\end{equation}
is a co-isometry. See also \cite[Section 2.7]{Har22} for more details.

In the function space picture, the algebra $\mathcal{L}_d$ also has a very nice description:
it corresponds to
\begin{equation*}
  H^\infty_{nc} = \{ \Phi \in H^2_{nc}: \Phi \text{ is bounded on } \mathbb{B}_d^{nc} \};
\end{equation*}
see \cite[Section 3]{SSS18} for a detailed explanation.
The identification of $\mathcal{M}_d$ with a quotient of $\mathcal{L}_d$, explained in Section \ref{sec:sym}, now translates to the fact that the restriction map
\begin{equation*}
  H^\infty_{nc} \to \mathcal{M}_d, \quad \Phi \mapsto \Phi \big|_{\mathbb{B}_d},
\end{equation*}
is a (complete) quotient map; see \cite[Section 11]{SSS18} and \cite[Section 3.4]{Har22} for more details.
(This statement is considerably deeper than its Hilbert space counterpart \eqref{eqn:Fock_restriction}.)
Pushing this line of reasoning even further, one finds that $\mathcal{M}_d$ can be identified
with the restriction of $H^\infty_{nc}$ to the subvariety $\mathfrak{C} \mathbb{B}_d^{nc}$ of $\mathbb{B}_d^{nc}$ consisting
of all commuting tuples of matrices; see again \cite[Section 11]{SSS18}.
The multiplier norm of $H^2_d$, which is classically not the supremum
norm over $\mathbb{B}_d$ (Theorem \ref{thm:MnotHinfty}), now becomes a supremum norm
again, but over the set $\mathfrak{C} \mathbb{B}_d^{nc}$.

The idea behind $H^2_{nc}$ goes back to work of Popescu \cite{Popescu06a}.
The elements of $H^2_{nc}$ are really noncommutative holomorphic functions in the sense of Taylor \cite{Taylor72,Taylor73};
  see also \cite{AMY20} and \cite{KV14} for more recent treatments of the theory. In fact,
one can regard $H^2_{nc}$ as a noncommutative reproducing kernel Hilbert space as introduced by Ball, Marx and Vinnikov \cite{BMV16,BMV18};
see again \cite[Section 3]{SSS18} for this point of view on $H^2_{nc}$.

\subsection{Characteristic function and spectrum}

The classical Sz.-Nagy--Foias characteristic function of a completely non-unitary contraction $T$ on Hilbert space
is an operator-valued bounded analytic function on the disc that serves as a complete unitary invariant.
In addition, the spectrum of $T$ is encoded in function theoretic properties of the characteristic function;
see \cite[Chapter 6]{SzNF2010}.

Many of these ideas were extended to commuting row contractions by Bhattacharyya, Eschmeier and Sarkar \cite{BES05}.
Let $T= (T_1,\ldots,T_d)$ be a pure commuting row contraction on $\mathcal{H}$.
We regard $T$ as a row operator $\mathcal{H}^d \to \mathcal{H}$ and consider the defect operators
\begin{equation*}
  D_{T^*} = ( I_{\mathcal{H}^d} - T^* T)^{1/2} \in B(\mathcal{H}^d) \quad \text{ and } \quad
  D_{T} = (I_{\mathcal{H}} - T T^*)^{1/2} \in B(\mathcal{H})
\end{equation*}
and the defect spaces
\begin{equation*}
  \mathcal{D}_{T^*} = \overline{D_{T^*} \mathcal{H}^d} \quad \text{ and } \mathcal{D}_{T} = \overline{ D_{T} \mathcal{H}}.
\end{equation*}
(Compared to \cite{BES05}, the two defect operators and spaces are each interchanged, but the notation is chosen so as to be consistent with that in Subsection \ref{ss:defect}.)
The \emph{characteristic function} $\theta_T$ of $T$ is then defined to be
\begin{equation*}
  \theta_T: \mathbb{B}_d \to B(\mathcal{D}_{T^*},\mathcal{D}_{T}), \quad
  \theta_T(z) = - T + D_{T}(I_{\mathcal{H}} - Z(z) T^*)^{-1} Z(z) D_{T^*},
\end{equation*}
where $Z(z)$ stands for the row operator $[z_1 , \ldots, z_d]$.
The characteristic function is a transfer function in the sense of Theorem \ref{thm:transfer},
so it follows from that theorem that $\theta_T$ is a contractive operator-valued multiplier of $H^2_d$.
The characteristic function is related to the dilation map $W$ of Lemma \ref{lem:Poisson_kernel} via
\begin{equation*}
  W W^* + M_{\theta_T} M_{\theta_T}^* = I_{H^2_d \otimes \mathcal{D}_{T}},
\end{equation*}
see \cite[Lemma 3.6]{BES05}. In particular, this gives a description
of the dilation space $K$ in Theorem \ref{thm:model_pure}, namely as the orthogonal complement
of the range of $M_{\theta_T}$.

The following result is \cite[Theorem 4.4]{BES05}.

\begin{theorem}
  Two pure commuting row contractions are unitarily equivalent if and only if their characteristic functions
  are equal.
\end{theorem}

We say that $\theta_T$ is \emph{surjective at $\lambda \in \mathbb{B}_d$} if $\theta_T(\lambda) \mathcal{D}_{T^*} = \mathcal{D}_{T}$.
Moreover, we say that $\theta_T$ is surjective at $\lambda \in \partial \mathbb{B}_d$ if $\theta_T$ extends
to a holomorphic map $\widetilde{\theta}$ on an open set containing $\lambda$ and
$\widetilde{\theta}(\lambda) \mathcal{D}_{T^*} = \mathcal{D}_{T}$.

  The following description of the Taylor spectrum of $T$ was obtained by Didas, Eschmeier, Hartz and Scherer \cite[Corollary 10]{DEH+24}.

\begin{theorem}
  Let $T$ be a pure commuting row contraction with $\dim \mathcal{D}_{T} < \infty$. Then
  \begin{equation*}
    \sigma(T) = \{\lambda \in \overline{\mathbb{B}_d}: \theta_T \text{ is not surjective at } \lambda \}.
  \end{equation*}
\end{theorem}

This result follows from a description of the Taylor spectrum of certain quotients of the $d$-shift.
Other descriptions of the spectra of such quotients were obtained by Clou\^atre and Timko; see \cite{CT23}.

\subsection{de Branges--Rovnyak spaces and Alexandrov--Clark theory}

Let $b \in H^\infty(\mathbb{D})$ be a non-constant function with $\|b\|_\infty \le 1$.
The classical {de Branges--Rovnyak space} $\mathcal{H}(b)$ associated with $b$ is the reproducing kernel Hilbert space on the disc with reproducing
kernel
\begin{equation*}
  \frac{1 - b(z) \overline{b(w)}}{1 - z \overline{w}}.
\end{equation*}
Among other things, these spaces serve as model spaces for certain classes of operators.

The theory of de Branges--Rovnyak spaces was generalized to the setting of the Drury--Arveson space by Jury
\cite{Jury14} and by Jury and Martin \cite{JM18b}. Given a contractive multiplier $b \in \mathcal{M}_d$,
the corresonding \emph{de Branges--Rovnyak space} is defined to be the reproducing kernel Hilbert space on $\mathbb{B}_d$
with reproducing kernel
\begin{equation*}
  \frac{1 - b(z) \overline{b(w)}}{1 - \langle z,w \rangle }.
\end{equation*}
In particular, much of classical Alexandrov--Clark theory, which deals with rank one perturbations
of backward shifts on $\mathcal{H}(b)$ spaces, was generalized to the setting of the Drury--Arveson space \cite{Jury14,JM18b}.
The generalization is not straightforward, and remarkably, many of the arguments in several variables rely on the noncommutative theory.
For instance, a noncommutative analogue of the Herglotz formula for holomorphic functions with positive real part is used.
The role of the backward shift is played by contractive solutions to the Gleason problem in the de Branges--Rovnyak space.
As in one variable, the theory is substantially different depending on whether or not the contractive multiplier $b$
is an extreme point of the unit ball of $\mathcal{M}_d$; see \cite{JM18c,Hartz20}.

\subsection{von Neumann's inequality for row contractive matrices}

For $d \ge 2$, it follows from the incompatibility of multiplier norm and supremum norm (Theorem \ref{thm:MnotHinfty})
that there do not exist finite constants $C_d$ such that the von Neumann-type inequality
\begin{equation*}
  \|p(T)\| \le C_d \|p\|_\infty
\end{equation*}
holds for all polynomials $p$ and all $d$-variable commuting row contractions $T$.
However, when one restricts to Hilbert spaces of a fixed finite dimension, such constants do exist.
In fact, one can take them to be uniform in the number of variables $d$.
The following is the main result of \cite{HRS21}, proved by Hartz, Richter and Shalit.

\begin{theorem}
  For all $n \in \mathbb{N}$, there exists a constant $C_n$ such that for all $d \in \mathbb{N}$, the inequality
  \begin{equation*}
    \|p(T)\| \le C_n \|p\|_\infty
  \end{equation*}
  holds for all row contractions $T$ consisting of $d$ commuting $n \times n$ matrices and all polynomials $p$ in $d$ variables.
\end{theorem}

\subsection{Cyclic functions in the Drury--Arveson space}

A function $f \in H^2_d$ is said to be \emph{cyclic} if $\{f \cdot \varphi: \varphi \in \mathcal{M}_d \}$
is dense in $H^2_d$.
As in any function space, it is natural to try to determine
which functions in $H^2_d$ are cyclic.
The cyclic functions in $H^2(\mathbb{D})$ are precisely the outer functions.
For general $d \in \mathbb{N}$, no simple description of cyclic functions is known,
but there are necessary and sufficient conditions.

Since point evaluations are continuous, a cyclic function cannot vanish anywhere
on $\mathbb{B}_d$.
Even if $d=1$, this necessary condition is not sufficient,
as witnessed by singular inner functions.
Other necessary conditions come from the size of the zero set of $f$ on $\partial \mathbb{B}_d$ (appropriately interpreted); see for instance \cite[Theorem 4.8]{APR+24}.

The following result contains some sufficient conditions for cyclicity
obtained by Aleman, Perfekt, Richter, Sundberg and Sunkes.

\begin{theorem}
  Let $f \in H^2_d$ be a function without zeros in $\mathbb{B}_d$.
  Each of the following conditions is sufficient for cyclicity of $f$:
\begin{itemize}
  \item $|f|$ is bounded below; see the discussion following Theorem 3.7 in \cite{APR+24};
  \item $\frac{1}{f} \in H^2_d$; this follows from the Smirnov
factorization discussed below, see also \cite[Theorem 3.1]{APR+23} for a more general statement;
  \item $f$ has bounded argument and $\log f \in H^2_d$
\cite[Theorem 1.1]{APR+23}.
\end{itemize}
\end{theorem}

The question of cyclicity is even interesting for polynomials.
If $d \le 2$, then any polynomial without zeros in $\mathbb{B}_2$ is cyclic,
but if $d \ge 4$, then there are polynomials
without zeros in $\mathbb{B}_d$ that are not cyclic in $H^2_d$;
see \cite{KV23} and the discussion following Theorem 3.5 in \cite{APR+24}. The case $d = 3$ appears to be open.

\subsection{Inner/outer factorization}
\label{ss:inner/outer}

A cornerstone of the theory of the Hardy space $H^2(\mathbb{D})$ is the inner/outer factorization:
each non-zero function $f \in H^2(\mathbb{D})$ factors essentially uniquely as $f = \varphi g$,
where $\varphi \in H^\infty(\mathbb{D})$ is inner and $g \in H^2(\mathbb{D})$ is outer.
One way to interpret these terms is to say that $\varphi$ induces an isometric multiplication
operator, and $g$ is a cyclic function.

The factorization was generalized to the Drury--Arveson space by Jury and Martin;
see \cite{JM21} and \cite[Theorem 1.1]{JM18} for the precise statement below.

\begin{theorem}
  \label{thm:inner_outer}
  Each $f \in H^2_d \setminus \{0\}$ factors as
  $f = \varphi g$, where $\varphi \in \mathcal{M}_d$ with $\|\varphi\|_{\mathcal{M}_d} \le 1$
  and $g \in H^2_d$ is cyclic with $\|f\| = \|g\|$.
\end{theorem}

If $d=1$, then $\varphi$ in the previous theorem will automatically be inner, but if $d \ge 2$,
then there are no non-constant isometric multipliers of $H^2_d$; see \cite[Proposition 8.36]{AM02}.
The proof of Jury and Martin goes through the noncommutative universe and uses a
Beurling theorem of Arias--Popescu and Davidson--Pitts in the full Fock space.
No commutative proof of Theorem \ref{thm:inner_outer} appears to be known.

The factors $\varphi$ and $g$ in Theorem \ref{thm:inner_outer} can be characterized
intrinsically in terms of $H^2_d$, and there is a corresponding uniqueness statement.
The factor $g$ is what has been called \emph{free outer},
a property that is generally stronger than cyclicity. It was shown by Aleman, Hartz, \mcc\ and Richter
in \cite{AHM+22} that
the factorization in Theorem \ref{thm:inner_outer} becomes essentially unique
if one insists that $g$ be free outer.
The multiplier $\varphi$ in this factorization is called a \emph{subinner} multiplier,
meaning that the multiplication operator is isometric on the linear span of $g$.
Thus, the factorization is called the \emph{subinner/free outer} factorization.

The factorization extends to vector valued $H^2_d$:
Each non-zero $F \in H^2_d \otimes \ell^2$ factors as $F = \Phi g$,
where $\Phi$ is a multiplier
from $H^2_d$ to $H^2_d \otimes \ell^2$ of norm one
and $g \in H^2_d$ is a scalar-valued free outer function with $\|g\| = \|F\|$;
see \cite[Theorem 1.1]{JM18}.
This vector valued factorization is crucial for some applications, such as weak products, which will be discussed below. These results in fact hold more generally for normalized complete Pick spaces.

\subsection{Smirnov factorization}

The classical Smirnov class can be defined as
\begin{equation*}
  N^+ = \Big \{ \frac{\varphi}{\psi}: \varphi,\psi \in H^\infty(\mathbb{D}), \psi \text{ outer} \Big \}.
\end{equation*}
The Hardy space $H^2(\mathbb{D})$ is contained in $N^+$.
Similarly, one can define the Drury--Arveson--Smirnov class as
\begin{equation*}
  N^+(H^2_d) =
  \Big \{ \frac{\varphi}{\psi}: \varphi,\psi \in \mathcal{M}_d, \psi \text{ cyclic} \Big \}.
\end{equation*}
It follows from a result of Alpay, Bolotnikov and Kaptanoglu \cite{ABK02}
that $H^2_d \subset N^+(H^2_d)$.
The precise statement below is \cite[Theorem 1.1]{AHM+17a}.

\begin{theorem}
  \label{thm:Smirnov_factor}
  Let $f \in H^2_d$ with $\|f\|_{H^2_d} \le 1$.
  Then there exist $\varphi,\psi \in \mathcal{M}_d$ with $\|\varphi\|_{\mathcal{M}_d} \le 1$, $\|\psi\|_{\mathcal{M}_d} \le 1$
  and $\psi(0) = 0$ such that
  \begin{equation*}
    f = \frac{\varphi}{1 - \psi}.
  \end{equation*}
  Moreover, the multiplier $1 - \psi$ is cyclic.
  In particular, $H^2_d \subset N^+(H^2_d)$.
\end{theorem}

This Smirnov factorization result has a number of basic consequences. For instance, the zero sets for $H^2_d$ and for $\mathcal{M}_d$ agree; this improves on Theorem \ref{thm:sero_set}.
Moreover, the union of two $H^2_d$-zero sets is another $H^2_d$-zero set; again see \cite{AHM+17a}.

Given a function $f \in H^2_d$, one can explicitly write down multipliers $\varphi,\psi$ as in Theorem \ref{thm:Smirnov_factor};
see \cite[Theorem 1.1]{AHM+17c}.
Generally, the Smirnov factorization is different from the inner/outer factorization (i.e.\ the multipliers $\varphi$ in the two
factorizations are different and $\frac{1}{1 - \psi}$ in the Smirnov factorization does not equal the free outer factor $g$).
There is a vector valued version of this factorization; these results hold for all normalized complete Pick spaces \cite{AHM+17a,AHM+17c}.

\subsection{Common range of co-analytic Toeplitz operators}

A theorem of M\textsuperscript{c}Carthy describes the intersection of the ranges of all non-zero co-analytic Toeplitz
operators on $H^2(\mathbb{D})$. Thanks to the inner/outer factorization, this space is the same
as 
\begin{equation*}
  \bigcap_{\varphi \in H^\infty(\mathbb{D}) \text{ outer }} \operatorname{ran}(M_\varphi^*).
\end{equation*}

In \cite{AHM+ar}, Aleman, Hartz, \mcc\ and Richter described the common range
\begin{equation*}
  \mathcal{R} = \bigcap_{\varphi \in \mathcal{M}_d \text{ cyclic} } \operatorname{ran}(M_\varphi^*) \subset H^2_d.
\end{equation*}
Roughly speaking, a function belongs to $\mathcal{R}$ if and only if its Taylor coefficients
satisfy a simple decay condition. Moreover, $\mathcal{R}$ is the dual space of the Smirnov class $N^+(H^2_d)$.

\subsection{Weak products}
\label{ss:wp}

In the study of the Hardy space $H^2(\mathbb{D})$, several related spaces are relevant,
such has $H^1(\mathbb{D}), H^\infty(\mathbb{D})$ and more generally $H^p(\mathbb{D})$.
For the Drury--Arveson space, the role of $H^\infty(\mathbb{D})$ is played by the multiplier algebra
$\mathcal{M}_d$.
There is compelling evidence that the appropriate generalization of $H^1(\mathbb{D})$ is the
\emph{weak product space}
\begin{equation*}
  H^2_d \odot H^2_d = \Big\{ h = \sum_{n=1}^\infty f_n g_n: f_n,g_n \in H^2_d, \sum_{n=1}^\infty \|f_n\| \, \|g_n\| < \infty \Big\}.
\end{equation*}
The norm in the weak product space is the infimum of all sums on the right.

The definition of the weak product space goes back to Coiffman, Rochberg and Weiss \cite{CRW76}.
It is inspired by the classical fact that
\begin{equation*}
  H^1(\mathbb{D}) = \{h = f g: f,g \in H^2(\mathbb{D}) \},
\end{equation*}
but the definition is modified to ensure that the weak product space is a vector space (and better yet, a Banach space).

A remarkable theorem of Jury and Martin \cite{JM18} shows that the simple description
of $H^1(\mathbb{D})$ in fact generalizes. The precise version below is \cite[Theorem 1.3]{Hartz20}.

\begin{theorem}
  \label{thm:JM_weak_product}
  If $h \in H^2_d \odot H^2_d$, then there exist $f,g \in H^2_d$ with $h = f g$
  and $\|f\|_{H^2_d} \|g\|_{H^2_d} = \|h\|_{H^2_d \odot H^2_d} $.
\end{theorem}

The multiplier algebra of $H^1(\mathbb{D})$ is $H^\infty(\mathbb{D})$, which is also the multiplier algebra of $H^2(\mathbb{D})$.
Whereas the description of $\Mult(H^2_d) = \mathcal{M}_d$ is not as simple, the equality
of multiplier algebras for weak product and Hilbert space remains true.
The following result was proved by Richter and Wick for $d \le 3$ \cite{RW16}, and by Clou\^atre and Hartz
for general $d$ \cite{CH19}.
The precise statement is \cite[Theorem 1.4]{Hartz20}.

\begin{theorem}
  $\Mult(H^2_d \odot H^2_d) = \Mult(H^2_d)$ and the multiplier norms are equal.
\end{theorem}

Beurling's theorem shows that the multiplier invariant subspaces of $H^2(\mathbb{D})$ are of the
form $\varphi H^2(\mathbb{D})$ for some inner function $\varphi$. Similarly,
the multiplier invariant subspaces of $H^1(\mathbb{D})$ are of the form $\varphi H^1(\mathbb{D})$ for some
inner function $\varphi$. In particular, there is a one-to-one correspondence
between multiplier invariant subspaces of $H^2(\mathbb{D})$ and of $H^1(\mathbb{D})$.
This principle remains true in the Drury--Arveson space.
Part of the following result was shown in \cite{RS16} by Richter and Sunkes, the full statement
was obtained by Aleman, Hartz, \mcc\ and Richter in \cite[Theorem 3.7]{AHM+18}.

\begin{theorem}
  The maps
  \begin{align*}
    {M} &\mapsto \overline{{M}}^{H^2_d \odot H^2_d} \\
    {N} \cap H^2_d &\mapsfrom {N},
  \end{align*}
  are mutually inverse bijections between closed multiplier invariant subspaces ${M}$ of $H^2_d$
  and closed multiplier invariant subspaces ${N}$ of $H^2_d \odot H^2_d$.
\end{theorem}

There are versions of the inner/outer and the Smirnov factorization for the weak product.
The following inner/outer factorization is \cite[Theorem 1.11]{AHM+22}.

\begin{theorem}
  Let $h \in H^2_d \odot H^2_d \setminus \{0\}$.
  Then there exist $\varphi \in \mathcal{M}_d$ of multiplier norm one
  and a free outer function $g \in H^2_d$ with $h = \varphi g^2$ and $\|h\|_{H^2_d \odot H^2_d}
  = \|g^2\|_{H^2_d \odot H^2_d} = \|g\|_{H^2_d}^2$.
\end{theorem}

There is also a version of the Smirnov factorization for the weak product space,
established in \cite{AHM+18}. The precise statement below is \cite[Theorem 4.4]{Hartz20}.

\begin{theorem}
  Let $h \in H^2_d \odot H^2_d$ with $\|h\|_{H^2_d \odot H^2_d} \le 1$.
  Then there exist $\varphi,\psi \in \mathcal{M}_d$ with $\|\varphi\|_{\mathcal{M}_d} \le 1$, $\|\psi\|_{\mathcal{M}_d} \le 1$
  and $\psi(0) = 0$ such that
  \begin{equation*}
    f = \frac{\varphi}{(1 - \psi)^2}.
  \end{equation*}
\end{theorem}

All results in this section in fact hold for all normalized complete Pick spaces.

\subsection{The column-row property}

The results in Subsection \ref{ss:wp} all depend on the column-row property of $H^2_d$,
which will now be discussed. Given a sequence $(\varphi_n)$ in $\mathcal{M}_d$, one can
consider two (potentially unbounded) operators, namely the column operator
\begin{equation*}
  \begin{bmatrix}
    M_{\varphi_1} \\ M_{\varphi_2} \\ \vdots
  \end{bmatrix}: H^2_d \to H^2_d \otimes \ell^2
\end{equation*}
and the row operator
\begin{equation*}
  \begin{bmatrix}
    M_{\varphi_1} & M_{\varphi_2} & \cdots
  \end{bmatrix}: H^2_d \otimes \ell^2 \to H^2_d.
\end{equation*}
For general Hilbert space operators, there is no relationship between boundedness of the row
and boundedness of the column. For multiplication operators on $H^2(\mathbb{D})$ (i.e.\ $d=1$ above),
the norm of the row and that of the column are both equal to $\sup_{z \in \mathbb{D}} \|(\varphi_n(z)\|_{\ell^2}$.
But this in no longer true for $d \ge 2$, since multiplier norm and supremum norm
are not even comparable; see Theorem \ref{thm:MnotHinfty}.

Nonetheless, we have the following result from \cite{Hartz20}.

\begin{theorem}
  Let $(\varphi_n)$ be a sequence in $\mathcal{M}_d$. Then
\begin{equation*}
  \left\|
  \begin{bmatrix}
    M_{\varphi_1} \\ M_{\varphi_2} \\ \vdots
  \end{bmatrix} \right\|
  \le 
  \left\|\begin{bmatrix}
    M_{\varphi_1} & M_{\varphi_2} & \cdots
  \end{bmatrix} \right\|.
\end{equation*}
\end{theorem}

This result is usually phrased as ``$H^2_d$ satisfies the column-row property with constant $1$''.
It was shown earlier in \cite{AHM+18}, extending an argument of Trent \cite{Trent04},
that $H^2_d$ satisfies the column-row property with some constant $c_d \ge 1$, which appeared
as a factor on the right-hand side.
If $d \ge 2$, then there are examples of sequences in $\mathcal{M}_d$ that yield unbounded columns but bounded rows
\cite{AHM+18}.
Thus, the column-row property is really asymmetrical and there is no ``row-column property''.
The column-row property in fact holds for all normalized complete Pick spaces.

The relevance of the column-row property
can be explained for instance in the context of Theorem \ref{thm:JM_weak_product}.
Let $h  = \sum_{n=1}^\infty f_n g_n \in H^2_d \odot H^2_d$; the goal is to factor $h$
as a product of two functions in $H^2_d$.
By trading constant factors between $f_n$ and $g_n$, we may without loss of generality
assume that $\|f_n\| = \|g_n\|$ for all $n$, so $\sum_{n=1}^\infty \|f_n\|^2 < \infty$
and $\sum_{n=1}^\infty \|g_n\|^2 < \infty$.
We now apply the vector-valued inner/outer factorization to $(f_n) \in H^2_d \otimes \ell^2$
and $(g_n) \in H^2_d \otimes \ell^2$; see Subsection \ref{ss:inner/outer}.
Thus, we obtain $F, G \in H^2_d$ and sequences $(\varphi_n), (\varphi_n)$ in $\mathcal{M}_d$,
each forming a bounded column multiplier, such that $f_n = \varphi_n F$ and $g_n = \psi_n G$
for all $n$. Hence,
\begin{equation*}
  h = \sum_{n=1}^\infty \varphi_n \psi_n F G.
\end{equation*}
If we knew that $\theta: = \sum_{n=1}^\infty \varphi_n \psi_n \in \mathcal{M}_d$, then we would
get our desired factorization of $h$ into a product of two functions in $H^2_d$,
namely $\theta F$ and $G$.
But this is guaranteed by the column-row property, since
\begin{equation*}
  \theta =
  \begin{bmatrix}
    \varphi_1 & \varphi_2 & \cdots
  \end{bmatrix}
  \begin{bmatrix}
    \psi_1 \\ \psi_2 \\ \vdots
  \end{bmatrix},
\end{equation*}
the column of the $\psi_n$ is bounded by the statement of the inner/outer factorization,
and the row of the $\varphi_n$ is bounded by the column-row property, since the column is.
With a little extra work and using the fact that the column-row property holds with constant $1$,
one obtains the norm equality in Theorem \ref{thm:JM_weak_product}; see \cite[Theorem 1.3]{JM18} for details.

In addition to weak products, the column-row property also plays a role in the context
of interpolating sequences and of de Branges--Rovnyak spaces; see \cite{Hartz20} more details.

\subsection{Hankel operators}

A function $b \in H^2_d$ is said to be a \emph{Hankel symbol} if
there exists a constant $C \in [0,\infty)$ such that
\begin{equation*}
  | \langle \varphi f, b \rangle | \le \| \varphi\|_{H^2_d} \| f\|_{H^2_d}
  \quad \text{ for all } 
\varphi \in \mathcal{M}_d, f \in H^2_d.
\end{equation*}
We write $\Han(H^2_d)$ for the space of Hankel symbols.
Each $b \in \Han(H^2_d)$ gives rise to a (little) \emph{Hankel operator} $H_b$,
mapping $H^2_d$ into the conjugate Hilbert space $\overline{H^2_d}$. This operator
is characterized by the equation
\begin{equation}
  \label{eqn:Hankel}
  \langle H_b f, \overline{\varphi} \rangle_{\overline{H^2_d}}
  = \langle \varphi f, b \rangle_{H^2_d}
  \quad \text{ for all } \varphi \in \mathcal{M}_d, f \in H^2_d.
\end{equation}

If $d=1$, then a theorem of Fefferman shows that $\Han(H^2(\mathbb{D})) = \operatorname{BMOA}$,
the space of analytic functions of bounded mean oscillation;
see for instance \cite[Chapter VI]{Garnett}.
This space is important in the study of $H^2(\mathbb{D})$.
We think of $\Han(H^2_d)$ as playing the role of of $\operatorname{BMOA}$.
It was shown by Richter and Sunkes that $\mathcal{M}_d \subset \Han(H^2_d)$ for $d < \infty$,
see \cite[Theorem 1.1]{RS16}.

Again if $d=1$, it is classical theorem of Nehari that the dual space of $H^1(\mathbb{D})$
can be identified with $\Han(H^2(\mathbb{D}))$. This result
was generalized to the Drury--Arveson space by Richter and Sundberg \cite[Theorem 1.3]{RS14}.

\begin{theorem}
    $(H^2_d \odot H^2_d)^* \cong \Han(H^2_d)$.
\end{theorem}

Equation \eqref{eqn:Hankel} implies the intertwining relation
\begin{equation*}
  H_b M_\varphi = M_{\overline{\varphi}}^* H_b
  \quad \text{ for all } b \in \Han(H^2_d), \varphi \in \mathcal{M}_d.
\end{equation*}
It follows that $\ker(H_b)$ is a closed multiplier invariant subspace of $H^2_d$ for all $b \in \Han(H^2_d)$.
The following converse is due to Richter and Sunkes \cite[Theorem 4.2]{RS16}.

\begin{theorem}
  If ${M}$ is a non-zero closed multiplier invariant subspace of $H^2_d$, then there
  exists a sequence $(b_n)$ in $\Han(H^2_d)$ such that
  \begin{equation*}
    {M} = \bigcap_n \ker(H_{b_n}).
  \end{equation*}
\end{theorem}

This result continues to hold for all normalized complete Pick spaces; see \cite[Corollary 3.8]{AHM+18} and
\cite{Hartz20}.

\subsection{$H^p$-scale}

It is possible to use the complex method of interpolation of Banach spaces to define an $H^p$-scale
($1 \le p < \infty$) for $H^2_d$ by interpolating between the weak product space $H^2_d \odot H^2_d$
and the space $\Han(H^2_d)$ of Hankel symbols; see \cite{AHM+20}.
For $p=2$, one recovers the Hilbert space $H^2_d$.
Functions in the $H^p$-space
can at most grow like $(1 - \|z\|^2)^{-1/p}$ near $\partial \mathbb{B}_d$,
and this is sharp; see \cite[Theorem 3.6]{AHM+20}.
However, many basic questions, such as an intrinsic description of the elements of the $H^p$-space,
remain open.

\subsection{Membership in $\mathcal{M}_d$}

The known characterizations of multipliers, such as Theorem \ref{thm:mult_char}, are sometimes
difficult to use in practice. Thus, one looks for simpler necessary or sufficient conditions for membership in $\mathcal{M}_d$.
Many of these involve the reproducing kernel $k_z = k(\cdot,z)$ of $H^2_d$.

It is immediate that the condition
\begin{equation}
  \label{eqn:nec_mult}
  \sup_{z \in \mathbb{B}_d} \frac{\|f k_z\|}{\|k_z\|} < \infty
\end{equation}
is necessary for membership in $\mathcal{M}_d$. Since $f(z) = \|k_z\|^{-2} \langle f k_z,k_z \rangle$,
this condition in particular implies boundedness of $f$. However, it was shown by Fang and Xia that \eqref{eqn:nec_mult}
is not sufficient for membership in $\mathcal{M}_d$ \cite{FX15}.

As for sufficient conditions, we have the following result of Aleman, Hartz, \mcc\ and Richter, which is \cite[Corollary 4.6]{AHM+17c}.

\begin{theorem}
  If $f \in H^2_d$ satisfies
  \begin{equation}
    \label{eqn:real_part}
    \sup_{z \in \mathbb{B}_d} \operatorname{Re} \langle f, k_z f \rangle < \infty,
  \end{equation}
  then $f \in \mathcal{M}_d$.
\end{theorem}

If $d=1$, then \eqref{eqn:real_part} means that the Poisson integral of $|f|^2$ is bounded, so it is also necessary.
However, for $d \ge 2$, then \eqref{eqn:real_part} is not necessary for membership in $\mathcal{M}_d$, see \cite{FX20} and
\cite[Proposition 8.1]{AHM+20a}.

\subsection{Interpolating sequences}

Interpolating sequences for $\mathcal{M}_d$ were discussed in Subsection \ref{ss:is}.
In the meantime, Aleman, Hartz, \mcc\ and Richter characterized interpolating sequences in terms of Carleson measure
and separation conditions, extending Carleson's characterization of interpolating sequences
for $H^\infty(\mathbb{D})$.
The following is the main result of \cite{AHM+17}.

\begin{theorem}\label{thm:interpolating_DA}
  A sequence $(z_n)$ in $\mathbb{B}_d$ is an interpolating sequence for $\mathcal{M}_d$
  if and only if it is separated in the pseudo-hyperbolic metric of $\mathbb{B}_d$
  and the measure
  $\sum_{n=1}^\infty (1 - \|z_n\|^2) \delta_{z_n}$
  is a Carleson measure for $H^2_d$.
\end{theorem}
Explicitly, the Carleson measure condition means that
there exists a constant $C$ such that
\begin{equation*}
  \sum_{n=1}^\infty (1 - \|z_n\|^2) |f(z_n)|^2 \le C \|f\|^2 \quad \text{ for all } f \in H^2_d.
\end{equation*}
The first proof of this result relied on the solution of the Kadison--Singer problem due to Marcus, Spielman and Srivastava \cite{MSS15}.
There is a second proof using the column-row property of the Drury--Arveson space, see \cite[Remark 3.7]{AHM+17}, \cite[Section 4]{AHM+18} and \cite[Theorem 4.5]{Hartz20}.

A sequence $(z_n)$ in $\mathbb{B}_d$ is said to be \emph{simply interpolating} for $H^2_d$ if the map
\begin{equation*}
  f \mapsto \big( (1 - \|z_n\|^2)^{1/2} f(z_n)  \big)_{n=1}^\infty
\end{equation*}
maps $H^2_d$ onto (but not necessarily into) $\ell^2$. Theorem \ref{thm:ms_is} implies
in particular that every interpolating sequence for $\mathcal{M}_d$ is simply interpolating for $H^2_d$.

It was shown by Chalmoukis, Dayan and Hartz \cite[Theorem 1.1]{CDH23} that a sequence is simply interpolating if and only if it is what is called strongly separated.
In the case of $H^2(\mathbb{D})$, simply interpolating and (multiplier) interpolating sequences agree.
If $d \ge 2$, then there are simply interpolating sequences for $H^2_d$ that are not interpolating for $\mathcal{M}_d$ \cite[Theorem 1.2]{CDH23}.

Appropriate versions of the results discussed in this subsection hold for all normalized complete Pick spaces.

\subsection{Henkin theory and peak interpolation}
\label{ss:Henkin}

Henkin measures for $\mathcal{M}_d$ are complex regular Borel measures $\mu$ on $\partial \mathbb{B}_d$
that, very roughly speaking, can be thought of as being absolutely continuous
with respect to the multiplier algebra of $H^2_d$.
More precisely, $\mu$ is said to be \emph{$\mathcal{M}_d$-Henkin} if whenever $(p_n)$ is a sequence
of polynomials that is bounded in the multiplier norm of $H^2_d$ and that converges to $0$
pointwise on the open ball, then $\lim_{n \to \infty} \int_{\partial \mathbb{B}_d} p_n \, d \mu = 0$.
This is the same as demanding that there is a weak-$*$ continuous linear functional
on $\mathcal{M}_d$ that agrees with integration against $\mu$ on $\mathcal{A}_d$.

It is a consequence of the F.\ and M. Riesz theorem that if $d=1$, then a measure
is Henkin if and only if it is absolutely continuous with respect to Lebesgue measure
on the circle. In several variables, there is an extensive theory
of Henkin measures for $H^\infty(\mathbb{B}_d)$; see \cite[Chapter 9]{RudinBall}.
For the Drury--Arveson space, they were introduced by Clou\^atre and Davidson \cite{CD16}.
Every $H^\infty(\mathbb{B}_d)$-Henkin measure is $\mathcal{M}_d$-Henkin,
but the converse may fail \cite{Hartz17}.

Henkin measures play a role in the description of the dual space of the algebra $\mathcal{A}_d$.
We say that a measure $\mu$ is \emph{$\mathcal{M}_d$-totally singular} if it is singular
with respect to every $\mathcal{M}_d$-Henkin measure.
The space of $\mathcal{M}_d$-totally singular measures will be denoted by $\operatorname{TS}$.
We also let $(\mathcal{M}_d)_*$ be the subspace of $\mathcal{A}_d^*$ consisting of those
functionals that extend weak-$*$ continuously to $\mathcal{M}_d$.
The following result is due to Clou\^atre and Davidson; it is a combination of results in \cite[Section 4]{CD16};
see also \cite[Theorem 3.2]{DH20} for a different proof.

\begin{theorem}
  $\mathcal{A}_d^* = (\mathcal{M}_d)_* \oplus_1 \operatorname{TS}$. 
\end{theorem}

Henkin measures also come up in the context of peak interpolation.
A compact subset $E \subset \partial \mathbb{B}_d$ is said to be a \emph{peak
interpolation set} for $\mathcal{A}_d$ if for every non-zero continuous function
$g$ on $E$, there exists a function $f \in \mathcal{A}_d$ such that $f \big|_E = g$,
$\|f\|_{\mathcal{M}_d} \le \|g\|_\infty$ and $|f(z)| < \|g\|_\infty$ for all
$z \in \overline{\mathbb{B}_d} \setminus E$.
If $d=1$, then classical theorems of Rudin and Carleson show that the peak interpolation sets
for the disc algebra are precisely the compact subsets of $\partial \mathbb{D}$ of Lebesgue measure
zero.
For $d \ge 2$, the key notion is the following: A compact set $E \subset \mathbb{B}_d$ is called $\mathcal{M}_d$-totally null
if $|\mu(E)| = 0$ for all $\mathcal{M}_d$-Henkin measures $\mu$.
The following result was shown by Davidson and Hartz \cite[Theorem 1.8]{DH20},
a slightly weaker statement was obtained earlier by Clou\^atre and Davidson \cite{CD16}.

\begin{theorem}
  A compact subset of $\partial \mathbb{B}_d$ is a peak interpolation
  set for $\mathcal{A}_d$ if and only if it is totally null.
\end{theorem}
Being a peak interpolation set is equivalent
to being an \emph{interpolation set} (the same property as peak interpolation but without norm or pointwise control)
and also to being a \emph{peak set} (peak interpolation for $g=1$);
see \cite[Theorem 1.8]{DH20}.
These results can also be approached via the noncommutative interpolation theory
developed by Blecher and others; see \cite{BC24}.

\subsection{Functional calculus}

Drury's inequality (Theorem \ref{thm:DruryVNineq}) implies that every $d$-con\-traction $T$
admits an $\mathcal{A}_d$-functional calculus, i.e.\ the obvious polynomial functional
calculus $p \mapsto p(T)$ for $T$ extends to a continuous algebra homomorphism
on $\mathcal{A}_d$.
In case $d=1$, it is a classical result of Sz.-Nagy and Foias that every contraction
$T$ without unitary direct summand even admits an
$H^\infty(\mathbb{D})$-functional calculus.
The following generalization to $d$-contractions is due to Clou\^atre and Davidson
\cite[Theorem 4.3]{CD16a}; see also \cite[Theorem 1.1]{BHM18} for a different proof.

\begin{theorem}
  Let $T$ be a $d$-contraction without spherical unitary direct summand.
  Then $T$ admits an $\mathcal{M}_d$-functional calculus, i.e.\ the polynomial
  functional calculus for $T$ extends to a weak-$*$ continuous algebra
  homomorphism on $\mathcal{M}_d$.
\end{theorem}

Proofs of this result use the theory of Henkin measures. A general
$d$-contraction $T$ decomposes into a direct sum of a spherical unitary
and a $d$-contraction without spherical unitary summand. A $d$-contraction acting
on a separable Hilbert space 
admits a weak-$*$ continuous $\mathcal{M}_d$-functional calculus
if and only if the spectral measure of the spherical unitary part is $\mathcal{M}_d$-Henkin;
see \cite[Lemma 3.1]{CD16a} or \cite[Theorem 4.3]{BHM18}.

\subsection{Ideals in $\mathcal{A}_d$}

Closed ideals in the disc algebra are described by classical results of Carleson and Rudin; see for instance \cite[Chapter 6]{Hoffman62}
for an exposition. Closed ideals in the algebra $\mathcal{A}_d$ were studied
by Clou\^atre and Davidson \cite{CD18}. Given an ideal $J \subset \mathcal{A}_d$,
let $Z(J) \subset  \overline{\mathbb{B}_d}$ be the common zero set of the functions in $J$.
Conversely, if $K \subset \overline{\mathbb{B}_d}$, we let $I(K) \subset \mathcal{A}_d$
be the ideal of all functions vanishing on $K$.

The following is \cite[Theorem 4.1]{CD18}.

\begin{theorem}
  Let $J \subset \mathcal{A}_d$ be a closed ideal and let $K = Z(J) \cap \partial \mathbb{B}_d$.
  Then
  \begin{equation*}
    J = I(K) \cap \widetilde{J},
  \end{equation*}
  where $\widetilde{J}$ is the weak-$*$ closure of $J$ in $\mathcal{M}_d$.
\end{theorem}

In turn, weak-$*$ closed ideals (which are the same as WOT-closed ideals) are in one-to-one correspondence
with closed multiplier invariant subspaces of $\mathcal{M}_d$; see Theorem \ref{thm:complete_lattice_iso}
and also \cite[Remark 3.5]{AHM+ar}.

\subsection{Boundary behavior and potential theory}

A classical theorem of Fatou shows that every function in $H^2(\mathbb{D})$ has a non-tangential
limit at every point in $\partial \mathbb{D}$ outside of a set of linear Lebesgue measure zero;
see for instance \cite[Section II.3]{Garnett}.
In several variables, Kor\'anyi's theorem shows that every function
in the Hardy space on the ball has a non-tangential limit at every point in $\partial \mathbb{B}_d$
outside of a set of surface measure zero; in fact, one can take Kor\'anyi limits,
which are more general than non-tangential limits;
see \cite[Theorem 5.6.4]{RudinBall}.

Since $H^2_d$ is contained in the Hardy space on the ball, Kor\'anyi's theorem applies to functions in $H^2_d$, but much more can be said in this case.
It turns out that the sharp version of Fatou's theorem involves a suitable notion of capacity.
The following definitions and results are all contained in work of Chalmoukis and Hartz \cite{CH24}.
Briefly, the \emph{energy} of a regular Borel probablity measure $\mu$ on $\partial \mathbb{B}_d$
is defined to be
\begin{equation*}
  \mathcal{E}(\mu) = \sup_{0 \le r < 1} \int_{\partial \mathbb{B}_d} \int_{\partial \mathbb{B}_d} \operatorname{Re} K(rz,rw) \, d \mu(w) d \mu(z) \in [0,\infty],
\end{equation*}
where $K(z,w) = \frac{1}{1 - \langle z,w \rangle }$ is the reproducing kernel of $H^2_d$.
This is the same as the square of the norm of the densely defined
integration functional $f \mapsto \int_{\partial \mathbb{B}_d} f \, d \mu$.
The \emph{capacity} of a compact set $E \subset \partial \mathbb{B}_d$ is then defined to be
\begin{equation*}
  \operatorname{cap}(E) = \sup\{ \mathcal{E}(\mu)^{-1}: \mu \text{ is a probability measure supported on }E \}.
\end{equation*}
In particular, $E$ has capacity zero if and only if $E$ does not support a probability measure of finite energy.
From there, one can extend the definition of capacity to arbitrary Borel sets in a standard manner,
in particular by approximation by compact sets from within.

Fatou's / Kor\'anyi's theorem then takes the following form in $H^2_d$.

\begin{theorem}
  If $f \in H^2_d$, then there exists a Borel set $E \subset \partial \mathbb{B}_d$ of capacity zero
  such that $f$ has a Kor\'anyi (in particular non-tangential) limit at every point
  in $\partial \mathbb{B}_d \setminus E$.
\end{theorem}

The capacity zero condition is sharp in the following sense.

\begin{theorem}
  If $E \subset \partial \mathbb{B}_d$ is a compact set of capacity zero, then there exists $f \in H^2_d$
  with $\lim_{r \nearrow 1} |f(r \zeta)| = \infty$ for all $\zeta \in E$.
\end{theorem}

The notion of capacity is related to totally null sets (see Subsection \ref{ss:Henkin}) in the following way.

\begin{theorem}
  A Borel set $E \subset \partial \mathbb{B}_d$ has capacity zero if and only if it is $\mathcal{M}_d$-totally null.
\end{theorem}

The capacity also plays a role in the context of cyclic functions in $H^2_d$.

\subsection{Other representations of the Drury--Arveson space}
In Subsection \ref{sec:BS}, we saw several Sobolev-type norms that are equivalent to the $H^2_d$-norm.
Using slightly more complicated differential operators, one can actually obtain a Sobolev-type norm that
is equal (and not just equivalent) to the $H^2_d$-norm; see \cite[Theorem 6.1]{AMP+19} of
Arcozzi, Monguzzi, Peloso and Salvatori.

In the classical theory of the Hardy space on the disc, it is sometimes convenient to work with the upper half plane
instead of the unit disc. In several variables, the unit ball is biholomorphically equivalent to the Siegel upper half space,
which plays the role of the upper half plane. The theory of the Drury--Arveson space on the Siegel upper half space was developed by Arcozzi, Chalmoukis, Monguzzi, Peloso and Salvatori \cite{ACM+21}.

\subsection{Embedding dimension}

Theorem \ref{thm:compPickUniversal}, due to Agler and \mcc, shows that every separable
irreducible complete Pick space embeds into $H^2_d$ for some $d \in \mathbb{N} \cup \{\infty\}$.
For many spaces, such as the Dirichlet space on the unit disc, one has to take $d = \infty$;
see \cite{Rochberg16} and \cite[Corollary 11.9]{Hartz17a}.
For the Dirichlet space, this remains true if one merely wants to realize the multiplier algebra
algebraically as one of the algebras $\mathcal{M}_V$ introduced in Subsection \ref{subsec:quot_var}; see \cite{Hartz22}. 
See \cite{MironovThesis} for a related study.

\subsection{Spaces of Dirichlet series}

The Drury--Arveson space $H^2_d$ consists of holomorphic functions in $d$ variables.
Especially in the case when $d = \infty$, this complicates function theoretic approaches to $H^2_d$.
M\textsuperscript{c}Carthy and Shalit showed that even in case $d = \infty$, the space $H^2_d$
is weakly isomorphic to a Hilbert space of Dirichlet series, whose elements are holomorphic functions of one complex
variable; see \cite{MS15}.



\makeatletter
\newcounter{mybibitem}
\def\bibitem#1{\stepcounter{mybibitem}\@lbibitem[A\the\value{mybibitem}]{#1}}
\makeatother

\renewcommand{\refname}{References for appendix}

\end{document}